%% file: main.tex
\documentclass[english]{amsart}
\usepackage{preamble}
\begin{document}

\begin{abstract}
    We prove that for finitely generated abelian groups $A$ and $B$, the space of $\E_\infty$-ring maps between the spherical groups rings $\Sph[A] \to \Sph[B]$ is equivalent to the discrete set of group homomorphisms $A \to B$. We also prove generalizations where the sphere is replaced by other ring spectra, e.g. we give a formula for the strict units in group rings of the form $R[A]$ for $A$ a finite $p$-group and $R$ $p$-completely chromatically complete.
\end{abstract}

\title{Maps between Spherical Group Rings}
\author{Shachar Carmeli, Thomas Nikolaus, and Allen Yuan}
\date{\today}
\maketitle

\tableofcontents{} 

\section{Introduction}
\input{intro.tex}

\section{Rank $1$ units in $\delta$-rings}\label{sec:delta}
\input{delta.tex}

\section{Tautological units, rigidity, and staticity}\label{sec:groups}
\input{groups.tex}

\section{Rigidity for finite $p$-groups}\label{sec:rig}
\input{rigidity.tex}

\section{Rigidity for free abelian groups}\label{sec:free}
\input{GmGm.tex}

\section{Strict units and maps of spherical group algebras}\label{sec:end}
\input{everything.tex}

\bibliographystyle{alpha}
%\phantomsection\addcontentsline{toc}{section}{\refname}
\bibliography{grpalg}

\end{document}

%% file: intro.tex
This paper deals with a question about commutative ring spectra, also known as $\mathbb{E}_\infty$-rings.  In recent years, many concepts from the ordinary algebra of rings have been generalized to the higher algebra of ring spectra, leading to significant breakthroughs in many classical areas, such as algebraic $K$-theory \cite{BHM, HM} and $p$-adic Hodge theory \cite{BMS}. However, some very basic questions in the area remain open --- most notably, understanding maps between ring spectra is generally very hard.  For an abelian group $A$, we consider the commutative ring spectrum 
\[
\Sph[A] := \Sigma^\infty_+ A ,
\] 
called the \emph{spherical group ring of $A$}, which is
a lift of the usual group ring $\Z[A]$ from the integers $\Z$ to the sphere spectrum $\Sph$. The main question that we address in this paper is the question of which maps exist between these ring spectra.
Our answer can be summarized as follows. 

\begin{thm}\label{theorem_one}
The spherical group ring functor $A \mapsto \Sph[A]$
from the category of finitely generated abelian groups to the $\infty$-category of commutative ring spectra  
is fully faithful. 
\end{thm}

In other words: for finitely generated abelian groups $A$ and $B$, the space of $\E_\infty$-ring maps $\Sph[A] \to \Sph[B]$ is discrete and in bijection with the set of group homomorphisms $A \to B$. 

This result is very surprising, at least to the authors. Indeed, the analogous result in ordinary algebra is completely false: there are many maps $\Z[A]\to \Z[B]$ which do not arise from group homomorphisms $A\to B$.  For instance, when $A= \Z$, such maps are simply units in $\Z[B]$, and the question of units in group rings has been extensively studied \cite{higman1940units, karpilovsky1983finite, jespers1996units, Gardam}.  In addition to the unit $-1$, which does not come from an element of $B$, we have for instance the \emph{Bass cyclic units}, see, e.g., \cite[\S 2]{jespers1996units}.

\Cref{theorem_one} asserts that none of these units lift to $\Sph$, implying that much of the number theory underlying these units, such as the theory of cyclotomic fields and cyclotomic units, does not extend to the sphere. The theorem can also be seen as a strong rigidity result about group rings over $\Sph$. In fact, it says more --- the discreteness of the mapping spaces asserts that the maps that do lift, lift in a unique way. More precisely, the fibers of the map of spaces 
\begin{equation}\label{mapone}
\Map_{\CAlg(\Sp)}(\Sph[A], \Sph[B]) \to \Map_{\mathrm{Ring}}(\Z[A], \Z[B])
\end{equation}
given by base-change along $\Sph \to \Z$ (or equivalently by taking $\pi_0$) are empty or contractible.

\subsection{Background and motivation}

%There is a slick way to see that the ring homomorphisms in (2) and (4) do not lie in the image of (\ref{mapone}): we note that the ring $\Z[A]$ carries more structure than just being a commutative ring. For every prime $p$, the ring $\Z[A]$ has a Frobenius lift given by the ring endomorphism $a \mapsto a^p$. This makes $\Z[A]$ a $\delta_p$-ring for every prime $p$. Using elementary algebra, one can show that the ring homomorphisms $\Z[A] \to \Z[B]$ that are compatible with these $\delta_p$-ring structures for every prime $p$ are precisely those morphisms induced by group homomorphisms $A \to B$. 

%Moreover, using the Frobenius homomorphism for commutative ring spectra studied in \cite{NS}, one can show that a lift of an ordinary ring $R$ to an $\E_\infty$-ring spectrum $\Sph_R$\footnote{i.e. $\Sph_R$ such that $\Sph_R \otimes_\Sph \Z \simeq R$, and in particular, a Moore spectrum. Note that the underlying spectrum of $\Sph_R$ is determined by $R$, so this is really  about the $\E_\infty$-structure.} naturally endows $R$ with the structure of what we call a \emph{$\widehat\delta$-ring}: a commutative ring $R$ together with a $\delta_p$-ring structure on each of its $p$-completions.\footnote{We warn the reader that this is non-standard terminology and there are some subtleties about derived completion hidden.} 

There is a slick way to see that the image of (\ref{mapone}) consists only of maps induced by group homomorphisms $A\to B$.  The idea is that such maps must be compatible with certain natural operations acting on the homotopy groups of commutative ring spectra.

%compatible with what we call a {$\widehat\delta$-ring structure.}

Namely, suppose $R$ is an ordinary ring and we have a lift $\Sph_R$ to a commutative ring spectrum such that $\Sph_R \otimes_\Sph \Z \simeq R$ (i.e., a Moore spectrum)\footnote{Note that the underlying spectrum of $\Sph_R$ is determined by $R$, so this is really about the $\E_\infty$-structure.}.  Then, the Tate-valued Frobenius map of \cite{NS} naturally endows $R^{\wedge}_p$ with a ($p$-)Frobenius lift for every prime $p$.  In fact, $R$ acquires a natural $\delta_p$-ring structure on each of its $p$-completions -- together, we refer to this as a $\widehat\delta$-ring structure.\footnote{We warn the reader that this is non-standard terminology and there are some subtleties about derived completion hidden.} 

In the special case of group rings $\Z[A]$, the $\widehat\delta$-ring structure arising from the spherical lift $\Sph[A]$ is simple: the corresponding Frobenius lifts are induced from the group homomorphisms $A\oto{p} A$ for every prime $p$.  While there can be many interesting ring homomorphisms $\Z[A]\to\Z[B]$, one can show by elementary algebra that the only ones compatible with this $\widehat\delta$-ring structure are the ones induced by group homomorphisms $A\to B$.  

Summarizing this discussion, the map (\ref{mapone}) naturally lands in morphisms of $\widehat \delta$-rings and our main result is equivalent to the assertion that the resulting map
\begin{equation*}
\Map_{\CAlg(\Sp)}(\Sph[A], \Sph[B]) \to \Map_{\hat\delta\text{-}\mathrm{Ring}}(\Z[A], \Z[B])
\end{equation*}
is an equivalence.   We view this result as answering an important special case of the following general question: %say a little more about delta-hat rings, and after the question say that we show another example of this phenomenon

\begin{question}\label{question}
For which $\widehat\delta$-rings $R$ does there exist an $\E_\infty$-lift to the sphere $\Sph_R$, and for which pairs of lifts $\Sph_R, \Sph_{R'}$ is the induced map 
\begin{equation}\label{map_delta}
\Map_{\CAlg(\Sp)}(\Sph_R, \Sph_{R'}) \to \Map_{\hat\delta\text{-}\mathrm{Ring}}(R, R')
\end{equation}
an equivalence?
\end{question}

\begin{exm}
    If $R$ is a $\mathbb{Q}$-algebra, then each $p$-completion of $R$ vanishes and thus a $\widehat\delta$-ring structure is no additional data.  Moreover, we have $\Sph_R = R$, so (\ref{map_delta}) is tautologically an equivalence.
\end{exm}

On the other hand, a $p$-complete $\widehat\delta$-ring is just a $\delta_p$-ring, as all other completions vanish.  In this case, \Cref{question} is understood to have a positive answer in formally \'{e}tale situations, such as the spherical Witt vectors $\Sph_{W(\kappa)}$ of $R=W(\kappa)$ by \cite[Example 5.2.7]{Lurie_Ell2}.  A variant of our main theorem extends this to the case of $p$-complete group rings:

\begin{thm}[\Cref{cor:maps_group_alg_p_complete}]\label{thm:intro_p_main}
Suppose that $S, S'$ are $\E_{\infty}$-rings of the form $\Sph_{W(\kappa)}[M]^{\wedge}_p$ for some perfect ring $\kappa$ of characteristic $p$ and some finitely generated abelian group $M$.  Then the natural map
\[
\pi_0: \Map_{\CAlg(\Sp)}(S, S') \to \Map_{\dpring}(\pi_0(S), \pi_0(S'))
\]
is an equivalence (cf. \Cref{cnstr:groupalgdelta} for the $\delta_p$-structures in the target).  
% Let $\kappa$ be a perfect ring of characteristic $p$ and let $\Sph_{W(\kappa)}$ denote the ring of spherical Witt vectors of $\kappa$.  Then for any finitely generated abelian groups $M$ and $N$, the natural map
% \[
% \Map_{\calg(\Sp)}(\Sph_{W(\kappa)}[N]^\wedge_p,\Sph_{W(\kappa)}[M]^\wedge_p) \iso \Map_{\delta_p-\mathrm{Ring}}(W(\kappa)[N]^\wedge_p,W(\kappa)[M]^\wedge_p)
% \]
% is an equivalence.
\end{thm}

We point out that unlike in the integral case, the equivalent mapping spaces in \Cref{thm:intro_p_main} are \emph{not} in general equivalent to the corresponding set of group homomorphisms, unless the groups involved have no prime to $p$ torsion (\Cref{thm:deltarig}).

We believe the above results to be the first systematic computations of mapping spaces between such flat $\Sph$-algebras that are not (formally) \'etale extensions of the sphere spectrum or completions thereof.  Yet, some fundamental cases of \Cref{question} remain open: for instance, one can study spherical \emph{monoid} rings, rather than group rings.   In fact, to the knowledge of the authors, even the case of polynomial rings is open in all nontrivial cases:

\begin{question}
    What is the space of maps $\Map_{\CAlg(\Sp)}(\Sph[\mathbb{N}^i],\Sph[\mathbb{N}^j])$ for any $i,j\geq 1$?
\end{question}

\begin{exm}
    The map (\ref{map_delta}) is not an equivalence in general.  For example, if $\Sph_R = \Sph \oplus \Sph$ is a square zero extension of $\Sph$ by itself, then the space of $\EE_\infty$-maps $\Sph_R \to \Sph_R$ contains $\Omega^\infty\Sph$ is a retract and hence it is not discrete. 
\end{exm}

As for the first part of \Cref{question}, the lifting question is equivalent to the question of highly structured multiplications on Moore spectra and has been extensively studied \cite{MR760188,burklund2022multiplicative,bhattacharya2022higher}, though we do not know many interesting examples of spherical lifts in the $\E_{\infty}$ case beyond the spherical Witt vectors of \cite[Example 5.2.7]{Lurie_Ell2}.  For instance, we believe the following is open:

\begin{question}
    Does the free $\delta_p$-ring on one generator lift to $\Sph$? 
\end{question}

%Since a $p$-complete $\widehat\delta$-ring is just a $\delta_p$-ring, this shows that the map (\ref{map_delta}) is an equivalence in this case as well. %, i.e. when $\Sph_R$ and $\Sph_{R'}$ are of the form $\Sph_{W(\kappa)}[A]^{\wedge}_p$.  
%f $\Sph_R$ and $\Sph_R'$ are $p$-completions of group rings of finitely generated abelian groups over the spherical Witt vectors $\Sph_{W(k)}$ for $k$ a perfect ring of characteristic $p$.  
% We point out that unlike in the integral case, 

Our interest in studying \Cref{question} arose in work of the second author joint with B. Antieau and A. Krause on prismatic cohomology: for a given $\delta_p$-ring $R$ and any $R$-algebra $S$, there is prismatic cohomology $\Prism_{S/R} \in \mathcal{D}(\Z)$, an important invariant in $p$-adic geometry generalizing crystalline cohomology to mixed characteristic \cite{BMS,bhatt2022prisms,AKN}. The development of this theory was inspired by topological periodic homology and  one can generally ask if $\Prism_{S/R}$ admits a topological refinement, i.e. if it is the associated graded of a filtration on topological periodic homology of a cyclotomic spectrum\footnote{We are suppressing Breuil-Kisin twists, Frobenius twists and Nygaard completions for simplicity.}. It will be shown in forthcoming work of the second author with  Antieau and Krause that this is indeed the case if $R$ admits a lift $\Sph_R$ to the sphere  with the cyclotomic spectrum in question given by $\mathrm{THH}(S / \Sph_R)$. The cyclotomic spectrum $\mathrm{THH}(S / \Sph_R)$
is clearly functorial in the lift $\Sph_R$. This relates \Cref{question} to a question of P. Scholze about which prisms can be realized by cyclotomic spectra and how functorially this can be done \cite{386521}.

 %In another direction, one could wonder what happens in the case that the finite generation hypothesis in \Cref{theorem_one} is relaxed.
%\atodo{There are some more questions related to this rigidity for spectral abelian varieties/group schemes.}

\subsection{Outline and further results}
The core content of the paper is \Cref{thm:intro_p_main} in the case $S=\Sph[\Z]^\wedge_p$ and $S'=\Sph[M]^\wedge_p$, i.e. proving the equivalence
\[
\Map_{\calg(\Sp)}(\Sph[\Z]^\wedge_p,\Sph[M]^\wedge_p) \iso \Map_{\dpring}(\ZZ[\Z]^\wedge_p,\ZZ[M]^\wedge_p).
\]

Our approach is to compute both sides.  The right-hand side is purely algebraic, and we start by completely understanding it in the self-contained \S\ref{sec:delta}.  The mapping space on the left-hand side is known as the \emph{strict units} of $\Sph[M]^{\wedge}_p$, denoted by $\Gm(\Sph[M]^{\wedge}_p)$. 

\begin{rem}
    Strict units of commutative ring spectra are of significant interest in homotopy theory \cite{Fung}, as they are closely related to questions of adjoining roots to ring spectra \cite{lawson2020roots} and multiplicative twists of cohomology theories by ordinary cohomology classes \cite{MaySig, ABG, SatiWesterland}.  Despite this, the difficulty of computing maps between commutative ring spectra has allowed for only a few computations so far --- notably the cases of Lubin--Tate theories (unpublished work of Rezk and Hopkins--Lurie, cf. also \cite{Null}) and the sphere and its completions \cite{CarmeliStrict}.
\end{rem}

Using the augmentation of $\Sph[M]^{\wedge}_p$, we obtain a splitting
%\[
%\Gm(\Sph[N]) = \Map(\Sph[\Z],\Sph[N])%\Omega^{\infty} \map_{\Sp}(\Z, \mathrm{gl}_1(\Sph[N])).
%\]
%Since $\Sph[B]$ is augmented over $\Sph$, we have a splitting 
\[
\Gm(\Sph[M]^{\wedge}_p) \simeq \Gmrp{\Sph}{M} \oplus \Gm(\Sph_p)
\]
into $\Gm(\Sph_p)$, which has been computed in \cite{CarmeliStrict}, and what we call the \emph{reduced units} $\Gmrp{\Sph}{M}$.  We study these via the ``obvious'' map $\theta: M \to \Gmr{\Sph}{M}$ given informally by $m \mapsto [m]$.  A key idea is that $\theta$ is a special case of a map
 \[
\Theta_R : \underline{M}(R) \to \Gmr{R}{M}
 \]
 which is natural in the commutative ring spectrum $R$ and, critically, often an equivalence.\footnote{Here, $\underline{M}(R)$ denotes locally constant functions $\Spec(\pi_0(R)) \to M$.}  This property, which we call \emph{rigidity} (``$R$ is \emph{$M$-rigid}''), is an organizing principle in this paper, as we can analyze its closure under changing $R$ and $M$.  After compiling some basic techniques for proving rigidity in \S\ref{sec:groups}, the remainder of the paper proceeds as follows:

% We organize our work around showing that $\Theta_R$ is often an equivalence; if this is the case, we say $R$ is \emph{$M$-rigid}.  Critically, rigidity is a property and so we can analyze its closure properties as $R$ and $M$ change.  After compiling some basic techniques for proving rigidity in \S\ref{sec:groups}, the remainder of the proof goes as follows:

% \atodo{This part is broken right now, Allen will fix.}
% The core content of the paper is proving \Cref{thm:intro_p_main} in the case $\kappa = \F_p$, i.e. that 
% \[
% \Map_{\calg(\Sp)}(\Sph[M]^\wedge_p,\Sph[N]^\wedge_p) \iso \Map_{\mathrm{Ring}^{\delta_p}}(\ZZ[M]^\wedge_p,\ZZ[N]^\wedge_p).
% \]
% \atodo{reduced to $M=Z$ for both faster}

% To understand the left-hand side, we first use colimits to reduce to the case $M = \Z$, in which this mapping space $\Map(\Sph[\Z]^{\wedge}_p,\Sph[N]^{\wedge}_p) =: \Gm(\Sph[N]^{\wedge}_p)$ is the space of strict units of $\Sph[N]^{\wedge}_p$.  Using the augmentation of $\Sph[N]^{\wedge}_p$, we obtain a splitting
% %\[
% %\Gm(\Sph[N]) = \Map(\Sph[\Z],\Sph[N])%\Omega^{\infty} \map_{\Sp}(\Z, \mathrm{gl}_1(\Sph[N])).
% %\]
% %Since $\Sph[B]$ is augmented over $\Sph$, we have a splitting 
% \[
% \Gm(\Sph[N]^{\wedge}_p) \simeq \Gmrp{\Sph}{N} \oplus \Gm(\Sph_p)
% \]
% into $\Gm(\Sph_p)$, which has been computed in \cite{CarmeliStrict}, and what we call the \emph{reduced units} $\Gmrp{\Sph}{N}$, which we aim to understand.  
 
%  This proceeds in a few steps:%, which correspond to \S \ref{sec:rig}, \ref{sec:free}, and \ref{sec:end}, respectively.  

\textbf{\S\ref{sec:rig}: When $R$ is $p$-complete and $M$ is a finite abelian $p$-group.}  We show:

\begin{thm}\label{thm2}
Assume that $R$ is $p$-complete, $R/p$ is chromatically complete, and $M$ is a finite abelian $p$-group. Then $R$ is $M$-rigid.
% the natural map
% \[
% \Theta_R: \underline{M}(R) \to \Gmr{R}{M}
% \]
% is an equivalence.  
%spectrum of reduced strict units
%$\Gmr R M$ is equivalent to the abelian group of locally constant $M$-valued functions on the topological space $\Spec(\pi_0R)$. 
\end{thm}

In fact, we prove a more general form of this theorem as \Cref{cor:p-rigid}.  We start with the case where $R$ is instead a $T(n)$-local ring spectrum that admits a primitive higher root of unity in the sense of  \cite{CSYCycl} (e.g., any algebra over a height $n$ Lubin--Tate theory).  Here, we can verify \Cref{thm2} directly by computation using the chromatic Fourier equivalence \cite{BCSY, HL}
\[
R[M] \simeq R^{B^nM^*}.
\]
 Then, closure under limits allows us to bootstrap to the general setting.  

\textbf{\S\ref{sec:free}: Extending to torsion free $M = \Z^m$.}  Here, we specialize to $R = \Sph_{W(\kappa)}$, the spherical Witt vectors of a perfect ring $\kappa$ (though intermediate results are proven in greater generality).  We show that the short exact sequence $M\xrightarrow{p} M \to M/p$ induces a cofiber sequence of spectra
\begin{equation}\label{eqn:seq}
\AG{\Sph_{W(\kappa)}[M]^\wedge_p} \to \AG{\Sph_{W(\kappa)}[M]^\wedge_p} \to \AG{\Sph_{W(\kappa)}[M/p]^\wedge_p}.
\end{equation}
It turns out that the first map can be identified with the map induced by the Tate-valued Frobenius, which for general reasons induces multiplication by $p$ on strict units.  This allows us to bootstrap our knowledge of $\AG{\Sph_{W(\kappa)}[M/p]^\wedge_p}$ from \Cref{thm2} to obtain:

\begin{thm}
    Let $\kappa$ be a perfect $\F_p$-algebra and  $M$ be a finitely generated abelian group.  Then $\Sph_{W(\kappa)}$ is $M$-rigid.  
\end{thm}

Combined with \Cref{thm2}, we obtain an analogous result when $M$ is additionally allowed to have $p$-power torsion part (\Cref{cor:freeprig}).

% work with $R$ $p$-complete of finite Tor-amplitude (e.g. $R$ a Moore spectrum, cf. \Cref{defn:finiteTor}).  We show that the short exact sequence $M\xrightarrow{p} M \to M/p$ induces a cofiber sequence of spectra
% \begin{equation}\label{eqn:seq}
% \AG{R[M]^\wedge_p} \to \AG{R[M]^\wedge_p} \to \AG{R[M/p]^\wedge_p}.
% \end{equation}
% When $R$ is perfect, i.e. if the Tate-valued Frobenius $R\to R^{tC_p}$ is an equivalence, the first map in (\ref{eqn:seq}) can be identified with the map induced by the Tate-valued Frobenius on $R[M]^{\wedge}_p$.  The Frobenius induces multiplication by $p$ on strict units, so (\ref{eqn:seq}) gives a precise relationship between  $\AG{R[M]^\wedge_p}$ and $\AG{R[M/p]^\wedge_p}$, which we understand by \Cref{thm2}.  We obtain:

%induced from the short exact sequence $M \to M \to M/p$ of abelian groups\todo{shachar:I think this is too cryptic, we should think how to better summarize this}. If $R$ is perfect, i.e. if the Tate valued Frobenius $R \to R^{tC_p} \simeq R$ is an equivalence, then we show that the map $\AG{R[M]^\wedge_p} \to \AG{R[M]^\wedge_p}$ is equivalent in the arrow category to multiplication by $p$ and deduce the following result:

% \begin{thm}
% For $R$ $p$-complete, perfect, and of finite Tor-amplitude and $\Lambda$ a finitely generated free abelian group, we have an equivalence
% \[
% \AG{R[M]^\wedge_p} \simeq \underline{M}(\pi_0(R)).
% \]
% \end{thm}

%We apply this in the case $R = \Sph^{\wedge}_p$ to show that $\Sph^{\wedge}_p$ is $M$-rigid when $M$ is a finitely generated abelian group without prime to $p$ torsion.  

\textbf{\S\ref{sec:end}: Assembling the pieces.}  We first add in the prime to $p$ torsion in $M$, in particular proving \Cref{thm:intro_p_main}.  Then, we use an arithmetic fracture square argument to combine the $p$-complete results of the previous sections to prove our main theorem about $R=\Sph$.

% \textbf{Step 3 (\S\ref{sec:delta}): The analogous question for $\delta$-rings.}  To organize the answer and return from computing $\Gmred$ to \Cref{question}, we study the algebraic analogue of the previous two sections -- that is, the analogous question in the world of $\delta_p$-rings.  Our main result is a computation of the $\delta_p$-units of $\delta_p$-rings of the form $A[M]$, where $A$ is a $p$-complete $\delta_p$-ring and $M$ is a finitely generated abelian group with no prime to $p$ torsion.

% \textbf{Step 4 (\S\ref{sec:end}): The case $R=\Sph$.}   We finish the proof of our main theorem, using an arithmetic fracture square argument to combine the $p$-complete results of the previous sections.  

% Finally combining these results with an arithmetic fracture square and an identification of $\delta$-ring maps finishes the proof of the main result. 

% in the area called Higher Algebra
% (aka Waldhausen's Brave New Algebra). The main idea in that area is to replace concepts from ordinary algebra, such as commutative rings, by the higher categorical analogues, such as commutative ring spectra (aka $\mathbb{E}_\infty$-rings). 

% The category of commutative ring spectra contains the category of ordinary commutative rings, but also non-classical objects such as the sphere spectrum $\Sph$, which is the higher analogue of the ring of integers $\mathbb{Z}$. 
%To do so, we study the functor 
%\[
%R \mapsto \Hom_R(R[A], R[B]).
%\]

%Make sure to focus the outline on rigidity (and maybe staticity)
\subsection{Acknowledgements}
The authors would like to thank Zhouhang Mao, Tomer Schlank, and the entire Copenhagen homotopy theory seminar group for useful discussions related to this material, as well as Achim Krause for sharing his ideas related to \S\ref{sub:moore}. We would also like to thank Edith H\"ubner and Robert Szafarczyk for helpful discussions and suggestions concerning Section 2.
The first author was partially supported by the Danish National Research Foundation through the Copenhagen Centre for Geometry and Topology (DNRF151).
  The second author was funded by the Deutsche Forschungsgemeinschaft
(DFG, German Research Foundation) - Project-ID 427320536 - SFB 1442, as well as
under Germany's Excellence Strategy EXC 2044 390685587, Mathematics M\"{u}nster:
Dynamics-Geometry-Structure. 
The third author was supported in part by NSF grant DMS-2002029.

%% file: delta.tex
In this section, we consider the algebraic analogue of the main results of this paper.  We start by reviewing the relevant structures to set notation.  

\begin{defn}\label{defn:delta}
    For a prime $p$, a $\delta_p$-ring\footnote{These are often simply called $\delta$-rings, but we add the $p$ to the notation to allow for different primes.} is a commutative ring $R$ together with an operation $\delta_p:R \to R$ satisfying the identities:
    \begin{align*}
        \delta_p(0)&=\delta_p(1) = 0,\\
        \delta_p(x+y) &= \delta_p(x)+\delta_p(y) + \frac{1}{p}(x^p+y^p - (x+y)^p), \\
\delta_p(xy) &= x^p\delta_p(y) + y^p \delta_p(x) + p \delta_p(x)\delta_p(y).
    \end{align*}
    We denote the category of these by $\dRing$.  
\end{defn}

The identities are rigged so that the operation $\psi(x) = x^p + p\delta_p(x)$ is a Frobenius lift.   In fact, a $\delta_p$-structure on a torsion-free ring is simply a Frobenius lift.  This becomes even cleaner in the animated setting.

\begin{defn}
Let $\Acr$ denote the category of animated commutative rings (cf. e.g. \cite[\S 1.2]{cesnavicius2019purity}).  Then an animated $\delta_p$-ring is an animated commutative ring $R$ together with a lift 
\[
\begin{tikzcd}
    R \arrow[r,dashed]\arrow[d] & R\arrow[d]\\
    R\mm p \arrow[r,"\varphi_{R\mm p}"] & R\mm p.
\end{tikzcd}
\]
Here, $R\mm p$ denotes the derived reduction mod $p$, and $\varphi_{R\mm p}$ denotes the animated Frobenius.  We let $\Acr^{\delta_p}$ denote the category of animated $\delta_p$-rings.  
\end{defn}

This definition is compatible with \Cref{defn:delta} in the sense that the objects in $\Acr^{\delta_p}$ with discrete underlying ring are exactly $\delta_p$-rings, and $\Acr^{\delta_p}$ is the animation of $\dRing$ \cite[Thm. 2.4.4]{Holeman}.

%In this section, we consider the algebraic analogue of the main results of the paper: instead of commutative ring spectra, we consider $\delta_p$-rings.  In lieu of discussing the basics of $\delta_p$-rings, we refer the reader to \cite{some stuff}.  We will also work in the animated context at several points, cf. \S\ref{subsub:AS} for the relevant notations.\todo{insert refs, perhaps expand}

We will be particularly interested in the following  $\delta_p$-rings:

\begin{cnstr}\label{cnstr:groupalgdelta}
Let $M$ be an abelian group. Since the group ring $\ZZ[M]$ is torsion free, a $\delta_p$-ring structure on it is the same as a lift of Frobenius.  Accordingly, we equip the group rings $\ZZ[M]$ with the $\delta_p$-structures corresponding to the lift of Frobenius 
\[
\varphi(\sum_{m\in M} a_m\cdot [m]) = \sum_{m\in M} a_m \cdot [pm]. 
\]  
More generally, for a $\delta_p$-ring $R$, we endow $R[M] = R\otimes \ZZ[M]$ with the tensor product $\delta_p$-ring structure, i.e. $\varphi_{R[M]} = \varphi_R \otimes \varphi_{\ZZ[M]}$. % This also equips $\ZZ[M]$ with a $\widehat{\delta}$-ring structure in an evident way.  
\end{cnstr}
%In the special case $M=\ZZ$ we shall denote 
%$
%\ZZ[\ZZ]=:\ZZ[t^{\pm 1}].
%$ 
%\todo{This should probably be introduced earlier, as an example when E-infty rings are related to delta-p-rings.}

The analogue of strict units in this setting are the following:

\begin{defn}\label{defn:delta-unit}
    A \tdef{rank $1$ unit} in a $\delta_p$-ring $A$ is a morphism of $\delta_p$-rings
\[
\ZZ[\ZZ] =: \ZZ[t^{\pm 1}] \to A.
\]
We denote by $\mdef{\Gmdp{A}}$ the set of rank $1$ units, which can be described as the units $t\in \GG_m(A)$ such that  $\delta_p(t)=0$.  In fact, the identities for $\delta_p$ immediately imply that $\Gmdp{A}$ is a \emph{subgroup} of $\Gm(A)$.  
\end{defn}

\begin{rem}\label{rem:delta_p_right_adjoint}
The functor $A\mapsto \Gmdp{A}$ from $\delta_p$-rings to abelian groups is right adjoint to the group algebra functor of \Cref{cnstr:groupalgdelta}.
\end{rem}

\begin{exm}[Witt vectors, cf. e.g. \cite{BhattNotes}]\label{exm:wittunit}
    For any perfect ring $\kappa$, the ring of $p$-typical Witt vectors $W(\kappa)$ admits a unique $\delta_p$-structure, with Frobenius lift induced by the Frobenius map $\kappa \to \kappa$. In fact, these are precisely the perfect $\delta_p$-rings, i.e. those with invertible lift of Frobenius.  In this case, 
the multiplicative lift and the projection $W(\kappa) \to \kappa$ induce isomorphisms
\[
\kappa^{\times} \iso \Gmdp{W(\kappa)} \iso \kappa^\times.
\]  
\end{exm}

Our first goal in this section is to compute the rank $1$ units of group algebras.  The techniques are quite different from those in the $\Sph$-linear situation of the later sections -- they rely on the deformation theory of the functor $\Gmdp{-}$, which we study in \S\ref{sub:deltadef}.  Applying this in \S\ref{sub:deltarig}, we show:

\begin{thm}\label{thm:delta_p_main}
    Let $A$ be a $p$-complete $\delta_p$-ring and $M$ be a finitely generated abelian group whose torsion part is $p$-power torsion.  Then there is a natural isomorphism
    \[
    \Gmdp{A[M]^{\wedge}_p} \cong \Gmdp{A} \oplus \underline{M}(A).
    \]
    Here, $\underline{M}(A)$ denotes the abelian group of locally constant $M$-valued functions on $\Spec(A)$.  
\end{thm}

We finish with some material connecting this story to the remainder of the paper.  
In \S\ref{sub:deltahat}, we globalize these results (across primes $p$) using the notion of a $\widehat \delta$-ring, which is essentially a ring with a $\delta_p$-structure on its $p$-completion for each $p$. 
Finally, in \S\ref{sub:moore}, we relate this discussion back to the topological story by showing that if $R$ is a commutative ring spectrum such that $R\otimes \Z$ is discrete (i.e., a \emph{Moore spectrum}), then $\pi_0(R)$ acquires a natural $\widehat \delta$-ring structure. %on each $p$-completion -- a structure we call a $\widehat \delta$-ring.  

\subsection{Deformation theory of rank 1 units}\label{sub:deltadef}

We study rank $1$ units of group algebras by deformation theory, starting in \S \ref{subsub:ideals} with the case of a square-zero extension of $\delta_p$-rings.  Here, by explicit analysis, we obtain an exact sequence  
\begin{equation}\label{eqn:fundexact}
0\to I^{\delta_p=1} \to \Gmdp{A} \to \Gmdp{A/I}.
\end{equation}
The key ingredient to going further is the observation that units of rank $1$ satisfy a certain ``square nil-invariance.''  That is, under appropriate completeness conditions on an ideal $I\subset A$, we prove in \S \ref{subsub:nil} an isomorphism $\Gmdp{A}\cong \Gmdp{A/I^2}$.   This gives us (\ref{eqn:fundexact}) more generally.  

The final ingredient we need is to identify the fiber term in the case where $A$ is a group algebra.  It turns out that this can be done with a form of Artin--Schreier theory for $\delta_p$-rings, which we discuss in \S \ref{subsub:AS}.

\subsubsection{$\delta_p$-ideals and square-zero extensions}\label{subsub:ideals}
\begin{defn}
We say an ideal $I$ in a $\delta_p$-ring $A$ is a \mdef{$\delta_p$-ideal} if $\delta_p(I)\subseteq I$.  %For such an ideal, let $\cl{I}$ denote its closure in the $p$-adic topology on $A$.  
\end{defn}

%For such an ideal, we will often be interested in the closure of $I$ under the $p$-adic metric on $A$, which we denote $\cl{I}$.

$\delta_p$-ideals  are particularly easy to analyze in the case where $I$ is \emph{square-zero}, that is, $I^2=0$.  

\begin{prop}\label{prop:ugly-delta_identity}
Let $A$ be a $\delta_p$-ring and let $I\subset A$ be a square zero $\delta_p$-ideal.  Then:
\begin{enumerate}
\item The map $\delta_p\colon I\to I$ is a $\varphi$-semilinear morphism of $A$-modules. Namely, it is a group homomorphism and satisfies 
\[
\delta_p(xa) = \varphi(x)\delta_p(a), \quad \forall x\in A, a\in I.
\]
\item For $x\in A$ and $a\in I$ we have 
\[
\delta_p(x+a) = \delta_p(x) + \delta_p(a) - x^{p-1}a. 
\]
\end{enumerate}

\end{prop}

\begin{proof}
Both statements follow immediately from the additive and multiplicative identities of $\delta_p$, using the fact that $I$ is a square zero ideal. Specifically: 
\[
\delta_p(a+b) = \delta_p(a)+\delta_p(b) + \frac{a^p+b^p - (a+b)^p}{p} = \delta_p(a)+\delta_p(b) \quad \forall a,b\in I,
\]
\[
\delta_p(xa) = x^p\delta_p(a) + a^p \delta_p(x) + p \delta_p(x)\delta_p(a) = (x^p + p \delta_p(x))\delta_p(a) = \varphi(x)\delta_p(a) \quad \forall x\in A, a\in I,  
\]
and 
\[
\delta_p(x+a) = \delta_p(x) + \delta_p(a) + \frac{x^p + a^p - (x+a)^p}{p} = \delta_p(x) + \delta_p(a) - x^{p-1}a \quad \forall x\in A, a\in I.  
\]
\end{proof}

This gives us the following statement about how $\Gmdp{-}$ changes in square-zero extensions.

\begin{prop}\label{prop:sq0exactseq}
Let $A$ be a $\delta_p$-ring and let $I$ be a square-zero $\delta_p$-ideal in $A$. Then, we have an exact sequence 
\[
0 \too I^{\delta_p = 1} \too  \Gmdp{A} \too \Gmdp{A/I}  
\]
\end{prop}
\begin{proof}
Let $1+x \in \Gmdp{A}$ be an element of $\ker(\Gmdp{A}\to\Gmdp{A/I})$, where $x\in I$.  Then we have (using \Cref{prop:ugly-delta_identity}(2))
\[
0 = \delta_p(1+x) = \delta_p(1) + \delta_p(x) - x.
\]
Noting that $\delta_p(1)=0$, we obtain the conclusion.  
\end{proof}

\begin{rem}
As we analyze the behavior of $\Gmdp{-}$ along deformations, we focus on the case when the relevant ideals are closed in the $p$-adic topology.  However, even if $I\subseteq A$ is closed, its powers $I^n$ may not be, so we let $\cl{I^n}$ denote their closures in the $p$-adic topology.  (In any case, for our applications, this distinction will not be important as our ideals will be finitely generated.)  
\end{rem}

The following lemma ensures that these operations play well with the $\delta_p$-structure:

\begin{lem}\label{lem:deltapideals}
Let $I\subset A$ be a $\delta_p$-ideal.  Then:
\begin{enumerate}
    \item For $n\geq 0$, $I^n$ is a $\delta_p$-ideal.
    \item The ideal $\cl{I}\subset A$ is a $\delta_p$-ideal.
\end{enumerate}
\end{lem}
\begin{proof}
    The first statement is immediate from the addition and multiplication formulas for $\delta_p$, and the latter follows from the fact that $\delta_p$ is $p$-adically continuous.  
\end{proof}

\subsubsection{Square nil-invariance}\label{subsub:nil}
While the functor $\Gmdp{-}$ is not nil-invariant in general, we will show here that it is invariant under quotient by \emph{the square} of an ideal with respect to which a $\delta_p$-ring is complete.  The key input is the following lemma:

\begin{lem}\label{lem:deltatangent}
Let $A$ be a $\delta_p$-ring and $I\subset A$ be a $\delta_p$-ideal.  Then, for $m\geq 2$, the images of the induced maps
\begin{align*}
\delta_p: I^m/I^{m+1} &\to I^m/I^{m+1}\\
\delta_p: \cl{I^m}/\cl{I^{m+1}} &\to \cl{I^m}/\cl{I^{m+1}}
\end{align*}
are contained in $p\cdot I^m/I^{m+1}$ and $p\cdot \cl{I^m}/\cl{I^{m+1}}$ respectively.  
\end{lem}
\begin{proof}
    The second statement is immediate from the first because $\delta_p$ is $p$-adically continuous.  
    
    To see the first statement, we note that by the additivity statement of \Cref{prop:ugly-delta_identity}, it suffices to consider elements of the form $ab \in I^m$ for $a\in I$ and $b\in I^{m-1}$.  Then we have 
\begin{align*}
\delta_p(ab) &= a^p\delta_p(b) + b^p\delta_p(a) + p\delta_p(a)\delta_p(b) \\
&\equiv p\delta_p(a)\delta_p(b)\pmod{I^{m+1}}.
\end{align*}
\end{proof}

Given this, we have our main statement:

\begin{prop}\label{prop:sqnil}
Let $A$ be a $\delta_p$-ring and let $I$ be a $\delta_p$-ideal such that $A$ is classically $(I,p)$-complete. Then, the map 
\[
A\to A/\cl{I^2}
\]
induces an isomorphism 
\[
\Gmdp{A} \iso \Gmdp{A/\cl{I^2}}.
\]
\end{prop}

\begin{proof}
Our assumption that $A$ is $(I,p)$-complete implies that 
\[
A\simeq \invlim A/\cl{I^m}. 
\]
To see this, note that the map from $A$ to the limit is clearly surjective by assumption, and so it is enough to check that $\cap \cl{I^m} = \{ 0\}$.  But if $x \in \cap \cl{I^m}$, then for all $j\in \NN$,  $x \in \cap I^m$ modulo $p^j$, and hence $x\equiv 0 \mod p^j$ since $A/p^j$ is $I$-complete. Since $A$ is $p$-complete this implies that $x=0$. 

It therefore suffices to show that the induced map
\[
\Gmdp{A/\cl{I^{m+1}}} \to \Gmdp{A/\cl{I^m}}
\]
is an isomorphism when $m\geq 2$.  For this, let $x\in \Gmdp{A/\cl{I^m}}$ and let $\tilde{x}$ be a lift of $x$ to $A/\cl{I^{m+1}}$.  We will show there is a unique $a\in \cl{I^m}/\cl{I^{m+1}}$ such that $\tilde{x}+a \in \Gmdp{A/\cl{I^{m+1}}}$.  As $A$ is $I$-complete, $\tilde{x}+a$ is certainly a unit, so $\tilde{x}+a\in \Gmdp{A/\cl{I^{m+1}}}$ if and only if $\delta_p(\tilde{x}+a) = 0$.  Since the ideal $\cl{I^m} / \cl{I^{m+1}}\subset A/\cl{I^{m+1}}$ is square-zero, we have by \Cref{prop:ugly-delta_identity}(2)
\[
\delta_p(\tilde{x}+a) = \delta_p(\tilde{x}) + \delta_p(a)-\tilde{x}^{p-1}a
\]
Since $x \in \Gmdp{A/\cl{I^m}}$, the first term is in $\cl{I^m}/\cl{I^{m+1}}$.  It therefore suffices to show that the function
\begin{align*}
\cl{I^m}/\cl{I^{m+1}} &\to \cl{I^m}/\cl{I^{m+1}}\\
a &\mapsto \delta_p(a) - \tilde{x}^{p-1}a
\end{align*}
is a bijection.  But $\tilde{x}$ is a unit, so it's enough to show that $\id - \frac{1}{\tilde{x}^{p-1}}\delta_p$ is an isomorphism.  But this has inverse given by 
\[
b\mapsto \sum_{j=0}^{\infty} \left(\frac{1}{\tilde{x}^{p-1}}\delta_p \right)^{(j)}(b),
\]
which is a well-defined element of $\cl{I^{m}}/\cl{I^{m+1}}$ because the $j$-th term of this series is divisible by $p^j$ by \Cref{lem:deltatangent}.

\end{proof}

Combining this with \Cref{prop:sqnil}, we have the following extension of \Cref{prop:sq0exactseq}:

\begin{prop}\label{prop:nil-inv_delta_units}
Let $A$ be a $\delta_p$-ring and let $I$ be a $\delta_p$-ideal in $A$ such that $A$ is classically $(I,p)$-complete.  Then $A\to A/\cl{I^2}$ induces an isomorphism on $\Gmdp{-}$ and we obtain an exact sequence
% then the map 
% \[
% A \to A/I^2 
% \]
% induces an isomorphism
% \[
% \Gmdp{A} \iso \Gmdp{A/I^2},
% \] 
% and we have an exact sequence 
\[
\xymatrix{
0 \ar[r] & (\cl{I}/\cl{I^2})^{\delta_p = 1} \ar[r] & \Gmdp{A} \ar[r] & \Gmdp{A/\cl{I}} 
}
. 
\]
Here, the left term is the kernel of the homomorphism $\delta_p -\id \colon \cl{I}/\cl{I^2} \to \cl{I}/\cl{I^2}   $. 
\end{prop}

\begin{rem}
In the situation of \Cref{prop:nil-inv_delta_units}, the map $\Gmdp{A} \to \Gmdp{A/\cl{I}}$ is not generally surjective.  However, it will be in our cases of interest, as we take $A \to A/\cl{I}$ to be the augmentation of an augmented algebra (the group ring), which has a section.  
\end{rem}

%\begin{rem} \label{rem:delta_units_explicit}
%Note that $\ZZ[t]^\wedge_p$ is the quotient of the free $\delta_p$-ring on a single generator $t$ by the $\delta$-ideal generated by $t$. It follows that a $\delta_p$-unit in $A$ is simply a unit $x\in \GG_m(A)$ satisfying 
%\[
%\delta_p(x) = 0.
%\] 
%If $A$ is moreover $p$-torsion free, this is equivalent to the condition 
%\[
%\varphi(x) = x^p
%\] 
%on the corresponding lift of Frobenius.
%\end{rem}

\subsubsection{Artin--Schreier theory}\label{subsub:AS}

%\todo{Add something about animation, as this will be the first time it shows up; say animated delta ring is animated ring w lift of frob}
Another general result regarding $\delta_p$-rings that we need is Artin-Schreier theory. 
In this sub-section, for an animated commutative ring $R\in \Acr$, we regard the derived scheme $\Spec(R)$ as endowed with the \'{e}tale topology on its small \'{e}tale site (see, e.g.,  \cite[\S 5.2.3]{cesnavicius2019purity}). We refer to sheaves on $\Spec(R)$ simply as \'{e}tale sheaves, and unless stated otherwise consider spectral coefficients. For an abelian group $A$, we denote by $\ct{A}^{\et}$ the constant \'{e}tale sheaf on $\Spec(R)$. 
\begin{defn}
Let $R\in \Acr_{\ZZ/p^r}$ and let $\ol{R}:= R\otimes _{\ZZ/p^r} \ZZ/p$. A \tdef{lift of Frobenius} for $R$ is a map $\varphi_R\colon R\to R$ which lifts the derived Frobenius map $\varphi_{\ol{R}}\colon \ol{R}\to \ol{R}$ along the reduction map $R\onto \ol{R}$. More precisely, it is the data of a commutative square of the form:
\[
\xymatrix{
R \ar[d] \ar@{..>}^{\varphi}[r]& R \ar[d] \\ 
\ol{R} \ar^{\ol{\varphi}}[r] & \ol{R}
}
\]
\end{defn}
Our main example for lifts of Frobenius is the following.
\begin{exm}\label{ex:delta_frob_lifts}
Let $A$ be a $\delta_p$-ring, and $\varphi\colon A\to A$ the corresponding Frobenius map. Then, by definition, the map $A\mm p^r \to A\mm p^r$ induced from $\varphi$ by reducing mod $p^r$ extends canonically to a lift of Frobenius. Moreover, if $A$ is derived $p$-complete, we may view the $\delta_p$-ring structure on $A$ as a compatible choice of such lifts of Frobenius.  
\end{exm} 

It turns out that a lift of Frobenius for $R$ automatically extends to all \'{e}tale $R$-algebras. 
\begin{lem}\label{Frob_unique}
Let $R$ be an animated commutative $\ZZ/p^r$-algebra endowed with a lift of Frobenius $\varphi_R\colon R\to R$. Then, every \'{e}tale $R$-algebra $S$ admits a unique lift of Frobenius $\varphi_S\colon S\to S$ compatible with $\varphi_R$.  
\end{lem}

\begin{proof}
Let $\ol{S}:= S\otimes_{\ZZ/p^r} \ZZ/p$, so that we have a reduction mod $p$ map $S\to \ol{S}$. 
A compatible lift of Frobenius for $S$ is a solution to the lifting problem 
\[
\xymatrix{
R\ar^{\varphi_R}[r] \ar[d]  &  R\ar[r]&  S \ar[d] \\
S \ar[r]\ar@{..>}[rru]  & \ol{S} \ar^{\varphi_{\ol{S}}
}[r]& \ol{S}
}.
\]
Since $S$ is an animated $\ZZ/p^r$-algebra as well, the right vertical map is a nil-thickening, and the uniqueness of the lift now follows from the lifting property of \'{e}tale morphisms of animated commutative rings (i.e. by the vanishing of the cotangent complex $L^{\mathrm{alg}}_{S/R}$). 
\end{proof}

\begin{cor}
Let $R$ be an animated commutative $\ZZ/p^r$-algebra with a lift of Frobenius $\varphi_R\colon R\to R$. Then, the map $\varphi_R$ extends uniquely to an endomorphism of sheaves of $\ZZ/p^r$-algebras $\varphi \colon \mO_R \to \mO_R$ on the small \'{e}tale site of $\Spec(R)$. Here, $\mO_R$ denotes the structure sheaf of $\Spec(R)$.   
\end{cor}
\begin{proof}
The data of a morphism $\varphi\colon \mO_R \to \mO_R$ lifting $\varphi_R$ is the same as that of a map of sheaves of animated commutative $R$-algebras $\varphi_R^*\mO_R \to \mO_R$. Note that both the source and target consist of sheaves of \'{e}tale algebras over $R$.% More presicely, \'{e}tale sheaves on $\Spec(R)$ fully faithfully embed in presheaves on \'{e}tale $R$-algebras, and the restriction of both $\mO_R$ and $\varphi^*\mO_R$ to such presheaves consisnt of \'{e}tale $R$-algebras. 

Since the $\infty$-category of \'{e}tale $R$-algebras is equivalent to the $1$-category of \'{e}tale $\pi_0R$-algebras (see, e.g., \cite[Proposition 5.2.4]{cesnavicius2019purity}), it follows that a natural transformation $\varphi_R^*\mO_R \to \mO_R$ is determined by its restriction to objects and morphisms in the category of \'{e}tale $R$-algebras. Individual lifts at the objects exist uniquely by \Cref{Frob_unique}, and they are compatible with morphisms immediately by the uniqueness. 
%Since a lift of Frobenius $\varphi\colon \mO_R\to \mO_R$ is the same as a natural transformation $\varphi^*\mO_R \to \mO_R$ of presheaves of \'{e}tale $R$-algebras, the data of such a lift is determined uniquely by its restriction to objects and morphisms of the category of \'{e}tale $R$-algebras. Individual lifts at the objects exist uniquely by \Cref{Frob_unique}, and they are compatible with morphisms immediately by the uniqueness. 
\end{proof}

Since $\varphi_R$ is a morphism of sheaves of $\ZZ/p^r$-algebras, we have a canonical null sequence of sheaves 
\[
(\ct{\ZZ/p^r})^{\et} \too \mO_R \oto{\varphi - 1} \mO_R,  
\]
where $(\ct{\ZZ/p^r})^{\et}$ is the constant \'{e}tale sheaf on the value $\ZZ/p^r$. 

\begin{lem} \label{prop:Frob_nilp_higher_homotopies}
Let $R$ be an animated commutative $\ZZ/p^r$-algebra with a lift of Frobenius $\varphi_R\colon R\to R$. For all $i>0$, the map $\varphi_R$ induces a nilpotent endomorphism of $\pi_iR$. In particular, $\varphi_R-1$ is an isomorphism on $\pi_iR$.
\end{lem}

\begin{proof}
Let $\ol{R}:=R\otimes_{\ZZ/p^r}\ZZ/p$. 
Tensoring the $p$-adic filtration of $\ZZ/p^r$ with $R$, we obtain a natural finite filtration $R_\bullet$ with associated graded 
\[
\gr
(R_\bullet) \simeq R\otimes_{\ZZ/p^r} \gr((\ZZ/p^r)_\bullet) \simeq  \ol{R}[t]/t^r.
\] 
Since the filtration is finite and natural in $R$, it is  stable under $\varphi_R$ and it suffices to show that $\varphi_R$ is nilpotent on $\pi_i\gr(R_\bullet)$; in fact we will show it is zero.  Indeed, we have 
\[
\varphi_R(at^j) = \varphi_{\ol{R}}(a)t^j \quad \forall a\in \pi_i \ol{R},
\] 
which is $0$ because $\varphi_{\ol{R}}$ vanishes in positive degrees (cf. \cite[Prop 11.6, Rem 11.8]{BS}).
\end{proof}

\begin{cor}\label{Frob_fixed_cartesian_pi_zero}
Let $R$ be an animated commutative $\ZZ/p^r$-algebra with a lift of Frobenius $\varphi_R\colon R\to R$. Then, the square 
\[
\xymatrix{
\mO_R\ar^{\varphi-1}[r]  \ar[d]    & \mO_R\ar[d] \\ 
\pi_0\mO_R \ar^{\varphi-1}[r]& \pi_0\mO_R
}
\]
is a Cartesian square of \'{e}tale sheaves on $\Spec(R)$. 
\end{cor}

\begin{proof}
It suffices to prove that the square becomes Cartesian after evaluation at $\Spec(S)$ for an \'{e}tale $R$-algebra $S$, namely that
\[
\xymatrix{
S \ar^{\varphi_S -1}[r]\ar[d] & S \ar[d] \\ 
\pi_0S \ar[r] \ar^{\varphi_S -1}[r] & \pi_0 S 
}
\]
is Cartesian. 
The fiber of both vertical maps is $\tau_{\ge 1}S$, and $\varphi_S-1$ is an isomorphism on this fiber by  \Cref{prop:Frob_nilp_higher_homotopies}. This implies the result. 
\end{proof}
We are ready to state and prove a version of Artin-Schreier theory for Frobenius lifts. 

\begin{prop}[Artin-Schreier for Frobenius Lifts]\label{prop:as-main}
Let $R\in \Acr_{\ZZ/p^r}$ with a lift of Frobenius $\varphi_R\colon R\to R$, and let $\varphi\colon \mO_R\to \mO_R$ be its unique extension from \Cref{Frob_unique}. Then, 
the null sequence 
\[
(\ct{\ZZ/p^r})^{\et} \too \mO_R \oto{\varphi - 1} \mO_R 
\]
is a fiber sequence of sheaves of $\ZZ/p^r$-module spectra on $\Spec(R)$, and it induces a short exact sequence   
\[
0\too 
(\ct{\ZZ/p^r})^{\et} \too \pi_0\mO_R \oto{\varphi - 1} \pi_0\mO_R \too 0 
\]
of sheaves of abelian groups on $\Spec(R)$. 
\end{prop}

\begin{proof}
%Since tensoring with $\ZZ/p$ over $\ZZ/p^r$ is conservative on sheaves of $\ZZ/p^r$-modules, its enough to check that this is a fiber sequence after base-changing to $\ZZ/p$. Let 
%$\ol{R}:= R\otimes_{\ZZ/p^r} \ZZ/p$. Then, the base-change of the sequence 
%\[
%(\ct{\ZZ/p^r})^{\et} \too \mO_R \oto{\varphi - 1} \mO_R,  
%\]
%along $\ZZ/p^r\to \ZZ/p$ identifies with the corresponding sequence 
%\[
%(\ct{\ZZ/p})^{\et} \too \mO_{\ol{R}} \oto{\varphi - 1} \mO_{\ol{R}}
%\]
%associated with the simplicial commutative $\ZZ/p$-algebra $\ol{R}$. Note that in the latter sequence, $\varphi$ is the (derived) Frobenius map, which acts trivially on higher homotopy groups [Ref]. As a result, the map $\varphi-1$ is an isomophism on $\pi_i\mO_{\ol{R}}$ for $i>0$, and hence the map $\mO_R \to \mO_{\pi_0\ol{R}}$ induces an isomorphism on the fibers of $\varphi-1$. 
%We therefore reduced to show that the null sequence 
%\[
%(\ct{\ZZ/p})^{\et} \too \mO_{\pi_0\ol{R}} \oto{\varphi - 1} \mO_{\pi_0\ol{R}}
%\]
%is a fiber sequence of \'{e}tale sheaves. This, in turn, follows from classical Artin-Schreier theory, see [Ref]. 
First, observe that the first sequence being a fiber sequence and the second sequence being exact are equivalent, by \Cref{Frob_fixed_cartesian_pi_zero}. Now, the claim for $r=1$ follows from classical Artin-Schreier theory, which says that for a (classical) $\ZZ/p$-scheme $X$ the sequence 
\[
0 \too (\ct{\ZZ/p})^{\et} \too \mO_X \oto{\varphi_X - 1} \mO_X \too 0  
\]
is a short exact sequence of \'{e}tale sheaves. 
For the general case, let $R\in \Acr_{\ZZ/p^r}$ with lift of Frobenius $\varphi_R$, and set $\ol{R} := R\otimes_{\ZZ/p^r} \ZZ/p$. Then, via the equivalence of \'{e}tale sites $\Spec(R)\simeq \Spec(\ol{R})$, the base-change of the null sequence 
\[
(\ct{\ZZ/p^r})^{\et} \too \mO_R \oto{\varphi - 1} \mO_R 
\]
along the map $\ZZ/p^r\to \ZZ/p$ identifies with the sequence 
\[
(\ct{\ZZ/p})^{\et} \too \mO_{\ol{R}} \oto{\varphi_{\ol{R}}-1} \mO_{\ol{R}}. 
\]
Since the latter is a fiber sequence by the case $r=1$ and the functor $\ZZ/p \otimes_{\ZZ/p^r} (-) $ is exact and conservative, we deduce that the former sequence is a fiber sequence and the result follows. 
\end{proof}

\begin{cor}\label{cor:Artin_Schreier_delta}
Let $A$ be a derived $p$-complete $\delta_p$-ring. For every $r\in \NN$ we have an exact sequence 
\[
0 \too \ct{\ZZ/p^r}(A) \to A/p^r \oto{\varphi_A-1} A/p^r.
\]
\end{cor}

\begin{proof}
Let $R= A\mm p^r$, which is an animated $\ZZ/p^r$-algebra with lift of Frobenius. Applying \Cref{prop:as-main} to it, we get an exact sequence 
\[
0 \too (\ct{\ZZ/p^r})^{\et} \too \pi_0\mO_R \oto{\varphi - 1} \pi_0\mO_R \too 0  
\]
of \'{e}tale sheaves on $\Spec(R)$. Taking the long exact sequence of cohomologies associated with this exact sequence of sheaves, we get an exact sequence 
\[
0 \too H^0_{\et}(\Spec(R);\ZZ/p^r) \to A/p^r \to A/p^r.
\] 
Now, the zeroth \'{e}tale cohomology with constant coefficients is locally constant functions so 
\[
H^0_{\et}(\Spec(R);\ZZ/p^r) \simeq \ct{\ZZ/p^r}(R) = \ct{\ZZ/p^r}(A/p^r). 
\]  
Finally, the assumption that $A$ is derived $p$-complete implies that $(A,(p^r))$ is an Henselian pair (\cite[Lemma 15.93.10]{stacks-project}), so that $A\to A/p^r$ induces a bijection on Zariski-components \cite[Lemma 15.11.6]{stacks-project}, and hence on global sections of the constant sheaf $\ct{\ZZ/p^r}$. This concludes the proof. 
%\atodo{insert ref, see also prop 2.35}
\end{proof}

\subsection{Rigidity for $\delta_p$-group algebras}\label{sub:deltarig}

In this section, we compute the rank $1$ units of a $p$-complete group algebra $A[M]^{\wedge}_p$ in terms of $M$ and the rank $1$ units in $A$.  It will be convenient to phrase our results in terms of a slight variant of $\Gmdp{-}$.

\begin{defn}
For an abelian group $M$, set %\todo{Do we ever use the non-$p$-complete one?}
\begin{align*}
\mdef{\Gmdred{A[M]}} &:= \Ker(\Gmdp{A[M]} \to \Gmdp{A}),\\
\mdef{\Gmdred{A[M]^\wedge_p}} &:= \Ker(\Gmdp{A[M]^\wedge_p} \to \Gmdp{A}).
\end{align*}
Here, the $\delta_p$-structure on the $p$-completion of a $\delta_p$-ring $R$ is the unique one for which $R \to R^\wedge_p$ is a $\delta_p$-ring map, see \cite[Lemma 2.17]{bhatt2022prisms}.  We refer to these as the \emph{reduced units} of rank $1$.
\end{defn}

Since the augmentation $A[M]^{\wedge}_p \to A$ is split, \Cref{thm:delta_p_main} is equivalent to:

\begin{thm}\label{thm:delta_p_main_red}
    Let $A$ be a $p$-complete $\delta_p$-ring and $M$ be a finitely generated abelian group whose torsion part is $p$-power torsion.  Then there is a natural isomorphism
    \[
    \Gmdred{A[M]^{\wedge}_p} \cong \underline{M}(A).
    \]
\end{thm}

We handle the case where $M$ is a finite abelian $p$-group in \S \ref{subsub:rigtors} and the general case in \ref{subsub:riggen}.

\subsubsection{The torsion case}
\label{subsub:rigtors}

Our strategy is to apply the square nil-invariance results to the augmentation ideal in $A[M]$ when $M$ is a finite abelian $p$-group.  
%\Cref{prop:nil-inv_delta_units} to the $\delta_p$-ring $A[M]^{\wedge}_p$ and the augmentation ideal $I\subset A[M]^{\wedge}_p$.   
To do so, we need the following lemma:

\begin{lem}
Let $A$ be a $p$-complete commutative ring, let $M$ be a finite abelian $p$-group, and let $J\subset A[M]$ denote the kernel of the augmentation $A[M]\to A$.  Then $A[M]$ is $(p,J)$-complete.
\end{lem}
\begin{proof}
Note that $J$ is generated by the finite set of elements $[m]-1$ for $m\in M$.  
Since $A$ is assumed to be $p$-complete, it suffices to show that some power of each $[m]-1$ is divisible by $p$.  But we have that
\[
([m]-1)^{p^j} \equiv [m]^{p^j}  -1 = [p^jm]-1\pmod{p}
\]
and so this is $0 \pmod{p}$ for $p^j$ larger than the order of torsion in $M$.
\end{proof}

It follows from \Cref{prop:nil-inv_delta_units} that we may identify $\Gmdred{A[M]}$ with $(J/J^2)^{\delta_p=1}$.  We now compute this group.

\begin{lem}\label{lem:tangentstr}
Let $A$ be a $\delta_p$-ring, let $M$ be a finite abelian $p$-group, and let $J\subset A[M]$ denote the kernel of the augmentation $A[M]\to A$.  Then there is a canonical isomorphism of abelian groups
\begin{align*}
A\otimes M &\simeq J/J^2 \\
a\otimes m &\mapsto a([m]-1)\in J.
\end{align*}
Under this isomorphism, the operation $\delta_p$ on the right-hand side is identified with the endomorphism $\varphi_A \otimes \id_M$ on the left-hand side. 

% Then the square
% \[
% \xymatrix{
% A\otimes M \ar^\wr[d] \ar^{\varphi_A \otimes \Id_M}[r] & A\otimes M \ar^{\wr}[d]\\
% J/J^2\ar^{\delta_p}[r] & J/J^2
% }
% \]
% commutes.
\end{lem}

\begin{proof}
To see this isomorphism of abelian groups, we note that $J$ is a free $A$-module with generators $[m]-1$ for $m\in M$, and that $J^2$ is generated by the relations $([m]-1)([n]-1)=0$, that is,
\[
([m]-1) +([n]-1) \equiv [m+n]-1 \mod J^2.
\]
Thus, the quotient $J/J^2$ is visibly the same as the left-hand side, so it remains to identify these two operations $\varphi_A\otimes \id_M$ and $\delta_p$.

Since $\delta_p(a([m]-1)) \equiv \varphi_A(a) \delta_p([m]-1)$ on $J/J^2$ by \Cref{prop:ugly-delta_identity}(1), it suffices to prove the claim for $a=1$. In this case we have
\[
\delta_p([m]-1) = \frac{\varphi([m]-1) - ([m]-1)^p}{p} = \frac{([m]^p-1) - ([m]-1)^p}{p}.  
\]
The right-hand side is a polynomial $q([m])$ with integer coefficients such that $q(1) = 0$ and $q'(1) = 1$, and hence 
\[
q([m]) \equiv [m]-1 \mod ([m]-1)^2
\]
as desired.
\end{proof}

Given this identification, we can use the results of \S\ref{subsub:AS} to prove the desired result:
\begin{prop}\label{prop:red_delta_p_taut_finite}
        Let $A$ be a $p$-complete $\delta_p$-ring and $M$ be a finite abelian $p$-group. Then there is a natural isomorphism
    \[
    \Gmdred{A[M]^{\wedge}_p} \simeq \underline{M}(A).
    \]
\end{prop}
\begin{proof}
By \Cref{prop:nil-inv_delta_units} and \Cref{lem:tangentstr}, we may identify $\Gmdred{A[M]^{\wedge}_p}$ with the fixed points of $\varphi_A \otimes \mathrm{id}_M$ on $A\otimes M$. Since $M$ is a finite abelian $p$-group, it is a finite direct sum of quotients of the form $\ZZ/p^r$, so we may reduce to the case $M= \ZZ/p^r$. In this case, we have $A\otimes M = A/p^r$ and  
\Cref{cor:Artin_Schreier_delta} identifies the fixed points of $\varphi_A \otimes \id_M$ on this object with $\underline{\ZZ/p^r}(A) = \underline{M}(A)$.  
\end{proof}

\subsubsection{The general case}
\label{subsub:riggen}

\begin{thm}\label{thm:deltarig}
Let $M$ be a finitely generated abelian group whose torsion part is $p$-power torsion and let $A$ be a $p$-complete $\delta_p$-ring. Then
\[
\Gmdred{A[M]^{\wedge}_p} \simeq \underline{M}(A).
\]
\end{thm}

\begin{proof}
Note that there is a canonical map $\ct{M}(A)\to \Gmdred{A[M]^{\wedge}_p}$ locally given by $m\mapsto [m]$. We shall refer to its image as the ``tautological units''. Note further that the assumption on $M$ implies that $\bigcap_{r}p^rM = \{0\}$. As a result, for every finite subset $T\subseteq M$, there exists $r>0$ such that $T$ injects into $M/p^r$. 

Given a reduced unit $f\in \Gmdred{A[M]^{\wedge}_p}$ of rank $1$, we wish to show that we can find a \emph{finite} decomposition 
$A\simeq \prod_i A_i$ such that the image of $f$ in $A_i[M]^\wedge_p$ is of the form $[m]$ for some $m\in M$. 
Expand $f$ as a $p$-adically convergent series  
\[
f= \sum_{m\in M} x_m [m],   
\]
and let $f_{r,s}$ be the image of $f$ in $A/p^r[M/p^s]$.  In this notation, we will allow $r=\infty$ or $s=\infty$, in which case we write $f_{\infty,s}$ (resp. $f_{r,\infty}$) to mean the image of $f$ in $A[M/p^s]$ (resp. $A/p^r[M]$). Thus, our aim is to show that $f:=f_{\infty,\infty}$ is a tautological unit.  

Since $f_{\infty,s}$ is a reduced unit of rank $1$, by \Cref{prop:red_delta_p_taut_finite} it is a tautological unit for every $s\in \NN$. On the other hand, $f_{1,\infty}$ is an element of $A/p[M]$ so it is a sum $f_{1,\infty} = \sum_{m\in T} \ol{x}_m[m]$, for some finite $T\subseteq M$. Choosing $s\in \NN$ for which the map $T\to M/p^s$ is injective, we see that the fact that $f_{1,s}$ is tautological implies the same for $f_{1,\infty}$. This implies that there is a finite decomposition $A/p\simeq \prod_i \ol{A}_i$ such that the image of $f_{1,\infty}$ in $\ol{A}_i[M]$ is tautological. 
By Hensel's lemma, we can lift the mentioned decomposition to a decomposition $A \simeq \prod_i A_i$. Replacing $A$ by one of the $A_i$'s, we may thus assume that $f_{1,\infty} = [m]$ for some $m\in M$. Multiplying by $[-m]$, we are thus reduced to the case where $f_{1,\infty} = 1$. 
We claim that in this case, we already have $f = 1$, and in particular that it is tautological. Indeed, assuming the contrary, let $0\ne m_0\in M$ be an element for which $x_{m_0}\ne 0$. Choose $r$ large enough so that $x_m \neq 0 \mod p^r$, and write again 
\[
f_{r,\infty} = \sum_{m\in T'}x_m [m]
\] 
for some finite set $T'\subseteq M$. Choose $s$ large enough so that the composition $T' \into M \onto M/p^s$ is injective. Then, the fact that $x_{m_0}\ne 0 \mod p^r$ implies that 
\[
f_{r,s} = 1 + \ol{x}_{m_0}[m_0] + \dots,
\]
 where $\ol{x}_{m_0}$ is the (non-vanishing) reduction of $x_{m_0}$ modulo $p^r$. Since $f_{r,s}$ is the image of $f_{\infty,s}$ under the map $A\to A/p^r$, it is also tautological, and since the coefficient of $0\in M$ is invertible, this forces all other coefficients, including $\ol{x}_{m_0}$, to be $0$. This is a contradiction and we are done.
\end{proof}

\subsection{$\widehat\delta$-rings}\label{sub:deltahat}

We now introduce a global variant of $\delta_p$-rings which we call $\widehat{\delta}$-rings. In contrast with the theory of $\lambda$-rings, our notion of $\widehat{\delta}$-rings ``decouples'' different primes. 

\begin{defn}
    Define the category of $\widehat \delta$-rings by the pullback
    \[
    \begin{tikzcd}
        \widehat \delta\mathrm{-Ring} \arrow[r]\arrow[d] & \prod_p \Acr^{\delta_p}\arrow[d]\\
        \CAlg_\Z^{\heartsuit} \arrow[r,"\prod (-)^{\wedge}_p"] & \prod_p \Acr^{\wedge}_p.
    \end{tikzcd}        
    \]
    In other words, a $\widehat \delta$-ring is a (discrete) commutative ring $R$ with a lift of each derived $p$-completion $R^{\wedge}_p$ to an animated $\delta_p$-ring.  
\end{defn}

Although we assume the underlying ring $R$ is discrete, its (derived) $p$-completions may not be, and thus we need to work in the animated context for the above definition.  Nevertheless, it is easy to see that the category of $\widehat\delta$-rings is (equivalent to) a 1-category. 
Additionally, we will generally only consider $\widehat \delta$-rings with bounded torsion, so this point will not be important.  

\begin{rem}
    If $R$ is already derived $p$-complete, then a $\widehat{\delta}$-ring structure on $R$ is exactly given by a $\delta_p$-ring structure, since all other completions vanish. On the other hand, if $R$ is rational, then a $\widehat \delta$-ring structure on $R$ is no data at all. 
\end{rem}
%We also note that while the definition of $\widehat\delta$-rings 
%involves animated commutative rings, it is easy to see that the category of $\widehat\delta$-rings is (equivalent to) a 1-category. 

As in \Cref{cnstr:groupalgdelta}, $\Z[M]$ for an abelian group $M$ acquires a natural $\widehat \delta$-ring structure.

\begin{defn}
    Define the group of rank $1$ $\widehat \delta$-units of $R$ by 
    \[
\mdef{\Gmdhat{R} := \Map_{\dhring}(\Z[t^{\pm 1}],R)}.
    \]
\end{defn}

While $\Gmdhat{R}$ can in general be non-discrete, we will typically consider $R$ with bounded torsion, in which case $\Gmdhat{R}$ can be described as the subgroup of units $t \in R^{\times}$ such that $\delta_p(t)=0$ for all $p$.

This notion of $\widehat{\delta}$-rings allows us to make an integral analogue of the previous results in the section, which we view as the algebraic analogue of our main theorem (\Cref{theorem_one}).

\begin{thm}\label{thm:delta_gp_alg_ff}
Let $M$ and $N$ be finitely generated abelian groups.  Then the natural map
\[
\Hom_{\mathrm{Ab}}(N,M) \to \Hom_{\dhring}(\Z[N], \Z[M])
\]
is an isomorphism. In particular, $\Gmdhat{\ZZ[M]}\simeq M$.  
\end{thm}
\begin{proof}
Both sides send colimits in $N$ to limits, so it suffices to consider $N = \ZZ$, where the statement becomes $\Gmdhat{\ZZ[M]}\simeq M$.  Here, write $M = M'\oplus \Lambda$ where $\Lambda$ is free and $M'$ is finite. 
Since $\ZZ[M']$ has connected Zariski spectrum, we have an isomorphism $\G_m(\ZZ[M'][\Lambda])\simeq  \G_m(\ZZ[M'])\oplus \Lambda$ (cf. \cite[Theorem 1]{karpilovsky1983finite}).  It follows that
\[
\Gmdhat{\ZZ[M'][\Lambda]}\simeq  \Gmdhat{\ZZ[M']}\oplus \Lambda.
\]
Hence, it suffices to prove the result for $M$ finite.  We start by showing that in this case $\Gm^{\hat{\delta}}(\ZZ[M])$ is finite. Choose a prime $p$.  Since the $p$-completion map
\[
\Gm^{\hat{\delta}}(\ZZ[M])\to \Gm^{\delta_p}(\ZZ[M]^\wedge_p)
\]
is injective, it suffices to show that 
$\Gm^{\delta_p}(\ZZ[M]^\wedge_p)$ is finite. Writing $M = M_p \oplus M_{\bar{p}}$ where $M_p$ is a $p$-group and $M_{\bar{p}}$ is of order prime to $p$, we have by \Cref{prop:red_delta_p_taut_finite} that
\[\Gm^{\delta_p}(\ZZ[M]^\wedge_p) \simeq \ct{M_{p}}(\ZZ[M_{\bar{p}}]^{\wedge}_p) \oplus \Gm^{\delta_p}(\ZZ[M_{\bar{p}}]^\wedge_p).
\] 
The first summand is finite because $\Spec(\ZZ[M_{\bar{p}}]^\wedge_p)$ has finitely many connected components, and the second is finite because $\ZZ[M_{\bar{p}}]^\wedge_p$ is perfect and hence by \Cref{exm:wittunit} we have 
\[
\Gm^{\delta_p}(\ZZ[M_{\bar{p}}]^\wedge_p) \simeq \Gm(\FF_p[M_{\bar{p}}]).
\]

Finally, once we know $\Gm^{\hat{\delta}}(\ZZ[M])$ is finite, the result follows from \cite[Theorem 2]{higman1940units}, which asserts that all torsion reduced units in the group algebra are in the image of the map $M\to \Gm(\ZZ[M])$.
\end{proof}

\subsection{$\delta$-rings and commutative ring spectra}\label{sub:moore}
Now we will link the existence of $\widehat \delta$-ring structures on $R$ to lifts $\mathbb{S}_R$ to the sphere spectrum. 
%The relation between the notions originates from the Frobenius map of ring spectra. 
Recall from \cite[\S IV.1]{NS} that a commutative ring spectrum $R$ admits two natural maps to the Tate construction $R^{tC_p}$: the Tate-valued Frobenius
\[
\frob\colon R\to R^{tC_p},
\]
given by the composite of the Tate diagonal $R\to (R^{\otimes p})^{tC_p}$ with multiplication, and the map 
\[
\triv\colon R\to R^{tC_p}.
\]
given by the composite $R\to R^{hC_p} \to R^{tC_p}$. Note that the latter is defined for an arbitrary spectrum and does not depend on the multiplication, but when $R$ is a commutative ring spectrum then $\triv$ is a commutative ring map.  
%\begin{defn}
%A \tdef{spectral $\delta_p$-ring} is a commutative ring spectrum $R$ together with a lift of the Tate-valued Frobenius $R\to R^{tC_p}$ along the trivial map $R\to R^{tC_p}$: 
%\[
%\xymatrix{
%R\ar@{..>}^{\varphi_R}[rr]\ar_{\text{frob}}[rd] & & R\ar^{\text{triv}}[ld] \\ 
%  &R^{tC_p} &
%}
%\]
%\end{defn}

%The main example for us is the following 

%\begin{exm}
%Let $R$ be a connective commutative ring spectrum for which the trivial map $R\to R^{tC_p}$ is an isomorphism (so that, in particular, $R$ is $p$-complete). Then $R$ admits a unique structure of a spectral $\delta_p$-ring, for which 
%\[
%\varphi_R= \text{frob}\circ \text{triv}^{-1}.
%\]
%\end{exm}

%This includes the following key family of examples. 
The Segal conjecture supplies a large collection of spectra for which the $\triv$ map is an isomorphism. 
\begin{defn}\label{defn:finiteTor}
Let $X$ be a spectrum. We say that $X$ is of \tdef{finite Tor-amplitude} if $X$ is bounded below and $X\otimes \ZZ$ is bounded. Equivalently, $X$ admits a cell structure with cells of bounded dimension. We further say that $X$ is a \tdef{Moore spectrum} if $X\otimes \ZZ$ is concentrated in degree $0$.      
\end{defn}

%Note that if $X$ is $p$-complete and bounded below, it is enough to check that $X\otimes \FF_p$ is bounded. Furthermore, in this case $X$ admits a finite dimensional $p$-complete cell structure: there is a finite filtration on $X$ with associated graded pieces of the form $\Sigma^k(\bigoplus_I \Sph)^\wedge_p$. 

\begin{prop}[{{\cite[Theorem 1.2, Example 4.4]{burklund2024note}}, cf.\ also \cite[Lemma 6.7]{yuan2022integral}}]\label{prop:Moore_triv_iso}
Let $X$ be a spectrum of finite Tor-amplitude. Then, 
$\triv: X\to X^{tC_p}$
exhibits $X^{tC_p}$ as the $p$-completion of $X$.
\end{prop}

%\begin{proof}
%Since $X$ has finite Tor-amplitude, it admits a filtration with associated graded pieces of the form $\Sigma^k \bigoplus_I \Sph$, for which the claim is \cite[Lemma 6.7]{yuan2022integral}.%and it suffices to show the claim for these pieces.  This follows from the Segal conjecture for the group $C_p$ and the fact that $(-)^{tC_p}$ commutes with such $p$-completed infinite direct sums (cf. \cite[Lemma 6.7]{yuan2022integral}).
%We first note that for any bounded below spectrum $X$ the Tate construction is $p$-complete. Thus in order to show the statement is enough to show that the map is a $p$-adic equivalence and we may even assume without loss of generality that $X$ is $p$-complete. In this case we have to show that the map is an equivalence. 

%Since $X$ is $p$-complete of finite Tor-amplitude, we can put on $X$ a finite filtration with associated graded pieces of the form $\Sigma^k(\bigoplus_I \Sph)^\wedge_p$ for some sets $I$. Since $(-)^{tC_p}$ is exact, it remains to show that for every (possibly infinite) set $I$, the map 
%\[
%(\bigoplus_I \Sph)^\wedge_p \to ((\bigoplus_I \Sph)^{\wedge}_p)^{tC_p} 
%\]
%is an isomorphism. This follows from the Segal conjecture for the group $C_p$ (\cite{lin1980conjectures, adams1992segal}) and the fact that $(-)^{tC_p}$ commutes with such $p$-completed infinite direct sum (see \cite[Lemma 6.7]{yuan2022integral}).
%\end{proof}

Thus, any commutative ring spectrum $R$ of finite Tor-amplitude admits a canonical map 
\[
\mdef{\varphi_R}:= \triv^{-1}\circ \frob \colon R\to R^\wedge_p,
\]
which we call the \tdef{Frobenius map} of $R$.  Our goal for the rest of the section is to show that in the special case where $R$ is a Moore spectrum (i.e., $R\otimes \Z$ is \emph{discrete}), these Frobenius maps for varying $p$ induce a $\widehat{\delta}$-ring structure on $\pi_0R$.  We need a few preliminary lemmas.  

\begin{lem}\label{lem:acrkan}
    Let $F,G :\Acr \to \mathcal{C}$ be functors and $F_0,G_0$ denote their restriction to the full subcategory $\mathrm{Poly}_{\Z}\subset \Acr$ of (discrete) polynomial algebras.  If $F$ commutes with sifted colimits, then there is an isomorphism of spaces of natural transformations
    \[
    \mathrm{Nat}(F,G) \simeq \mathrm{Nat}(F_0,G_0).
    \]
\end{lem}

\begin{proof}
    This follows from the fact that $\Acr = P_{\Sigma}(\mathrm{Poly}_{\Z})$ is the the animation of the category of polynomial rings. 
\end{proof}

\begin{lem} \label{lem:Tate_derived_mod_p}
    Let $R$ be a discrete commutative ring, so that the Frobenius and trivial maps $R\to R^{tC_p}$ factor uniquely through $\tau_{\ge 0}R^{tC_p}$. 
Then, there is a natural isomorphism $\tau_{[0,1]} R^{tC_p} \simeq R\mm p$ in $\calg_\ZZ^\cn$ fitting into the commutative diagrams
\[
\begin{tikzcd}
R \arrow[r,"\triv"]\arrow[d]& \tau_{\geq 0}R^{tC_p} \arrow[d,"\wr"] & & & R \arrow[r,"\frob"]\arrow[d]& \tau_{\geq 0}R^{tC_p}\arrow[d,"\wr"]\\
R\mm p \arrow[r,"="]& R\mm p \simeq \tau_{[0,1]}R^{tC_p} & & & R\mm p \arrow[r,"\varphi_{R\mm p}"]& R\mm p \simeq \tau_{[0,1]}R^{tC_p}.\\
\end{tikzcd}
\]
\end{lem}
\begin{proof}
Note that for $R\in \mathrm{Poly}_\Z$, the constructions $R\mm p$ and $\tau_{[0,1]}R^{tC_p}$ are both given by the discrete ring $R/p$ and thus are naturally identified.  By \Cref{lem:acrkan}, we obtain a natural transformation 
\[
R\mm p \xrightarrow{\nu} \tau_{[0,1]}R^{tC_p}
\]
which is an isomorphism at least for $R\in \mathrm{Poly}_\Z$.  Moreover, the diagrams commute because they commute for $R\in \mathrm{Poly}_\Z$; this is clear for the first one and for the second, it amounts to the statement that $\mathrm{frob}$ induces the usual Frobenius on $\pi_0R/p$ \cite[Example IV.1.2]{NS}.

It remains to check that $\nu$ is an isomorphism for \emph{all} discrete rings.
Let $G(R): \CAlg^{\mathrm{an}}_{\Z} \to \Mod_{\Z}$ be the cofiber of the composite
\[
R_{hC_p} \xrightarrow{\mathrm{Nm}} R^{hC_p} \xrightarrow{\mathrm{res}} R.
\]
Note that we have a natural transformation $\tau_{[0,1]}R^{tC_p} \to \tau_{[0,1]}G(R)$ which is an isomorphism on discrete rings.  Therefore, it suffices to show that for $R$ discrete, the induced map
\[
\nu': R\mm p \to \tau_{[0,1]}G(R)
\]
is an isomorphism.  But the source and target both preserve colimits, as functors to $\Mod_{\Z}^{[0,1]}$, so the conclusion follows from the fact that $\nu$, and therefore $\nu'$, is an isomorphism for $R\in \Poly_{\Z}$.

% Note that by the standard complex computing Tate cohomology, the functor

% Since abstractly we can identify
% \begin{align*}
% \pi_0R\mm p &\simeq R/p  \simeq \hat{H}^0(C_p;R) \simeq \pi_0R^{tC_p}\\
% \pi_1R\mm p &\simeq R[p] \simeq \hat{H}^{-1}(C_p;R) \simeq \pi_1R^{tC_p},
% \end{align*}
% our natural transformation provides us with endomorphisms of $R\mapsto R/p$ and $R\mapsto R[p]$ which we want to show are the identity.  This is clear by construction for the former (by mapping in a polynomial ring).  For the latter, consider any element $\alpha \in R[p]$. 

% it is enough t

% it is easy to see that the natural transformation is always an isomorphism in $\pi_0$.

% To see that it also induces an isomorphism on $\pi_1$, it suffices to consider the universal example of a $p$-torsion element $R = \Z[x]/px$.  But here we observe that in both $R\mm p$ and $R^{tC_p}$, the degree $0$ and $1$ classes corresponding to $x$ are connected by a Bockstein, which must be respected by any $\Z$-linear map.  \todo{Allen: I added this, please check and then you can erase this note.}

% By \Cref{lem:acrkan}, this induces a natural transformation $R\mm p \to \tau_{[0,1]}R^{tC_p}$ which is easy to check is an isomorphism when $R$ is discrete, as we have \todo{explain better.}
% \begin{align*}
% \pi_0M\mm p &\simeq M/p \simeq \widehat{H}^0(C_p;M),\\
% \pi_1M\mm p &\simeq M[p] \simeq \widehat{H}^{-1}(C_p;M).
% \end{align*}

\end{proof}

\begin{lem} \label{lem:Acr2Comm_truncated}
Let  $R,S\in \Acr$ and assume that $S$ is $1$-truncated. Then 
\[
\Map_{\Acr}(R,S) \to \Map_{\calg_\ZZ}(R,S)
\]
is an inclusion of connected components, which is an equivalence if $\pi_1(S)$ does not have $2$-torsion. 
\end{lem}
%In particular: the restriction of the forgetful functor $\Acr \to \calg_\ZZ$ to the $1$-truncated algebras is a faithful embedding of $\infty$-categories. 
%Moreover the further restriction to those which are $1$-truncated and have 2-torsion free $\pi_1$ is fully faithful.

\begin{proof}
    First, since the condition of the theorem is closed under limits, it suffices to consider the case $R = \Z[x]$ where we consider the map $\Omega^{\infty} S \to \Map_{\calg_\ZZ}(\Z[x],S)$.
   % The forgetful functor $\Acr \to \calg_\ZZ^{\mathrm{cn}}$ admits a right adjoint given on underlying by $S\mapsto \Map_{\calg_\ZZ}(\ZZ[x],S)$ (cf. \cite[Rem. 2.6.3]{DAG}).  Hence, it suffices to show that the natural map $ \Map_{\calg_\ZZ}(\ZZ[x],S) \to S$ satisfies the conclusion of the theorem. 
    Letting $\ZZ\{ x\} = \bigoplus_{n\geq 0} \Z [B\Sigma_n]$ denote the free commutative $\ZZ$-algebra on a degree 0 class, we note that $\pi_0 \ZZ\{ x\} = \Z[x]$ and $\pi_1 \ZZ\{ x\} = \beta \Z/2[x]$ (for instance because the sign character $\Sigma_n \to C_2$ is the abelianization map).  Consequently, defining $P$ by the pushout
    \[
    \begin{tikzcd}
        \ZZ\{ S^1/2\} \arrow[r]\arrow[d,"\beta"]& \ZZ \arrow[d] \\
        \ZZ \{ x \} \arrow[r] & P,
    \end{tikzcd}
    \]
    the natural map $P \to \Z[x]$ is an isomorphism in degrees $0,1$ and surjective on $\pi_2$.  It follows that $\Map(P,S) \simeq \Map(\Z[x],S)$, so it suffices to show that $\Omega^{\infty} S \to \Map_{\calg_\ZZ}(P,S)$. The conclusion then follows by considering commutative $\Z$-algebra maps from the above pushout square into $S$.
\end{proof}

\begin{lem}\label{lem_completion}
Let $X$ be a bounded below spectrum. Then 
\[
(X^\wedge_p) \otimes \ZZ = (X \otimes \ZZ)^\wedge_p.
\]
\end{lem}
\begin{proof}
This follows because $(X^\wedge_p) \otimes \ZZ$ is $p$-complete \cite[Lemma 3.3]{BB} and the natural map $
(X^\wedge_p) \otimes \ZZ\to (X \otimes \ZZ)^\wedge_p$
becomes an isomorphism after tensoring with $\Sph/p$.

% we get a map of $p$-complete spectra
% \[
% (X \otimes \ZZ)^\wedge_p \to (X^\wedge_p) \otimes \ZZ
% \]
% and this is clearly an isomorphism after tensoring with $\Sph/p$. 
%We first claim that $(X^\wedge_p) \otimes \ZZ$ is $p$-complete. To this end we simply note that $\ZZ$ is finite type and can thus be written as a colimit of finite spectra along maps of increasing connectivity. Tensoring $(X^\wedge_p)$ with a finite spectrum is still $p$-complete and therefore the colimit remains $p$-complete since it still has increasing connectivity by virtue of the fact that $X$ is bounded below.  \atodo{Tweak wording}
%Once we know that 
\end{proof}

We can now prove the main statement about the relationship between commutative ring Moore spectra and $\widehat \delta$-rings:

\begin{prop} \label{prop:Moore_delta_ring}
Let $\CAlg^{\mathrm{M}}\subset \CAlg(\Sp)$ denote the full subcategory of commutative ring Moore spectra.  Then for $R\in \CAlg^{\mathrm{M}}$, $\pi_0(R)$ acquires a natural $\widehat{\delta}$-ring structure, i.e., $\pi_0(-)$ lifts to a functor
\[
\pi_0: \CAlg^{\mathrm{M}}\to \dhring.
\]
%Let $R$ be a commutative ring spectrum whose underlying spectrum is a Moore spectrum. 
%Then $\pi_0R$ admits a natural structure of a $\widehat{\delta}$-ring.
\end{prop}
\begin{proof}
For each prime $p$, consider the diagram of commutative ring spectra
% \[
% \xymatrix{
% R^{\wedge}_p \ar^{\varphi_R}[rr] \ar[dd]\ar^{\frob}[rd]        &                  & R^\wedge_p\ar[dd]\ar^{\triv}[ld]  \\
%                               &   R^{tC_p}\ar[dd]       &     \\
% (\pi_0R)^{\wedge}_p\ar^{\frob}[rd]\ar[dd]                        &                  & (\pi_0R)^\wedge_p\ar^{\triv}[ld] \ar[dd] \\
%                               &      \tau_{\ge 0}(\pi_0R)^{tC_p}\ar[rd] & \\
%  \pi_0R\mm p\ar^{\varphi_{\pi_0R\mm p}}[rr]                              &                                  & \pi_0R\mm p,  \\
% }
% \]

\[
\begin{tikzcd}[column sep=1cm, row sep=1.5cm]
R^{\wedge}_p \arrow[rd,"\frob"] \arrow[rr]\arrow[dd,"\varphi_R"']        &       & (\pi_0R)^{\wedge}_p\arrow[rd,"\frob"]\arrow[rr]       & &\pi_0R\mm p\arrow[dd,"\varphi_{\pi_0R\mm p}"]   \\
 & R^{tC_p}\arrow[rr]    & &  \tau_{\ge 0}(\pi_0R)^{tC_p}\arrow[rd] &\\
R^\wedge_p\arrow[rr]\arrow[ru,"\triv"']  & &(\pi_0R)^\wedge_p\arrow[ru,"\triv"'] \ar[rr]  & & \pi_0R\mm p
\end{tikzcd}
\]
where the right vertical map is induced by the animated Frobenius.  
The commutativity of the diagram is clear other than the right two triangles, which commute by \Cref{lem:Tate_derived_mod_p}.  
Since the right vertical map is $\Z$-linear and $\Z \otimes R^{\wedge}_p \simeq (\pi_0 R)^{\wedge}_p$ by \Cref{lem_completion}, the outer square induces a square
\[
\begin{tikzcd}
    (\pi_0 R)^{\wedge}_p \arrow[r,"\varphi_R \otimes \Z"]\arrow[d] & (\pi_0 R)^{\wedge}_p\arrow[d]\\
    \pi_0R\mm p \arrow[r,"\varphi_{\pi_0 R\mm p}"] & \pi_0 R\mm p.
\end{tikzcd}
\]
To see that this induces an animated $\delta_p$-ring structure on $(\pi_0 R)^{\wedge}_p$, it suffices to lift the left vertical map to an animated ring map.  But $(\pi_0 R)^{\wedge}_p$ is $1$-truncated, and we claim there is no $2$-torsion in $\pi_1$ in case $p=2$.  Indeed,  if there were, then $\pi_0(R)^{\wedge}_2 \mm 2 = \pi_0(R)\mm 2$ would have nontrivial $\pi_2$, which is impossible.  Thus, by \Cref{lem:Acr2Comm_truncated}, $\varphi_R \otimes \Z$ lifts uniquely to an animated map, as desired.% and we have our desired derived Frobenius lift.  
\end{proof}

\begin{rem}
    In fact, A. Krause has shown that $\widehat{\delta}$-ring structures on a torsion-free commutative ring are \emph{equivalent} to $H_{\infty}$-ring structures on the corresponding Moore spectrum \cite{KrauseH}.    
\end{rem}

%% file: groups.tex
Let $R_0$ be a (discrete) commutative ring and let $M$ be a finitely generated abelian group.  Then there is a natural map of abelian groups
\begin{align*}
    \theta: M &\to R_0[M]^{\times}\\
    m &\mapsto  [m]
\end{align*}
picking out what we call the \emph{tautological units}.  There are at least two reasons why this map fails to be an isomorphism:
\begin{enumerate}
    \item The target also contains units coming from $R_0^{\times}$.  
    \item If $R_0$ splits as a product $R_0' \times R_0''$, then we have units of the form $([m],[m'])\in (R_0'\times R_0''[M])^{\times}$.
\end{enumerate}

These obstructions carry over verbatim when $R_0$ is replaced by a commutative ring spectrum $R$ and units are replaced by strict units.  To side-step (1), we introduce the following definition:

\begin{defn}
Let $R$ be a commutative ring spectrum and $S$ be an augmented commutative $R$-algebra (e.g. $R[M]$ or $R[M]^{\wedge}_p$ with their canonical augmentations).  Then we denote
\[
\mdef{\Gmred(S) := \fib(\Gm(S) \to \Gm(R))}
\]
so that there is a splitting 
\[
\Gm(S) \simeq \Gm(R) \oplus \Gmred(S).
\]
We refer to $\Gmred(S)$ as the \mdef{reduced strict units} of $S$.  \footnote{We warn that the notation is slightly abusive, as $\Gmred(S)$ depends on both $S$ and its augmentation, but we will clarify whenever there is risk for confusion.}
\end{defn}

Thus, the question of computing $\Gm(R[M])$ reduces to separately computing $\Gm(R)$ and $\Gmred(R[M])$.  In our case of interest $R=\Sph$, the former vanishes by \cite{CarmeliStrict}, and so we begin our analysis in \S\ref{sub:gmred} with the basic properties of reduced units $\Gmred(-)$.

Next, while the map $\theta$ lands in reduced units, (2) still obstructs it from being an isomorphism.  Thus, we consider instead a certain natural map
\[
\Theta: \underline{M}(R) \to \Gmred(R[M])
\]
where $\underline{M}(R)$ denotes the group of locally constant $M$-valued functions on $\Spec(R)$.  The map $\Theta$ is a better approximation to $\Gmred(R[M])$ than $\theta$, and turns out to be an isomorphism for a large class of commutative ring spectra $R$.  We refer to this feature of $(R,M)$ as \emph{rigidity}.  Much of the paper is dedicated to proving rigidity in various situations, so we set notation more carefully in \S\ref{sub:rigidity}.

Finally, in \S\ref{sub:staticity}, we discuss a condition under which rigidity is easy to check.  We say that a pair $(R,M)$ is $\Gm^{\mathrm{red}}(-)$-\emph{static} if the natural map
\[
\Gmred(R[M]) \to \Gmred(\pi_0R[M])
\]
is an inclusion of connected components.  In this case, rigidity reduces to an essentially algebraic question of checking whether all (reduced) strict units of $R[M]$ are tautological at the level of $\pi_0$ (cf. \Cref{prop:static_rigid}).  

\subsection{Strict units, reduced units, and descent}\label{sub:gmred}

The main object of study in this paper is the following:

\begin{defn}\label{defn:strictunit}
    A \mdef{strict unit} in a commutative ring spectrum $R$ is a map of commutative ring spectra
    \[
    \Sph[\Z] =: \Sph[t^{\pm 1}] \to R.
    \]
    We denote the spectrum of strict units by $\Gm(R)$, so that
    \[
    \Omega^{\infty}\Gm(R) = \Map_{\CAlg}(\Sph[t^{\pm 1}],R) \simeq \Map_{\Sp}(\Z, gl_1(R)).
    \]
    In fact, $\Gm(R)$ naturally has a connective $\Z$-module spectrum structure \cite[Cons. 1.6.10]{Lurie_Ell2}.  
\end{defn}

%\todo{At some point, we should say a little bit about the history of studying strict units, whether here or in the intro.  Should mention Junhou, hopkins lurie rezk, and maybe a few reasons why its interesting in homotopy theory.}

As mentioned earlier, we will need the following variant of strict units:

\begin{defn}[Reduced units]
Let $R$ be a commutative ring spectrum and $S$ be an augmented commutative $R$-algebra.  Then we denote
\[
\mdef{\Gmred(S) := \fib(\Gm(S) \to \Gm(R))}
\]
so that there is a splitting 
\[
\Gm(S) \simeq \Gm(R) \oplus \Gmred(S).
\]
We refer to $\Gmred(S)$ as the \mdef{reduced strict units} of $S$.  \footnote{We warn that the notation is slightly abusive, as $\Gmred(S)$ depends on both $S$ and its augmentation, but we will clarify whenever there is risk for confusion.}
\end{defn}

For us, we will generally take $S$ to be $R[M]$ or $R[M]^{\wedge}_p$ with their canonical augmentations.

\begin{rem}\label{rem:augmented}
The reduced strict units can alternately be interpreted as a space of \emph{augmented maps}: 
\[
\Omega^\infty\Gmr{R}{M} \simeq \Map_{\calg_R^\aug}(R[\ZZ],R[M]).
\]
Accordingly, we have for every $R\in \CAlg$ an adjunction
\[
R[-]\colon \Mod_\ZZ^\cn \adj \calg^{\aug}_{R}: \Gmred(-).
\]
\end{rem}
%need to define strict units as a thing landing in Z (give citation) and give that augmentation adjunction.

\subsubsection{Descent for reduced units}

The construction $\Gmred(R[M])$ is functorial in both the ring $R$ and the group $M$.  In the ring $R$, we have the following descent statement:

\begin{prop}\label{prop:red_units_sheaf_affine}
For $M\in \Mod_{\Z}^{\cn}$, the functors $\G_m(-[M]), \Gmr{-}{M} \in \Fun(\CAlg, \Mod_{\Z}^{\cn})$ are sheaves for the Zariski topology.  If $M$ is additionally finite and discrete, then $\G_m(-[M])$ and $\Gmr{-}{M}$ are moreover affine group schemes (i.e., they are co-representable by commutative $R$-algebras). 
\end{prop}
%Note that the claim for finite abelian groups means that $\Gmr{-}{M}$ preserves all small limits in the ring variable.  
\begin{proof} 
Since $\Gmred{-[M]}$ is a retract of $\Gm(-[M])$, it suffices to prove the claims for the latter functor.
Since the Zariski topology is finitary, the sheaf condition for $\Gm(-[M])$ amounts to the construction $R \mapsto \Gm(R[M])$ commuting with a finite limit, which follows from the facts that $R\mapsto R[M]$ preserves finite limits and $\Gm$ preserves all limits.  If $M$ is finite, then $R\mapsto R[M]$ and $\Gm$ are accessible and limit preserving, so the result follows from the Adjoint Functor Theorem.
\end{proof}

The behavior of $\Gmred(R[M])$ as $M$ changes is more subtle, but there are some cases when it interacts with limits or colimits in a predictable way.  First, while $\Gmred(R[-])$ does not preserve direct sums, one can always decompose $\AG{R[M\oplus N]}$ into simpler pieces.

\begin{obs} \label{prop:red_units_dsum}
Let $R$ be a commutative ring spectrum and let $M,N\in \Mod_\ZZ^\cn$. Then there are canonical isomorphisms
\begin{align*}
\Gmr{R}{M\oplus N} &\simeq \Gmr{R}{M}\oplus \Gmr{R[M]}{N}, \\
\Gmrp{R}{M\oplus N} &\simeq \Gmrp{R}{M}\oplus \Gmrp{R[M]}{N}
\end{align*}
where we regard $R[M][N]$ as being augmented over $R[M]$ (and similarly for the $p$-complete version).
\end{obs}
% \begin{proof}
% We prove the non $p$-complete version, as the two proofs are analogous.  
% By definition, $\Gmr{R}{M\oplus N}$ is the fiber of the map $\Gm(R[M\oplus N]) \to \Gm(R)$ induced from the augmentation $R[M\oplus N] \to R$. This augmentation is the composite 
% $R[M\oplus N] \to R[M] \to R$, so we can form the natural Cartesian diagram 
% \[
% \xymatrix{
% \AG{R[M][N]}\ar[r]\ar[d]& \AG{R[M\oplus N]}\ar[r]\ar[d] & \G_m(R[M\oplus N])\ar[d] \\ 
% 0\ar[r] &\AG{R[M]}\ar[d]\ar[r]  & \G_m(R[M])\ar[d] \\ 
%  &0\ar[r] & \G_m(R) 
% }
% \]
%     % Since the upper left vertical map admits a section, so does the map $\AG{R[M\oplus N]^\wedge_p} \to \AG{R[M]^\wedge_p}$, so the upper right square pullback square gives a split fiber sequence 
%     % \[
%     % \AG{R[M][N]^\wedge_p} \to \AG{R[M\oplus N]^\wedge_p} \to \AG{R[M]^\wedge_p}.
%     % \]
%     Since the upper right vertical map admits a section, the upper left square gives the desired split exact sequence.  
% \end{proof}

We also have the following descent statement along injections $M\to N$ of abelian groups.  % preservation of a certain collection of limit diagrams in $M\in \Mod_\ZZ^\cn$.

\begin{defn} \label{def:Cech_Nerve}
For a morphism $\alpha \colon M\to N$ in $\Mod_\ZZ$, its \tdef{\v{C}ech nerve} is the cosimplicial diagram
\[
\mdef{C^\bullet(\alpha)} :=\left(
\begin{tikzcd}
	{N} & {N\oplus_M N} & {N\oplus_M N\oplus_M N} & {}
	\arrow[shift right=0.8, from=1-1, to=1-2]
	\arrow[shift left=0.8, from=1-1, to=1-2]
	\arrow[from=1-2, to=1-3]
	\arrow[shift right=1.6, from=1-2, to=1-3]
	\arrow[shift left=1.6, from=1-2, to=1-3]
	\arrow[dotted, no head, from=1-3, to=1-4]
\end{tikzcd}    
\right) \qin (\Mod_\ZZ^{\cn})^{\Delta}
\]
obtained as an augmented cosimplicial object from $\alpha$ by left Kan extension along $\Delta_+^{\le 0} \into \Delta_+$.  %Note that we have an isomorphism $\invlim C^\bullet(\alpha)\simeq M$.\footnote{By stability, it suffices to verify the isomorphism after applying $-\oplus_M N$.  But there, it follows because the cosimplicial object is split.} 
\end{defn}
%\tcomment{We should give an argument or reference. This follows by stability... OK -- I put in an argument, if it looks ok, can erase this. }
% Moreover, the constructions $\alpha\mapsto C^\bullet(\alpha)$ assemble to an exact functor 
% \[
% C^\bullet(-)\colon \Mod_\ZZ^{\Delta^1} \to \Mod_\ZZ^{\Delta}.
% \] 

%The functors $\Gmr{R}{-}$ and $\Gmrp{R}{-}$ satisfy descent along injections of abelian groups, in the following sense. 
\begin{prop}\label{prop:descent_inclusion_groups}
Let $\alpha \colon M \into N$ be an injective map of abelian groups.  Then the natural maps
\begin{align*}
\Gmr{R}{M} &\to \invlim\Gmr{R}{C^\bullet(\alpha)}  \\
\Gmrp{R}{M}&\to \invlim\Gmrp{R}{C^\bullet(\alpha)} 
\end{align*}
are isomorphisms.
\end{prop} 
\begin{proof}
We consider the non $p$-complete case first.  Since $\Gmr{R}{M}$ is the fiber of the map $\Gm(R[M]) \to \G_m(R)$ and constant functors preserve contractible limits, it suffices to prove the analogous statement for $\Gm(R[M])$.  This follows from the fact that $\Gm$ preserves limits and $R[M]\to R[N]$ is descendable in the sense of \cite[Definition 3.18, Proposition 3.20]{AkhilGalois}, as it admits a retract.  The $p$-complete case follows similarly, using that $R[M]/p \to R[N]/p$ admits a retract.

% We consider the non $p$-complete case first. Since 
% $\Gmr{R}{M}$ is the fiber of the map 
% $\Gm(R[M]) \to \G_m(R)$ in $\Mod_\ZZ^\cn$, since constant functors preserve contractible limits, and since the formation of fibers between functors is compatible with limits, it suffices to show that %$\G_m(R[-]^\wedge_p)\colon \Mod_\ZZ^\cn \to \Mod_\ZZ^\cn$ carries the cone $M\to C^\bullet(\alpha)$ to a limit diagram. Namely, we have to show that the map 
% \[
% \Gm(R[M]) \to \invlim\Gm(R[C^\bullet(\alpha)])
% \]
% is an isomorphism. Since $\Gm$ preserves limits, it suffices to see that
% \[
% R[M]\to \invlim R[C^\bullet(\alpha)] \qin \calg(\Sp)
% \]
% is an isomorphism.
% The cosimplicial object $R[C^\bullet(\alpha)]$ identifies with the Amitsur complex of the 
% map $R[M]\to R[N]$.  But $R[M] \to R[N]$ admits a retraction as a map of $R[M]$-modules, so it is descendable in the sense of \cite[Definition 3.18]{AkhilGalois}, and so the conclusion follows from \cite[Proposition 3.20]{AkhilGalois}. 

% In the $p$-complete settings, we similarly reduced to show that the Amitsur complex of the map $R[M]^\wedge_p\to R[N]^\wedge_p$, formed in $\calg(\Sp^\wedge_p)$, has $R[M]^\wedge_p$ as its limit. This follows similarly to the non-complete case from the fact that $R[M]/p$ is a retract of $R[N]/p$ in the category of $R[M]$-modules. 
\end{proof}

\subsection{Rigidity}\label{sub:rigidity}

We begin by constructing the inclusion of the tautological units.

\begin{defn}[Constant group scheme {\cite[Exm. 2.5.7]{Lurie_Ell2}}]\label{defn:const}
    For $M\in \Mod^{\cn}_\Z$, denote by 
    \[
    \ct{M}:\CAlg \to \Mod^{\cn}_{\Z}
    \]
    the Zariski sheafification of the constant functor valued at $M$.
\end{defn}

\begin{rem}\label{rem:const}
    For an abelian group $M$ and commutative ring spectrum $R$, we have
\[
\ct{M}(R) = \{ \text{Locally constant functions } f: |\Spec (R)| \to M\}.
\]
Moreover, when $M$ is finite, the functor $\ct{M}$ is corepresented by the commutative ring spectrum $\Sph^M$.  %This follows for instance by combining \cite[Cns. 1.6.10]{Lurie_Ell2} with the theory of Cartier duality in \cite{Ell1}.
\end{rem}

\begin{cnstr}[Tautological units]\label{cnstr:theta}
    A consequence of \Cref{defn:const} is that for any Zariski sheaf $F$ on $\CAlg$, we have
    \[
    \mathrm{Nat}(\ct{M}(-), F) = \Map_{\Mod_\Z^{\cn}}(M, F(\Sph)),
    \]
    since the constant functor construction is adjoint to evaluation at the initial object.  
    Applying this to the Zariski sheaf $F=\Gmred(-[M])$ (\Cref{prop:red_units_sheaf_affine}), we define
    \[
    \mdef{\Theta_M}: \ct{M}(-) \to \Gmred(-[M])
    \]
    to be the unique natural transformation whose value on $\Sph$ agrees with the unit map $M \to \Gmred(\Sph[M])$ of the adjunction of \Cref{rem:augmented}.  
\end{cnstr}

The bulk of this paper is dedicated to analyzing the pairs $(R \in \CAlg, M \in \Ab)$ for which $\Theta_M(R)$ is an isomorphism.

\begin{defn}\label{def:rigidity}
Let $M$ be an abelian group and let $R$ be a commutative ring spectrum.  
\begin{enumerate}
    \item We say that $R$ is \tdef{$M$-rigid} if the map $\Theta_M(R)\colon \ct{M}(R)\to \Gmr{R}{M}$ is an isomorphism.
    \item  We say that $R$ is $\tdef{$p$-rigid}$ if it is $M$-rigid for every finite abelian $p$-group $M$. 
\end{enumerate}
\end{defn}

\begin{var}\label{var:p-complete_rigid}
If $R$ is additionally $p$-complete, we say $R$ is $\tdef{$p$-completely $M$-rigid or $p$-rigid}$ if the analogous conditions hold for the composite $\widehat{\Theta}_M(R)\colon \ct{M}(R)\xrightarrow{\Theta} \Gmr{R}{M} \to \Gmrp{R}{M}.$  
% .  Then 
% \begin{enumerate}
%     \item We say that $R$ is \tdef{$p$-completely $M$-rigid} if the composite $\widehat{\Theta}_M(R)\colon \ct{M}(R)\xrightarrow{\Theta} \Gmr{R}{M} \to \Gmrp{R}{M} $ is an isomorphism.
%     \item  We say that $R$ is $\tdef{$p$-completely $p$-rigid}$ if it is $p$-completely $M$-rigid for every finite abelian $p$-group $M$. 
% \end{enumerate}
Note that for $M$ finite, $R[M]\simeq R[M]^{\wedge}_p$ so this coincides with the above definition.
\end{var}

%Rigidity is not a completely trivial notion already in the context of classical commutative rings. 

\begin{exm}[see, e.g., {{\cite[Theorem 1]{karpilovsky1983finite}}}]\label{ex:rigidity_Z_discrete}
Let $R$ be a discrete commutative ring and let $\Lambda$ be a finitely generated free abelian group, i.e. $\Lambda\simeq \ZZ^k$. A commutative ring $R$ is $\Lambda$-rigid if and only if $R$ is \emph{reduced}, that is, contains no non-zero nilpotent elements. More generally, for every connected commutative ring $R$, the units of $R[\Lambda]$ are the products $r\cdot \lambda_0\cdot f$ where $r\in R^\times$, $\lambda_0\in \Lambda$, and $f\in R[\Lambda]$ is congruent to $1$ modulo the nil-radical of $R$.
\end{exm}

We next discuss closure properties of rigidity with respect to $M$ and $R$.
First, 
since $\Theta_M$ is a morphism of Zariski sheaves, the collection of $M$-rigid commutative ring spectra $R$ is always closed under Zariski covers. For \emph{finite} $M$, we have a much stronger closure property.

\begin{prop} \label{good_rings_closed_limits}
Let $M$ be a finite abelian group. Then the collection of $M$-rigid (resp. $p$-rigid) commutative ring spectra is closed under all limits in  $\calg(\Sp)$. %The same holds for the collection of $p$-rigid commutative ring spectra. 
\end{prop}

\begin{proof}
The second statement clearly follows from the first.  For fixed finite $M$, it suffices to note that $\Gmr{-}{M}$ and $\ct{M}(-)$ are representable (\Cref{prop:red_units_sheaf_affine}, \Cref{rem:const}).
\end{proof}

\subsection{Staticity}\label{sub:staticity}

Consider the forgetful map
\begin{equation}\label{eqn:static}
\G_m(R) \to \G_m(\pi_0(R)) = \pi_0(R)^{\times}.
\end{equation}
One difficulty in computing strict units is that $\G_m(R)$ may not be discrete, and so while it can be relatively easy to see which units in $\pi_0(R)$ lift to strict units (as we saw in the introduction), there may be a nontrivial space of ways in which they lift.  In this paper, we are interested in when this does \emph{not} happen -- that is, when (\ref{eqn:static}) is just an inclusion of connected components.  We call this phenomenon \emph{staticity}.  
%In general, the spectrum $\G_m(R)$ may not be discrete, and therefore a unit may be a strict unit 
%one standard abuse of notation 
%one difficulty in computing strict units 

%For a commutative ring spectrum $R$, the tautological units $\ct{M}(R)$ embed as a (discrete) subgroup of the group of units $\Gm(\pi_0R)$. We refer to this feature as the \emph{staticity} of the tautological units. Thus, if $R$ is $M$-rigid then the reduced units are all tautological and hence static. Our goal now is to study this staticity phenomenon, its closure properties and its relation with rigidity for reduced strict units.
%We begin with the definition of staticity for an arbitrary Zariski (pre)sheaf.\atodo{First sentence currently doesn't make much sense}

\begin{defn} \label{def:staticity}
Suppose $F:\CAlg^{\cn} \to \Spaces$ and $R\in \CAlg^{\cn}$.  Then we say $R$ is \tdef{$F$-static} (or symmetrically $F$ is $R$-static) if the map $F(R) \to F(\pi_0(R))$ is an inclusion of spaces.\footnote{Note that here, an inclusion (of spaces) means a map which is the inclusion of a union of connected components, or equivalently, a $(-1)$-truncated map.}  For functors $G$ taking values in spectra, we say $R$ is $G$-static if it is $\Omega^{\infty}G$-static.  
%is  presheaf of spaces over $R\in \calg(\Sp^\cn)$ (or, more generally, a presheaf) and let $S$ be a connective commutative $R$-algebra. We say that $S$ is \tdef{$F$-static} if the map $F(S)\to F(\pi_0(S))$ is an inclusion of spaces. Symmetrically, we shall also say in this case that $F$ is $R$-static. %or that the pair $(R,F)$ is static.   
\end{defn}

\begin{exm}\label{prop:const_static}
Let $M$ be an abelian group. Then every connective commutative ring spectrum $R$ is $\ct{M}$-static. 
\end{exm}

%Note that here, inclusion (of spaces) means a subspace which is a union of connected components, or equivalently, a $(-1)$-truncated map. For an abelian group sheaf $\GG$, we will say that $S$ is $\GG$-static if it is $\Omega^\infty\GG$-static.%, that is, if $\GG(S)\to \GG(\pi_0(S))$ is an inclusion of connective $\ZZ$-module spectra. 

%We will generally be interested in functors which are in the category $\ZShv

%$R$-static sheaves form a well-behaved full subcategory of all sheaves. 

Staticity exhibits closure with respect to many operations.  

\begin{lem}\label{prop:static_limits_groups}
Let $R$ be a connective commutative ring spectrum. %The collection of $R$-static functors is closed under limits and extensions in $\Fun(\CAlg, \Spaces)$.  That is:
\begin{enumerate}
    \item If $\{F_i\}_{i\in I}$ is a diagram of functors in $\Fun(\CAlg, \Spaces)$ and $F_i$ is $R$-static for every $i$, then $\invlim F_i$ is $R$-static.
    \item If $
F_0 \to F_1 \to F_2 $
is a fiber sequence in $\Fun(\CAlg,\Sp^{\cn})$ with $F_0$ and $F_2$ both $R$-static, then $F_1$ is $R$-static. 
\end{enumerate} 
\end{lem}
% \begin{prop}\label{prop:static_limits_groups}
% Let $R$ be a connective commutative ring spectrum. The collection of $R$-static sheaves is closed under limits and extensions in $\ZShv(\Sph)$.  That is:
% \begin{enumerate}
%     \item If $\{\GG_i\}_{i\in I}$ is a diagram of sheaves and $\GG_i$ is $R$-static for every $i$, then $\invlim\GG_i$ is $R$-static.
%     \item If 
% \[
% \GG_0 \to \GG_1 \to \GG_2
% \]
% is a fiber sequence in $\ZShv(\Sph)$ with $\GG_0$ and $\GG_2$ both $R$-static, then $\GG_1$ is $R$-static. 
% \end{enumerate} 
% \end{prop}

\begin{proof}
The claims follow by using that limits of functors are computed pointwise
 and $(-1)$-truncated maps are closed under limits and passage to total spaces of principal fibrations in $\Spc$. 
\end{proof}

The behavior with respect to the rings is more subtle, but at least we have the following:

\begin{lem}\label{prop:static_limits_rings}
Let $F: \CAlg \to \Spaces$ be representable (equivalently, limit preserving) and let 
$\{R_i\}_{i\in I}$ be a diagram of connective commutative ring spectra. If each of the $R_i$ is $F$-static and the coassembly map 
\[
\pi_0\invlim R_i \to\invlim\pi_0 R_i
\]
is an isomorphism\footnote{Here, on the right hand side, the limit is computed in $\CAlg(\Sp)$, i.e., it is the ``derived'' limit of a diagram of discrete commutative rings.}, then $\invlim R_i$ is $F$-static. 
\end{lem}

\begin{proof}
This follows again because $(-1)$-truncated maps are closed under limits.
% We have to show that the map 
% \[
% \GG(\invlim R_i) \to \GG(\pi_0\invlim R_i) \simeq \GG(\invlim \pi_0 R_i) \qin \Mod_\ZZ^\cn
% \]
% is an inclusion.
% Since $\GG$ is an abelian group scheme and hence a limit preserving functor, we can identify this map with the limit of the maps $\GG(R_i) \to \GG(\pi_0 R_i)$. Since each of these maps is an inclusion of connective $\ZZ$-modules by assumption, their limit is an inclusion as well and the result follows. 
\end{proof}

% \begin{prop}\label{prop:static_limits_rings}
% Let $\GG$ be an abelian group scheme and let 
% $\{R_i\}_{i\in I}$ be a diagram of connective commutative ring spectra. If each of the $R_i$ is $\GG$-static and the coassembly map 
% \[
% \pi_0\invlim R_i \to\invlim\pi_0 R_i
% \]
% is an isomorphism\footnote{Here, on the right hand side, the limit is computed in $\CAlg(\Sp)$, i.e., it is the ``derived'' limit of a diagram of discrete commutative rings.}, then $\invlim R_i$ is $\GG$-static. 
% \end{prop}

% \begin{proof}
% We have to show that the map 
% \[
% \GG(\invlim R_i) \to \GG(\pi_0\invlim R_i) \simeq \GG(\invlim \pi_0 R_i) \qin \Mod_\ZZ^\cn
% \]
% is an inclusion.
% Since $\GG$ is an abelian group scheme and hence a limit preserving functor, we can identify this map with the limit of the maps $\GG(R_i) \to \GG(\pi_0 R_i)$. Since each of these maps is an inclusion of connective $\ZZ$-modules by assumption, their limit is an inclusion as well and the result follows. 
% \end{proof}

% \begin{proof}
% This follows from the fact that the map $R\to \pi_0(R)$ induces a homeomorphism $\Spec(R)\simeq \Spec(\pi_0R)$. 
% \end{proof}
We now study staticity in a few special situations.

\subsubsection{Staticity is automatic away from the residual characteristic}
%The next example shows that staticity is automatic ``away from the residual characteristic.'' 
\begin{prop} \label{prop:p-invert_static}
Let $R$ be a $p$-complete commutative ring spectrum and let $\GG$ be an affine abelian group scheme such that the multiplication by $p$ map $ \GG \xrightarrow{p} \GG$ is invertible. Then the map $\GG(R) \to \GG(\pi_0R)$ is an isomorphism, and in particular $R$ is $\GG$-static. 
\end{prop}

\begin{proof}
%The proof is a standard obstruction theory argument, using the fact that the cotangent fiber of $\GG$ at its origin is a $\ZZ[\inv{p}]$-module. 
%We have to show that the map $\GG(R) \to \GG(\pi_0(R))$ is an isomorphism. 
Consider the Postnikov tower
\[
R\iso \invlim \tau_{\le m} R  \to \dots\to \tau_{\le 2} R \to \tau_{\le 1} R \to \tau_{\le 0} R = \pi_0(R).
\]
Since $\GG$ is an affine abelian group scheme, it is a limit preserving functor, and hence 
\[
\GG(R)\iso \invlim \GG(\tau_{\le m}R) \qin \Mod_\ZZ^\cn.
\]
It therefore suffices to show that the maps $\tau_{\le m+1} R \to \tau_{\le m}R$ become isomorphisms after applying  $\GG$. Each of these maps fits into a pullback square 
\[
\xymatrix{
\tau_{\le m+1} R\ar[d] \ar[r] & \tau_{\le m} R\ar[d] \\  
\pi_0 (R) \ar[r] & \pi_0(R) \oplus \Sigma^{m+1}\pi_{m+1} (R) \\  
}
\]
where the bottom right corner is the split square-zero extension of $\pi_0(R)$ by $\Sigma^{m+2}\pi_{m+1} R$. Thus, it suffices to show that the map $\pi_0 R \to \pi_0R \oplus \Sigma^{m+2}\pi_{m+1} R$ induces an isomorphism upon applying $\GG$. Since this map has a left inverse given by the canonical augmentation $\pi_0(R) \oplus \Sigma^{m+2}\pi_{m+1} (R)\to \pi_0(R)$, it remains to show that the fiber of the map 
\[
\GG(\Sigma^{m+2}\pi_{m+1} (R)\oplus \pi_0(R)) \to \GG(\pi_0(R))
\]
vanishes. Let $L$ be the cotangent fiber of $\GG$ at its identity element. Then, by definition, we have 
\[
\fib(\GG(\Sigma^{m+2}\pi_{m+1} (R)\oplus \pi_0(R)) \to \GG(\pi_0(R))) \simeq \hom_{\Sp^\cn}(L,\Sigma^{m+2} \pi_{m+1}(R))
\]
Since $p$ is invertible on $\GG$, it also acts invertibly on $L$. Since $R$ is $p$-complete, so is $\Sigma^{m+2} \pi_{m+1}(R)$. We conclude that 
\[
\hom_{\Sp^\cn}(L,\Sigma^{m+2} \pi_{m+1}(R)) \simeq 0
\] 
and the result follows.
\end{proof}

%\begin{exm}
%If $R$ is $p$-complete and the Tate-valued Frobenius $R \to R^{tC_p}$ is an isomorphism, then $R$ is $\G_m$-static. Indeed, it follows from \cite[???]{CarmeliStrict} that $\G_m(R)\simeq \G_m[\ZZ[\frac{1}{p}]](R)$, on which $p$ acts invertibly. Hence, it is $R$-static  \Cref{prop:p-invert_static}. 
%\end{exm}

\subsubsection{$\Gmr{-}{M}$-staticity}
We now specialize to functors of the form $\Gmr{-}{M}$. 
The staticity of these functors is compatible with direct sums in $M$, in the following sense. 
\begin{prop}\label{prop:staticity_direct_sum}
Let $M$ and $N$ be abelian groups and let $R$ be a commutative ring spectrum. If $R$ is $\Gmr{-}{M}$-static and $R[M]$ is $\Gmr{-}{N}$-static, then $R$ is $\Gmr{-}{M\oplus N}$-static. Similarly, if $R=R^{\wedge}_p$ is $\Gmrp{-}{M}$-static and $R[M]^\wedge_p$ is $\Gmrp{-}{N}$-static, then $R$ is $\Gmrp{-}{M\oplus N}$-static.   
\end{prop}

\begin{proof}
We prove the non $p$-complete version, as the $p$-complete one is analogous.
By \Cref{prop:red_units_dsum} we can identify the morphism
\[
\Gmr{R}{M\oplus N} \to \Gmr{\pi_0R}{M\oplus N} 
\]
with the direct sum of the morphisms
\begin{align*}
\Gmr{R}{M} &\to \Gmr{\pi_0R}{M} && \mathrm{ and } && \Gmr{R[M]}{N} \to \Gmr{\pi_0R[M]}{N},
\end{align*}
where in the latter we regard $R[M][N]\simeq R[M\oplus N]$ as an augmented $R[M]$-algebra. 
The result now follows from the fact that inclusions of connective spectra are closed under finite direct sums. 
\end{proof}

Next, we reformulate the notion of rigidity in terms of staticity with respect to $\Gmr{-}{M}$.
\begin{prop}
\label{prop:static_rigid}
Let $M$ be an abelian group and let $R$ be a 
commutative ring spectrum. Then, $R$ is $M$-rigid if and only if it satisfies the following two conditions:
\begin{enumerate} 
\item It is $\Gmr{-}{M}$-static, that is, the map $\Gmr{R}{M} \to \Gmr{\pi_0R}{M}$ is an inclusion, and so we can regard the left-hand side as a subgroup of the right-hand side.  
\item The subgroup 
\[
\Gmr{R}{M}\subseteq \Gmr{\pi_0R}{M}
\] 
is contained in the image of the map 
\[
\Theta_{\pi_0R}\colon \ct{M}(\pi_0R) \to \Gmr{\pi_0R}{M}. 
\]
In other words, the reduced strict units of $R[M]$ map to tautological units of $\pi_0(R[M])$.
%that is, the reduced strict units of $R$ map to tautological units of $\pi_0R$.
\end{enumerate}
Similarly, if $R$ is $p$-complete, then $R$ is $p$-completely $M$-rigid if and only if it is $\Gmrp{-}{M}$-static and the reduced strict units of $R[M]^\wedge_p$ map to  tautological units of $\pi_0(R[M]^\wedge_p)$.
\end{prop}
\begin{proof}
We show the non-$p$-complete case, as the $p$-complete one is proved identically. Consider the commutative square
\[
\xymatrix{
\ct{M}(R)\ar^-{\Theta_{R}}[r]\ar^\wr[d] & \Gmr{R}{M}\ar[d]\\ 
\ct{M}(\pi_0R)\ar@{^{(}->}^-{\Theta_{\pi_0R}}[r] & \Gmr{\pi_0R}{M}
}
\]
arising from the naturality of $\Theta$. In this square, the left vertical map is an isomorphism by \Cref{prop:const_static}, and the lower horizontal map is an inclusion of abelian groups.

Now, assume that $R$ is $M$-rigid. Then, the upper horizontal map $\Theta_R$ is an isomorphism, and hence the right vertical map is an inclusion, with image contained in that of $\Theta_{\pi_0 R}$, as we wanted to show. 

Conversely, if $R$ satisfies condition $(1)$ then, since inclusions of spaces satisfy cancellation from the right, this forces the upper vertical map to be an inclusion as well. Now, by the commutativity and the fact that the left vertical map is an isomorphism, condition $(2)$ forces it to be also surjective on connected components. Together, this shows that the upper horizontal map $\Theta_R$ is an isomorphism, so that $R$ is $M$-rigid.
\end{proof}

%% file: rigidity.tex
%\atodo{We could remark on how functoriality is important, or maybe this is already done in the intro.}
In this section, we will show that many commutative ring spectra are $p$-rigid (i.e., rigid for every finite abelian $p$-group).
%Let $M$ be a finite abelian $p$-group. 
%The aim of this section is to show that many commutative ring spectra are $p$-rigid in the sense of \Cref{def:rigidity}.  
%In this section we shall show that a large class of $p$-complete commutative ring spectra are $M$-rigid.  
%The group rings $R[M]$ have tautological strict units, which can be thought of as the monomials $[m]\in R[M]^\times$ for $m\in M$. Since these strict units map to $1\in \G_m(R)$ under the canonical augmentation $\varepsilon\colon R[M]\to R$, they belong to the reduced units $\Gmr{R}{M}$. 
%This construction assembles to a map 
%\begin{equation}\label{eqn:rig1}\tag{*}
%\theta:  M\to \Gmr{R}{M}
%\end{equation}
%(see Section \ref{sub:theta} for a rigorous construction) which one could hope is an isomorphism.  This is false in general for several reasons:
%Supposing that $R$ is discrete for the moment\scomment{I think we should avoid saying its discrete, as we emphesis that this is a chromatic height issue (which is correct)}, this is false in general for several reasons: 
\begin{exm} \label{exm:rigfail} Consider the following failures of $p$-rigidity for discrete $R$:%Not all commutative ring spectra are $p$-rigid:
\begin{enumerate}
    
%    \item If the Zariski spectrum of $R$ is not connected (i.e., if $R$ is a product of two rings), then we can form strict units of $R[M]$ which are different elements of $M$ at different connected components of $\Spec(R)$. Those units are not in the image of $M$.  
    \item  If $M$ is a finite group of exponent $p^e$ and $R$ is a $\Z[1/p]$-algebra with a primitive $p^e$-th root of unity, then there is a discrete Fourier isomorphism $R[M] \cong R^{M^*}$ of algebras.  This gives an isomorphism $\Gmr{R}{M} \simeq \G_m(R)^{M^*\setminus \{0\}}$, which is usually larger than $M$.
    %If $M$ is a finite group and $R$ admits a primitive root of unity of height $0$ in the sense of  \cite{CSYCycl},\footnote{We note that this forces the order of the root to be invertible in $R$, which is not the case if in an ordinary ring one defines a primitive root to be a zero of the polynomial $1+t+...+t^{p-1}$.} of order divisible by the exponent of $M$, then there is an isomorphism of algebras $R[M] \cong R^{M^*}$ with the algebra of functions on the Pontryagin dual $M^*$, known as the discrete Fourier transform.  Under this isomorphism, the augmentation of $R[M]$ corresponds to the evaluation map $R^{M^*} \to R$ at $0\in M^*$, so in this case $\Gmr{R}{M} \simeq \G_m(R)^{M^*\setminus \{0\}}$, which is usually larger than $M$.
    \item $p$-rigidity usually fails when $R$ is an $\FF_p$-algebra. For example, $\FF_p[C_p] \simeq \FF_p[t]/t^p$, and the (automatically strict) reduced units in this ring contain all the elements of the form $1 + a_1t + \dots + a_{p-1}t^{p-1}$ for arbitrary $a_i \in \FF_p$.
\end{enumerate}
\end{exm}

%We can address these difficulties as follows:
One can view these classical obstructions through the lens of chromatic homotopy theory: 

\begin{itemize}
%\item $(1)$ has to do with the connected components of the Zariski spectrum of $R$: while  the source $M$ of our map is independent of $R$ and in particular insensitive to these components, the functor $\Gmr{-}{M} \colon \calg \to \Mod_\ZZ^\cn$ is product preserving and hence splits according to the components of $\Spec(R)$ -- in fact, it is a \emph{Zariski sheaf}.  To solve this difficulty we will replace $M$ by by its Zariski sheafification $\ct{M}$. 

\item $(1)$ originates from the existence of primitive $p$-th roots of unity, a phenomenon that occurs only when $p$ is invertible (i.e., in chromatic height $0$), cf. \cite{Sanath}. We shall thus restrict our attention to $p$-complete commutative ring spectra. 

\item $(2)$ is related to $\FF_p$-linearity: any $\FF_p$-algebra has vanishing $T(n)$-localizations for $n\geq 0$ and is therefore ``concentrated in chromatic height $\infty$''.  We will therefore work with an appropriate category of commutative ring spectra with ``no height $\infty$ part.''  %As we shall see, this problem can be prevented by restricting to commutative ring spectra $R$ which are concentrated in finite chromatic heights, or a limit of such.
\end{itemize}
It turns out that the above obstructions are essentially the only ones.  Most of this section will be dedicated to showing that $T(n)$-local ring spectra are $p$-rigid (\Cref{T(n)_local_rigid}).  Since this condition is closed under limits, any iterated limit of $T(n)$-local rings will be $p$-rigid.  In particular, we will see that this holds for the $p$-complete sphere $\Sph_p$ and any algebra of finite Tor-amplitude over it, by chromatic convergence (\S\ref{sub:cconv}).

\subsection{$T(n)$-local rigidity}
%reduce to the case of having roots of unity
Throughout this section we fix a prime $p$ and  we consider the $T(n)$ localization of $p$-local spectra for $n\ge 1$, so that all spectra to which we apply this localization are implicitly assumed to be $p$-local. The aim of this section is to show that every $T(n)$-local commutative ring spectrum $R$ is $p$-rigid (\Cref{T(n)_local_rigid}), i.e. that the map 
\[
\Theta_M(R)\colon \ct{M}(R)\to \Gmr{R}{M}
\] 
is an isomorphism for $M$ any finite abelian $p$-group. 
%\begin{thm} \label{T(n)_local_rigid}
%For $n\ge 1$, every $T(n)$-local commutative ring spectrum is $p$-rigid. 
%\end{thm}

When $n=0$, we understood the $T(0)$-local (i.e. rational) failure of $p$-rigidity in \Cref{exm:rigfail}(2) using a Fourier isomorphism $R[M]\simeq R^{M^*}.$ Since the right-hand side is a product of rings, its units are easy to compute.  

%Recall from \Cref{exm:rigfail}(2) that when $n=0$, $T(0)$-local (i.e. rational) commutative ring spectra $R$ are not necessarily $p$-rigid.  To see this (when $R$ has sufficiently many roots of unity), we analyzed the units of $R[M]$ using the \emph{Fourier transform} isomorphism 
%\begin{equation}\label{eqn:t0fourier}
%R[M]\simeq R^{M^*}.
%\end{equation}
%Since the right-hand side is a product of rings, its units are easy to compute.  

In the case $n>0$, there is an analogous isomorphism known as the \emph{chromatic Fourier transform}, which identifies $R[M]$ as cochains on the classifying space $B^nM^*$.  We review this in \S\ref{sub:CFT} and use it to identify the target of
\[
\Theta_M: \ct{M} \to \Gmr{-}{M}
\]
in \S\ref{sub:thetatarget}.  More precisely, for each $r\geq 0$, we will show that if we restrict the target to $\Z/p^r$-modules $M$, it can be identified with the functor $\ct{M}$.   Then, we finish by using this to show that $\Theta_M$ induces an isomorphism in \S\ref{sub:rigfinish}.  

\begin{rem}
Let us indicate why the $n=0$ case behaves qualitatively differently.  The issue is that in the equivalence $R[M]\simeq R^{M^*}$, the right-hand side depends, as a ring, only on the number of elements of $M$ but not on the group structure. This means that $\Gmr{R}{M}$ depends only on the number of elements of $M$ whereas $\ct{M}(R)$ is clearly sensitive to the group structure of $M$ and so 
\[
\Theta_M(R): \ct{M}(R) \to \Gmr{R}{M}
\]
cannot be an isomorphism in general.  At heights $n>0$, the chromatic Fourier transform identifies $R[M]$ with $R^{B^nM^*}$, which, in contrast, \emph{does} depend on the group structure on $M^*$. % \todo{maybe say a stronger statement, like, that one can read the group structure from that space?}
\end{rem}

%In higher chromatic heights there is an analogous Fourier transform associated with higher roots of unity, constructed and studied in \cite{???}. As it turns out, while in height $0$ the Fourier transform obstructs $p$-rigidity, in positive chromatic heights is is helpful in proving it. Ultimatly, this follows from the fact that while the set $M^*$ forgets the group structure, this structure can be completely recovered from all the higher classifying spaces $B^nM^*$.

\subsubsection{Chromatic Fourier transform}\label{sub:CFT}
We start by briefly recalling the chromatic Fourier transform from \cite{CSYCycl, BCSY}.
For a $T(n)$-local commutative ring spectrum $R$, a \tdef{higher root of unity} of height $n$ and order $p^r$ in $R$ is a map 
\[
\omega \colon \Sph[\Sigma^n \Z/p^r] \to R
\]
of commutative ring spectra.  We denote the space of such higher roots of unity by $\mu_{p^r}^{(n)}(R)$.  Each $\omega \in \mu_{p^r}^{(n)}(R)$ determines a natural transformation 
\[
\mdef{\mathcal{F}_\omega(M)} \colon R[M] \to R^{\Omega^\infty\Sigma^n M^*} \qin \calg_R(\Sp_{T(n)}) 
\]
of functors of $M$, a $\pi$-finite connective $\Z/p^r$-module spectrum,
where 
\[
M^* := \hom_\Z(M,\QQ/\Z) \simeq \hom_{\Z/p^r}(M,\Z/p^r) \qin \Mod_{\Z/p^r}
\] 
denotes the Pontryagin dual of $M$.  We refer to $\mathcal{F}_\omega(M)$ as the chromatic Fourier transform.  

The natural quotient maps $\Z/p^r \to \Z/p^{r-1}$ induce maps $\mu_{p^{r-1}}^{(n)}(R)\to \mu_{p^r}^{(n)}(R)$ corresponding to the ``inclusion of $p^{r-1}$-roots of unity.''  A higher root of unity $\omega \in \mu_{p^r}^{(n)}(R)$ is called \tdef{primitive} if the base change to $\mu_{p^r}^{(n)}(R')$ for any non-zero, $T(n)$-local $R$-algebra $R'$  is not in the image of the corresponding map $\mu_{p^{r-1}}^{(n)}(R')\to \mu_{p^r}^{(n)}(R')$.  
%A higher root of unity $\omega \in \mu_{p^r}^{(n)}(R)$ is called \tdef{primitive} if the only commutative algebra map $f\colon R\to S$ for which $f(\omega) \in \mu_{p^r}^{(n)}(S)$ is of order $p^{r-1}$, is the zero algebra. \todo{Is there a simpler/more internal way, like not in the image of $n-1$ roots} \scomment{sure, its the same, we can just say that.}
For primitive roots, the Fourier transform is an isomorphism (i.e., one has a ``Fourier inversion''):

\begin{thm}[Chromatic Fourier Transform {{\cite[Theorem A]{BCSY}}}]\label{thm:Fourier}
Let $n\geq 0$ and let $R\in \calg(\Sp_{T(n)})$. If $\omega \in \mu_{p^r}^{(n)}(R)$ is primitive, then the Fourier transform 
\[
\mathcal{F}_\omega \colon R[M]\iso R^{\Omega^{\infty-n} M^*} \qin \calg(\Sp_{T(n)})
\]
is an isomorphism for every $\pi$-finite connective $\Z/p^r$-module spectrum $M$. 
\end{thm}

While not all commutative $T(n)$-local algebras have primitive roots of order $p^r$, they do after passing to a finite Galois extension. In fact, primitive higher roots of unity are classified by a finite Galois extension of the $T(n)$-local sphere $\Sph_{T(n)}$.

\begin{thm}[Higher Cyclotomic Extensions, {{\cite[Proposition 5.2]{CSYCycl}}}]
The functor that takes $R\in \calg(\Sp_{T(n)})$ to the subspace of $\mu_{p^r}^{(n)}(R)$ consisting of primitive roots is co-representable by a $T(n)$-local $\Z/p^r$-Galois extension $R_{n,r}^f$ of $\Sph_{T(n)}$.
\end{thm}

These Galois extensions $R_{n,r}^f$ are referred to in \cite{CSYCycl} as the \tdef{chromatic cyclotomic extensions}.  The upshot for us will be that, provided we work over $R_{n,r}^f$, there is a natural Fourier isomorphism:

\begin{cor}\label{cor:fourier}
There is a natural isomorphism of commutative and cocommutative Hopf algebras in $\Sp_{T(n)}$, defined for $R \in \CAlg_{R_{n,r}^f}(\Sp_{T(n)})$ and $M\in \Mod_{\Z/p^r}^{\cn}$ with $\pi_*M$ finite, of the form
\[
\mathcal{F} \colon R[M]\iso R^{\Omega^{\infty-n} M^*}.
\]
%of commutative and cocommutative Hopf algebras in $\Sp_{T(n)}$, natural in both $R$ and $M$.  
\end{cor}
\begin{proof}
Let $\omega \in \mu^{(n)}_{p^r}(R_{n,r}^f)$ be the universal primitive $p^r$-th root of unity of height $n$, and let $\mathcal{F} = \mathcal{F}_{\omega}$.  For the Hopf algebra structure, see \cite[\S 3.3]{BCSY}.  
\end{proof}

Implicit in \Cref{cor:fourier} is that if $M$ is $p$-power torsion and $n$-connected, then $R[M]\simeq R$.  We will need a torsion-free variant of this fact:

\begin{lem}\label{lem:Gmtrunc}
    Suppose $R$ is a $T(n)$-local commutative ring spectrum and suppose $j> n+1$. Then the unit map $R \to L_{T(n)}R[\Sigma^j \ZZ]$ is an isomorphism.  In particular, the connective $\ZZ$-module $\mathbb{G}_m(R)$ is $(n+1)$-truncated.  
\end{lem}
\begin{proof}
    The first statement follows from the formula
    \[
    L_{T(n)}R[\Sigma^j \ZZ] = L_{T(n)} \colim_r R[\Sigma^{j-1} \ZZ/p^r]
    \]
    and the aforementioned fact about $R[M]$ for $M$ $p$-power torsion and $n$-connected (cf. also \cite[Lemma 2.4.4, Proposition 3.2.1]{AmbiHeight}, for instance).  
    %\cite[]{AmbiHeight} (the essential input being the Ravenel--Wilson calculation \cite{RavenelWilson}),
    The ``in particular'' follows from considering $R$-algebra maps into $R$.
\end{proof}

For our applications, we will need one additional fact about the $R_{n,r}^f$:

\begin{lem}\label{Rnf_connected}
For $n,r\geq 0$, the commutative ring spectrum $R_{n,r}^f$ is Zariski-connected. 
\end{lem}

\begin{proof}
%For $n=0$, this is clear as $R_{n,r}^f = \Q(\zeta_{p^r})$.  

Suppose $R_{n,r}^f = R_1 \times R_2$ for nonzero rings $R_1, R_2$.  Let $E_n$ denote a Lubin-Tate theory of height $n$ and consider $L_{K(n)}(R_{n,r}^f \otimes E_n)$ with its action of the Morava stabilizer group $\mathbb{G}_n$. By \cite[Propositions 5.2, 5.4, 5.5]{CSYCycl}, this has Zariski spectrum isomorphic to $(\Z/p^r)^{\times}$, acted on transitively by $\mathbb{G}_n$ through the ``higher cyclotomic character.''

On the other hand, by the Devinatz--Hopkins--Smith Nilpotence theorem (cf. \cite[Corollary 5.1.17]{TeleAmbi}), the functor $L_{K(n)}(-\otimes E_n) : \Sp_{T(n)} \to \Sp_{T(n)}$ is nil-conservative, i.e., it does not send any rings to zero.  Thus, $L_{K(n)}(R_{n,r}^f \otimes R_i)$ is nonzero for $i=1,2$, and so the Zariski spectrum of $L_{K(n)}(R_{n,r}^f \otimes E_n)$ splits into at least two nontrivial $\mathbb{G}_n$-orbits, which is a contradiction.  
\end{proof}

\subsubsection{Identifying the Target of $\Theta$}\label{sub:thetatarget}

Recall that our goal is to show that $\Theta_M(R)$ is an isomorphism for every $R\in \calg(\Sp_{T(n)})$ and finite abelian $p$-group $M$.
In this section, we use the chromatic Fourier transform to abstractly identify $\Gmr{R}{M}$ with $\ct{M}(R)$ under the additional assumption that $R$ admits a primitive $p^r$-th root of unity of height $n$ and $M$ is $p^r$-torsion.  
Later, we will show that the existence of such natural isomorphisms forces every $T(n)$-local commutative ring spectrum to be $p$-rigid.

%identify the RHS of theta
\begin{prop}\label{prop:rhs} 
There is an isomorphism of functors 
\[
\Psi_M: \Gmr{-}{M} \iso \ct{M} \qin \ZShv(R_{n,r}^f;\Mod_{\Z}^\cn)
\]
which is natural in the finite discrete $\Z/p^r$-module $M$. 
In particular, for every commutative $R_{n,r}^f$-algebra $R$ and every such $M$, there is an isomorphism
\[
\Gmr{R}{M}\simeq \ct{M}(R).
\]
% Let $\eta_{n,r} \colon \Sph \to R_{n,r}^f$ be the unit of the $r$-th cylcotimic extension $R_{n,r}^f$ of $\Sph_{T(n)}$. 
% There is an equivalence
% \[
% \Psi_M\colon (\eta_{n,r})_! \AG{M}\iso \ct{M} \qin \ZShv(R_{n,r}^f;\Mod_{\Z/p^r}^\cn)
% \] 
% which is natural in the finite discrete $\Z/p^r$-module $M$. 
% In particular, for every commutative $R_{n,r}^f$-algebra $R$ and every such $M$, there is an isomorphism
% \[
% \Gmr{R}{M}\simeq \ct{M}(R).
% \]
\end{prop}

\begin{proof}
%We proceed by manipulating the left hand side using the Fourier transform. 
%Let $\omega \in \mu_{p^r}^{(n)}(R_{n,r}^f)$ be the universal primitive root.  By \Cref{thm:Fourier}(2), for $R\in \calg_{R_{n,r}^f}(\Sp_{T(n)})$ and $M$ a finite discrete $\Z/p^r$-module, we have a natural equivalence of $T(n)$-local commutative and cocommutative Hopf algebras
By \Cref{cor:fourier}, we have for $R \in \calg_{R_{n,r}^f}(\Sp_{T(n)})$ and $M$ a finite discrete $\Z/p^r$-module a natural equivalence of $T(n)$-local commutative and cocommutative Hopf algebras
\[
\mathcal{F} \colon R[M] \iso R^{\Omega^{\infty - n}M^*} \simeq R^{B^n M^*}.
\]
This yields natural equivalences of connective spectra: 
\begin{align*}
\Gmr{R}{M} &:=   \fib( \G_m(R[M]) \to \G_m(R)) \\
&\simeq \fib( \G_m(R^{B^nM^*}) \to \G_m(R)) \\
&\simeq \fib( \G_m(R)^{B^nM^*} \to \G_m(R))\\
&\simeq \Map_{\Spc_*}( B^nM^*, \G_m(R)),
\end{align*}
where the last space is the space of pointed maps $B^nM^* \to \G_m(R)$.  
%where we regard $\Hom_{\CAlg}(\bS[\Z],R)$ as having distinguished point the composite $\bS[\Z] \to \bS \to R$ of the canonical augmentation with the unit and $\Map_*$ denotes maps of pointed spaces.  
Now, $\G_m(R)$ is a connective, $(n+1)$-truncated $\ZZ$-module by \Cref{lem:Gmtrunc}, and hence there is a natural equivalence 
\[
\Map_{\Spc_*}( B^nM^*, \G_m(R)) \simeq \Map_{\Mod_\ZZ}( \tau_{\le n + 1}\widetilde{C}_*(B^nM^*;\ZZ), \G_m(R)),
\]
where $\widetilde{C}_*(B^nM^*;\ZZ)$ is the $\ZZ$-module of \emph{reduced} integral homology chains of $B^nM^*$ (that is, $\ZZ \otimes \Sigma^{\infty}B^nM^*$). Finally, by a classical computation of the homology of Eilenberg--MacLane spaces, we see that the counit map 
\[
\widetilde{C}_*(B^nM^*;\ZZ) \to \Sigma^nM^* \qin \Mod_\ZZ^\cn
\]
induces an isomorphism in homology up to degree $n+1$, and hence an isomorphism of the truncations:
\[
\tau_{\le n+1}\widetilde{C}_*(B^nM^*;\ZZ) \simeq \Sigma^nM^* \qin \Mod_{\Z}^\cn. 
\]
Combining the above, we obtain a natural equivalence 
\[
\Gmr{R}{M} \simeq \Hom_{\Mod_\ZZ^\cn}(\Sigma^n M^*,\G_m(R)) = \Hom_{\calg_{R_{n,r}^f}(\Sp_{T(n)})}(R_{n,r}^f[\Sigma^nM^*], R). %:= \G_m[\Sigma^n M^*](R). 
\]

Finally, the (inverse) Fourier transform gives an isomorphism 
\[
\mathcal{F} \colon R_{n,r}^f[\Sigma^n M^*]\iso (R_{n,r}^f)^{\Omega^\infty M}
\]
of commutative and cocommutative Hopf algebras in $\Sp_{T(n)}$ and so we have an equivalence
\[
\Gmr{-}{M} \simeq \ct{M} \qin \ZShv(R_{n,r}^f;\Sp^\cn)
\] 
of Zariski sheaves of connective spectra over $R_{n,r}^f$. Since $\ct{M}$ (and hence the left-hand side) is discrete and $p^r$-torsion, this isomorphism admits a unique $\ZZ/p^r$-linear structure, and the result follows. 

%Continuing, we have:
%\begin{align*}
%\Map_*(K(B^*, n), \Hom_{\CAlg}(\bS[\Z], R)) &\simeq \Map_*( K(B^*, n), \Hom_{\Sp}(\Z, \gl1(R)))\\
%&\simeq \Hom_{\Sp}(\Z \otimes \Sigma^{\infty}K(B^*, n), \gl1(R)).
%\end{align*}
%By \atodo{Insert citation}, we have that for any $p$-torsion finitely generated abelian group $H$ and integer $k>n$, the space $\Hom_{\Sp}(\Sigma^k H, \gl1(R))$ is contractible.  Therefore, the $\Z$-module in question is identified with
%\begin{align*}
%\Hom_{\Sp}(\tau_{\leq n} \Z \otimes \Sigma^{\infty} K(B^*,n), \gl1(R)) &\simeq \Hom_{\Sp}(\Sigma^n B^*, \gl1(R)) \\
%&\simeq \Hom_{\CAlg_R(\Sp_{T(n)})}(R[\Sigma^nB^*], R)\\
%&\simeq \Hom_{\CAlg_R(\Sp_{T(n)})}(R^{B}, R)\\
%&\simeq \underline{B}(R),
%\end{align*}
%where the second equivalence applies \Cref{thm:Fourier}(2) again.  
\end{proof}

\begin{rem}
The reason that we work with each finite $r$ rather than assuming $R$ has all $p^r$-th height $n$ roots of unity is that the argument in \Cref{sub:rigfinish} is by Galois descent.  While one has descent from each $R_{n,r}^f$, descent for the maximal higher cyclotomic extension $\colim_r R_{n,r}^f$ fails in heights $n\ge 2$, as shown by Burklund, Hahn, Levy, and Schlank in \cite[\S6.4]{burklund2023k}. 

On the other hand, if one is satisfied with the $K(n)$-local version of \Cref{T(n)_local_rigid} (which suffices for our applications), one can simply use descent from Lubin-Tate theory where the required Fourier transform is constructed by Hopkins and Lurie and shown to be an isomorphism for all $p$-local $\pi$-finite $M$, not just $\Z/p^r$-modules for a fixed $r$ (see \cite[Corollary 5.3.26]{AmbiKn}).  
\end{rem}

\subsubsection{$\Theta$ is a $T(n)$-local equivalence}\label{sub:rigfinish}

We would like to show that $\Theta_M: \ct{M} \to \Gmr{-}{M}$ is an isomorphism for any finite abelian $p$-group $M$.  It suffices to check this in the case when $M$ is a $\Z/p^r$-module, for every $r\geq 0$.  Naturally in a finite discrete $\Z/p^r$-module $M$, we have a diagram of $\Mod_\Z^{\cn}$-valued functors
\[
\begin{tikzcd}
\ct{M} \arrow[bend right=25, rr,  swap, "\widetilde{\Theta}_M"] \arrow[r, "\Theta_M"]& \Gmr{-}{M} \arrow[r,"\Psi_M",  "\simeq"'] & \ct{M}
\end{tikzcd}
\]
where $\widetilde{\Theta}_M$ denotes the composite $\Psi_M\circ \Theta_M$.  By \Cref{prop:rhs}, $\Psi_M$ is an isomorphism and thus showing $\Theta_M$ is an isomorphism amounts to showing that $\widetilde{\Theta}_M$ is an isomorphism. 
Note that both the source and target of $\widetilde{\Theta}_M$ are discrete and $p^r$-torsion, so we may regard $\widetilde{\Theta}_M$ as a natural morphism of sheaves of ordinary $\Z/p^r$-modules.

% Next, we show how the existence of the isomorphisms $\Psi_M(R)\colon \Gmr{R}{M}\iso \ct{M}(R)$ from \Cref{prop:rhs}  forces the corresponding maps $\Theta_M(R)\colon \ct{M}(R)\to \Gmr{R}{M}$ to be  isomorphisms.
% For $M$ a finite abelian $\Z/p^r$-module, we can compose these two maps to get an endomorphism of the constant sheaf
% \[
% \Psi_M \circ \Theta_M := \widetilde{\Theta}_{M}  : \ct{M} \to \ct{M} \qin \ZShv(R_{n,r}^f;\Mod_{\ZZ}^\cn). \]
%  Hence, instead of showing that $\Theta_M$ is an isomorphism, we may show that $\widetilde{\Theta}_M$ is. 
% Note that both the source and target are discrete and $p^r$-torsion, so we may regard $\widetilde{\Theta}_M$ as a natural morphism of sheaves of ordinary $\Z/p^r$-modules.  

\begin{prop}\label{prop:theta_iso_cyclo_p_r_torsion}
For every $r\ge 0$ and every finite $\Z/p^r$-module $M$,
the morphism 
\[
\widetilde{\Theta}_{M} \colon \ct{M} \to \ct{M} \qin \ZShv(R_{n,r}^f;\Mod_{\ZZ}^\cn)
\] 
is an equivalence.
\end{prop}
\begin{proof}
Since $\ct{(-)}$ is left adjoint to the evaluation at $R_{n,r}^f$, the natural transformation $\tilde{\Theta}_M$ corresponds to a morphism $M \to \ct{M}(R_{n,r}^f)$ of $\Z/p^r$-modules.  Since $R_{n,r}^f$ is connected by \Cref{Rnf_connected}, this can be identified with a natural transformation $\theta_M: M \to M$ (for instance given by evaluating $\tilde{\Theta}_M$ at $R_{n,r}^f$), and it suffices to show that $\theta_M$ is an isomorphism for all $M$.  

Since $\theta_M$ is a natural endomorphism of the identity functor on the category of finite discrete $\Z/p^r$-modules, it is given by multiplication by some scalar $a\in \Z/p^r$.  It suffices to show that $a$ is invertible.  For this, one only needs to show that $a$ is invertible mod $p$, or equivalently, that $\theta_{C_p}: C_p \to C_p$ is invertible.  But $\theta_{C_p}$ is a $\Z/p$-module map, so this amounts to showing $\theta_{C_p}$ is \emph{non-zero}.  

Up to isomorphisms, $\theta_{C_p}$ is identified with $\Theta_{C_p}(R_{n,r}^f)$, or even $\Theta_{C_p}(\bS)$.  This is non-constant by definition, so $\theta_{C_p}$ cannot be zero.

\end{proof}

%These results, taken together for all $r\ge 0$, implies that $\Theta_M(R)$ is an isomorphism for all $T(n)$-local $R$ and all $M$. 
We can now prove the main result of this section.

\begin{thm} \label{T(n)_local_rigid}
For $n\ge 1$, every $T(n)$-local commutative ring spectrum is $p$-rigid. 
\end{thm}

\begin{proof}
Let $R$ be a $T(n)$-local commutative ring spectrum and let $M$ be a finite abelian $p$-group. We wish to show that $R$ is $M$-rigid. Choose $r>0$ for which  $M$ is $p^r$-torsion and let $R\widehat{\otimes} R_{n,r}^f$ be the $T(n)$-local tensor product of $R$ and $R_{n,r}^f$.  By \Cref{prop:theta_iso_cyclo_p_r_torsion}, we have that $R\widehat{\otimes} R_{n,r}^f$ is $M$-rigid.  But the unit map $R\to R\widehat{\otimes} R_{n,r}^f$ is a faithful $(\Z/p^r)^\times$-Galois extension in $\calg(\Sp_{T(n)})$, and hence 
\[
R\simeq (R\widehat{\otimes} R_{n,r}^f)^{h(\Z/p^r)^\times} \qin \calg(\Sp_{T(n)}).
\]
Since the embedding $\calg(\Sp_{T(n)}) \to \calg(\Sp)$ is limit preserving, the same formula holds in $\calg(\Sp)$. Thus, since $p$-rigid rings are closed under limits, this implies that $R$ is $M$-rigid as well and the result follows.
\end{proof}

\subsection{Chromatic convergence and $p$-rigidity over $\Sph^{\wedge}_p$}  \label{sub:cconv}

Since $p$-rigidity is closed under limits in $\CAlg(\Sp)$, we immediately deduce from \Cref{T(n)_local_rigid} that:

\begin{thm}
    Let $\mathcal{K}\subset \Sp$ be the smallest subcategory which is closed under limits and which contains $\Sp_{T(n)}$ for any $n>0$.  Then any $R\in \CAlg(\Sp)$ whose underlying spectrum is in $\mathcal{K}$ is $p$-rigid.
\end{thm}

The class $\mathcal{K}$ contains many spectra of interest.

\begin{cor}\label{cor:p-rigid}%\label{chr_complete_p_rigid}
    Suppose $R\in \CAlg(\Sp)$ is $p$-complete and satisfies any of the following conditions:
    \begin{enumerate}
        \item $R$ is $L_n^f$-local for some $n>0$.
        \item $R/p$ is chromatically convergent (more generally, a retract of its chromatic tower).
        \item $R$ is a $\Sph_p$-algebra of finite Tor-amplitude (cf. \Cref{defn:finiteTor}).
    \end{enumerate}
    Then $R$ is $p$-rigid, i.e., there is a canonical isomorphism
    \[
\G_m(R[M]) \simeq \G_m(R) \oplus \ct{M}(R). 
\]
for any finite abelian $p$-group $M$.
\end{cor}
\begin{proof}
    Claim (1) follows by induction, starting from the case $n=1$ (note that $L_1^f$-local and $T(1)$-local are the same for $p$-complete spectra) and using the chromatic fracture squares
    \[
\xymatrix{
(L_n^fR)^{\wedge}_p\ar[r]\ar[d] & L_{T(n)} R\ar[d] \\ 
(L_{n-1}^fR)^{\wedge}_p\ar[r] &  (L_{n-1}^fL_{T(n)} R)^{\wedge}_p,
}
\]

Claim (2) follows immediately because $\lim L_n R$ is in $\mathcal{K}$ by limit closure (note that $L_n R$ is $L_n^f$-local) and the hypothesis implies $R \to \lim L_n R$ is a mod $p$ isomorphism of $p$-complete spectra and therefore an isomorphism. The statement about retracts follows because retracts are limits.  

Finally, any $R$ satisfying (3) has a finite filtration with quotients of the form $\Sigma^k (\bigoplus_{i\in I} \Sph_p)^\wedge_p$. Thus, by (2), it suffices to see that any spectrum of the form $\Sigma^k (\bigoplus_{i\in I} \Sph/p)^\wedge_p$ is chromatically convergent. 

Since $\Sph/p$ is a finite spectrum, it has finite projective $BP$-dimension in the sense of \cite[Definition 2.1]{barthel2016chromatic} (see, e.g., \cite[Proposition 2.4]{barthel2016chromatic}). Since direct sums of projective resolutions are projective resolutions, this implies that any sum of copies of $\Sph/p$ also has finite projective $BP$-dimension, and hence is chromatically complete by \cite[Theorem 3.8]{barthel2016chromatic}\footnote{One can also  argue directly using that the chromatic convergence on $\Sph$ is pro-constant on each homotopy group.}.  
\end{proof}

%% file: GmGm.tex
In the previous section, we gave conditions for the maps $\Theta_M(R)\colon \ct{M}(R)\to \Gmr{R}{M}$ to be isomorphisms in the case that $M$ is a finite abelian $p$-group --- this we called $p$-rigidity. In this section, we give analogous results when $M \simeq \ZZ^k$ is a free abelian group of finite rank, for a much more restrictive class of commutative ring spectra $R$.  In fact, our main result is a statement just for spherical Witt vectors (though we prove more specific statements along the way).  

\begin{thm}[{\Cref{thm:free-rig}}]
Let $\kappa$ be a perfect ring of characteristic $p$ and let $\Lambda$ be a free abelian group of finite rank. Then the ring of spherical Witt vectors $\Sph_{W(\kappa)}$ is $p$-completely $\Lambda$-rigid.\footnote{We work with $p$-complete rigidity (cf. \Cref{var:p-complete_rigid}) because for $\Lambda$ infinite, the spectrum $R[\Lambda]$ may not be $p$-complete even if $R$ is.}
\end{thm}

%\begin{rem}
%The ring spectra of the form $\Sph_{W(\kappa)}$ have several useful properties that we shall exploit. First, they are $p$-completely flat, and in particular of finite tor-amplitude. Moreover, they are \emph{perfect}, in the sense that the Frobinius map $\varphi_R\colon R\to R$ is an isomorphism when $R$ a ring of spherical Witt vectors. We expect these to completely characterise the rings of the form $\Sph_{W(\kappa)}$, but we shall neither use this characterisation nor prove it. 
%\end{rem}

Proving  rigidity for free abelian groups  presents new difficulties: first of all, the functor $\Gmr{-}{\Lambda}$ is no longer limit preserving when $\Lambda$ is infinite, and second, an analogue of the chromatic Fourier transform for infinite abelian groups is not known.  We therefore take a different approach: we access the $\ZZ$-module spectrum $\Gmr{R}{\Lambda}$ via the quotient map $\Gmr{R}{\Lambda}\to \Gmr{R}{\Lambda/p}$. Indeed, while $\Lambda$ is infinite, the group $\Lambda/p$ is a finite abelian $p$-group and therefore fits into the framework of the previous section.  We summarize our strategy as follows:

%There is an additional detail that, for $\Lambda$ infinite, the spectrum $R[\Lambda]$ may not be $p$-complete even if $R$ is.  We thus work internally to  $\Sp^\wedge_p$ and consider $p$-complete $\Lambda$-rigidity instead, see \Cref{var:p-complete_rigid}.  We can now summarize our strategy as follows: 

\begin{enumerate}
\item\label{enum:free_groups_plan_1} In \S \ref{sub:exactness}, we show that when $R$ is $p$-complete of finite Tor-amplitude, the short exact sequence $\Lambda \xrightarrow{[p]} \Lambda \to \Lambda/p$
% \[
% \xymatrix{
% 0\ar[r] & \Lambda\ar^{[p]}[r] & \Lambda\ar[r] & \Lambda/p \ar[r] & 0
% }
% \]
induces a cofiber sequence
\[
\xymatrix{
\AG{R[\Lambda]^\wedge_p}\ar^{[p]_*}[r] & \AG{R[\Lambda]^\wedge_p}\ar[r] & \AG{R[\Lambda/p]}
}.
\]
\item \label{enum:free_groups_plan_2} We then show in \S \ref{sub:frob} that the multiplication by $p$ map of the $\ZZ$-module $\AG{R[\Lambda]^\wedge_p}$ can be written as the composite of the map $[p]_*$ and the map induced by the Frobenius of $R$. Thus, when $R$ is \emph{perfect}, that is, has invertible Frobenius, the map $[p]_*$ can be replaced by the multiplication by $p$ map in the above cofiber sequence.
% \[
% \xymatrix{
% \AG{R[\Lambda]^\wedge_p}\ar^{p}[r] & \AG{R[\Lambda]^\wedge_p}\ar[r] & \Gmr{R}{\Lambda/p} 
% }.
% \]
\item \label{enum:free_groups_plan_3} Finally, in \S\ref{sub:Gmstaticity}, we use the above exact sequence and our results on $p$-rigidity to show that any perfect connective $R$ is $\Gmrp{-}{\Lambda}$-static, in the sense of \Cref{def:staticity}. In view of \Cref{prop:static_rigid}, this reduces our question to an algebraic problem, which we solve using our analysis of rank $1$ units in $\delta_p$-rings from \S \ref{sec:delta}, completing the proof.
\end{enumerate}

%The three parts of this strategy will be implemented in sections \ref{sub:exactness}, \ref{sub:frob}, and \ref{sub:staticity}, respectively.  Then, the proof of Theorem 
%\todo{complete outline here}

\subsection{Exactness of reduced strict units}\label{sub:exactness}
We begin by carrying out part (\ref{enum:free_groups_plan_1}) of our strategy. More generally,
% Let $\Lambda$ be a free abelian group of finite rank and let $R$ be a commutative ring spectrum.
% We begin by carrying out part (\ref{enum:free_groups_plan_1}) of our strategy. Namely, we show that under suitable assumptions on $R$, the null sequence
% \[
% \xymatrix{
% \AG{R[M]^\wedge_p}\ar^{[p]_*}[r] & \AG{R[M]^\wedge_p}\ar[r] & \Gmr{R}{\Lambda/p}}.
% \]
% is a cofiber sequence in $\Mod_\ZZ^\cn$. 
given an exact sequence of abelian groups 
\[
\xymatrix{
0\ar[r] & M\ar[r] & N\ar[r] & L \ar[r] & 0
},
\]
we may ask whether the induced sequence
\[
\xymatrix{
\AG{R[M]^\wedge_p}\ar[r] &\AG{R[N]^\wedge_p} \ar[r] & \AG{R[L]^\wedge_p} 
}
\]
is a (co)fiber sequence in $\Mod_\ZZ^\cn$.  This will turn out to be true under mild hypotheses involving $L$-rigidity (\Cref{prop:units_red_exact}).  We start by analyzing the special case of a \emph{split} exact sequence.  
%We shall show that this is the case under sufficient rigidity properties of $L$ (see [Ref] for a precise formulation). 
%\begin{rem}
%Note that it is crucial that only $L$-rigidity is used here: for our sequence of interest, $L=\Lambda/p$, whose rigidity is studied in the previous section, while the establishment of ($p$-complete) $\Lambda$-rigidity is the goal of this section.
%\end{rem}

%Roughly speaking, our analysis of the exactness of $\AG{R[-]^\wedge_p}$ consists of two parts. 
%First, we treat the special case where the short exact sequence \emph{splits}. This is based on the compatibility of $\Gmrp{R}{-}$ with direct sums (\Cref{prop:red_units_direct_sum}).
%Then, we use descent for $\AG{R[-]^\wedge_p}$ along injections  of abelian groups (\Cref{prop:descent_inclusion_groups}) to reduce the case of a general short exact sequence to the split case.
%We begin with the split case: 
\begin{prop} \label{prop:red_units_direct_sum}
Let $R$ be a commutative ring spectrum and let $M$ and $N$ be abelian groups. If $R[M]^\wedge_p$ is $p$-completely $N$-rigid and the map $R\to R[M]^\wedge_p$ induces a bijection $\pi_0\Spec(R[M]^\wedge_p)\simeq \pi_0\Spec(R)$, then the assembly map
\[
\AG{R[M]^\wedge_p}\oplus\AG{R[N]^\wedge_p} \to  \AG{R[M\oplus N]^\wedge_p}
\]
is an isomorphism. Equivalently, the split exact sequence 
\[
\xymatrix{
0\ar[r] & M\ar[r] & M\oplus N\ar[r] & N \ar[r] & 0
},
\]
is taken by $\AG{R[-]^\wedge_p}$ to a split cofiber sequence.
\end{prop}

\begin{proof}
By \Cref{prop:red_units_dsum}, the assembly map is identified with the map
\[
\AG{R[M]^\wedge_p}\oplus\AG{R[N]^\wedge_p} \to \AG{R[M]^\wedge_p}\oplus \AG{R[M][N]^\wedge_p}
\]
given by the sum of the identity of $\AG{R[M]^{\wedge}_p}$ and the map
\begin{equation}\label{eqn:agn}\tag{$*$}
\AG{R[N]^\wedge_p}\to \AG{R[M][N]^\wedge_p}
\end{equation}
induced by applying $\AG{(-)[N]^\wedge_p}$ to the unit map $R\to R[M]^\wedge_p$.  It therefore suffices to show that (\ref{eqn:agn}) is an isomorphism.  
Since $R$ is a retract of $R[M]^\wedge_p$, the assumption that $R[M]^\wedge_p$ is $p$-completely $N$-rigid implies that $R$ is $p$-completely $N$-rigid as well.  Thus, we can identify (\ref{eqn:agn}) 
with the corresponding map
\[
\ct{N}(R)\to \ct{N}(R[M]^\wedge_p).
\] 
But this is an isomorphism by the assumption that $\pi_0\Spec(R[M]^\wedge_p)\simeq \pi_0\Spec(R)$.
\end{proof}

Consider now a general exact sequence
\[
\xymatrix{
0\ar[r] & M\ar^\alpha[r] & N\ar^\beta[r] & L \ar[r] & 0
}. 
\] 
There is a standard procedure to approximate it by split exact sequences.  Namely, by forming the pushout along the map $\alpha$, we obtain a morphism of exact sequences of the form 
\[
\xymatrix{
0\ar[r]&M\ar^\alpha[r]\ar^{\alpha}[d] & N\ar^\beta[r]\ar^-{\delta_0}[d] & L\ar@{=}[d]\ar[r] & 0  \\
0\ar[r]&N\ar^-{\delta_1}[r] & N\oplus_M N  \ar^-{\beta'}[r] & L\ar[r] & 0  
}
\]
in which the bottom sequence is \emph{split} exact\footnote{A splitting is given by sending $\ell \in L$ to $(n,-n)\in N\oplus_M N$ for any lift $\beta(n)=\ell$.} and the vertical maps are all injections. The left square in the diagram is the first part of the \v{C}ech nerve of the map $\alpha$, as defined in \Cref{def:Cech_Nerve}. In fact, using the descent result of \Cref{prop:descent_inclusion_groups}, we can reduce the exactness question for $\Gmrp{R}{-}$ to the split case, and obtain the following result. 

%\begin{defn}
%For a morphism $\alpha \colon M\to N$ in $\Mod_\ZZ$, its \tdef{\v{C}ech nerve} is the cosimplicial diagram
%\[
%\mdef{C^\bullet(\alpha)} :=\left(
%\begin{tikzcd}
%	{N} & {N\oplus_M N} & {N\oplus_M N\oplus_M N} & {}
%	\arrow[shift right=0.8, from=1-1, to=1-2]
%	\arrow[shift left=0.8, from=1-1, to=1-2]
%	\arrow[from=1-2, to=1-3]
%	\arrow[shift right=1.6, from=1-2, to=1-3]
%	\arrow[shift left=1.6, from=1-2, to=1-3]
%	\arrow[dotted, no head, from=1-3, to=1-4]
%\end{tikzcd}    
%\right) \qin (\Mod_\ZZ^{\cn})^{\Delta}
%\]
%obtained as an augmented cosimplicial object from $\alpha$ by left Kan extension along the inclusion $\Delta_+^{\le 0} \into \Delta_+$. \scomment{whats the right way to say here?}
%\end{defn}
%Hence, in this case we have 
%$\invlim C^\bullet(\alpha)\simeq M$. \tcomment{We should give an argument or reference. This follows by stability..}
%Moreover, the constructions $\alpha\mapsto C^\bullet(\alpha)$ assemble to an exact functor 
%\[
%C^\bullet(-)\colon \Mod_\ZZ^{\Delta^1} \to \Mod_\ZZ^{\Delta}
%\] 

%The tool that allows us to reduce our problem to the split exact one is the descent along injections of $\Gmr{R}{-}$ from \Cref{prop:descent_inclusion_groups}.
%We now deduce the claim for general exact sequence from the split one by ``descent along injections''. 
\begin{prop}\label{prop:units_red_exact}
Let $R$ be a $p$-complete commutative ring spectrum and let 
\[
\xymatrix{
0\ar[r] & M\ar^\alpha[r] & N\ar^\beta[r] & L\ar[r] & 0
}
\]
be an exact sequence of abelian groups. 
Assume that for every $\ell\in \NN$:
\begin{itemize}
\item $R[N\oplus L^{\oplus \ell}]^\wedge_p$  is $L$-rigid. 
\item The unit map $R\to R[N\oplus L^{\oplus \ell}]^\wedge_p$ induces a bijection 
\[
\pi_0\Spec(R[N\oplus L^{\oplus \ell}]^\wedge_p)\simeq \pi_0\Spec(R). 
\]
\end{itemize}
Then, the null sequence 
\begin{equation}\label{eqn:units_red_exact}\tag{$*$}
\xymatrix{
\AG{R[M]^\wedge_p}\ar[r] & \AG{R[N]^\wedge_p} \ar[r] & \AG{R[L]^\wedge_p}
} 
\end{equation}
is a cofiber sequence in $\Mod_\ZZ^\cn$.
\end{prop}

\begin{proof}
We start by showing the weaker claim that this sequence is a \emph{fiber} sequence in $\Mod_\ZZ^\cn$. 
Consider the \v{C}ech nerve functor $C^\bullet \colon (\Mod_\ZZ)^{\Delta^1} \to  \Mod_\ZZ^{\Delta}$ from \Cref{def:Cech_Nerve}.
Applying the composition  
\[
\Gmrp{R}{C^\bullet(-)}\colon (\Mod_\ZZ^{\cn})^{\Delta^1} \oto{C^\bullet} (\Mod_\ZZ^{\cn})^{\Delta} \oto{\Gmr{R}{-}}(\Mod_\ZZ^{\cn})^{\Delta}
\] 
to the exact sequence of morphisms of abelian groups
\[
\xymatrix{
0\ar[r]& M\ar^\alpha[r]\ar^{\alpha}[d] & N\ar^\beta[r]\ar^{\delta_0}[d] & L\ar@{=}[d]\ar[r] & 0 \\
0\ar[r] & N\ar^{\delta_1}[r] & N\oplus_M N  \ar^{\beta'}[r] & L \ar[r] & 0 
}
\]
from above, 
we obtain a null sequence in $(\Mod_\ZZ^\cn)^{\Delta}$
\begin{equation}\label{eqn:gmrptot}\tag{$**$}
\xymatrix{
\AG{R[C^\bullet(\alpha)]^\wedge_p}\ar[r] & \AG{R[C^\bullet(\delta_0)]^\wedge_p}\ar[r] & \AG{R[C^\bullet(\id_L)]^\wedge_p 
}.
}
\end{equation}

Since the three maps $\alpha,\delta_0$ and $\id_L$ are injections, we deduce from \Cref{prop:descent_inclusion_groups} that the sequence (\ref{eqn:units_red_exact}) from the statement of the proposition is the limit of the sequences (\ref{eqn:gmrptot}) over $\Delta$.
Since the functor
\[
\lim_\Delta\colon (\Mod_\ZZ^\cn)^\Delta \to \Mod_\ZZ^\cn
\] 
is limit preserving (and in particular preserves fiber sequences), it suffices to show that for every $\ell\in \NN$ the sequence 
\begin{equation}\label{eqn:starell}\tag{$*_{\ell}$}
\xymatrix{
\AG{R[C^\ell(\alpha)]^\wedge_p}\ar[r] & \AG{R[C^\ell(\delta_0)]^\wedge_p}\ar[r] & \AG{R[C^\ell(\id_L)]^\wedge_p}} 
\end{equation}
is a fiber sequence in $\Mod_\ZZ^\cn$. 
We do this by verifying the conditions of \Cref{prop:red_units_direct_sum}: note that the exact sequences
\[
\xymatrix{
0\ar[r] & C^\ell(\alpha)\ar[r] & C^\ell(\delta_0) \ar[r]& C^\ell(\id_L)\ar[r] & 0
}
\]
are split, as they are pushouts of the split exact sequence 
\[
\xymatrix{
0\ar[r] & N\ar[r] & N\oplus_M N\ar[r] & L\ar[r] & 0
}
\]
along the injections $N\to C^\ell(\alpha)$.  In particular, this means that for $\ell = 0,1,\dots$, we obtain that $C^\ell(\delta_0)\simeq N \oplus L^{\oplus \ell}$, and we clearly have $C^\ell(\id_L)\simeq L$.  Therefore, our hypotheses exactly check the conditions of \Cref{prop:red_units_direct_sum}.  It follows that
%Our assumptions on $R$ therefore implies that $R[C^\ell(\delta_0)]^\wedge_p$ is $p$-completely $C^\ell(\id_L)$-rigid and that the unit map $R\to R[C^\ell(\delta_0)]^\wedge_p$ induces a bijection on the connected components of their Zariski spectra. 
%We deduce from \Cref{prop:red_units_direct_sum}
 (\ref{eqn:starell}) are fiber sequences, and so (\ref{eqn:units_red_exact}) is as well.

To show that (\ref{eqn:units_red_exact}) is in fact a cofiber sequence, observe that by our assumption that $R[N]^\wedge_p$ is $p$-completely $L$-rigid, the same holds for $R$ and hence we have an isomorphism 
\[
\AG{R[L]^\wedge_p} \simeq \ct{L}(R)
\]
fitting into a commutative square
\[
\xymatrix{
\ct{M}(R)\ar[d]\ar[r]\ar^{\widehat{\Theta}_M}[d] & \ct{L}(R) \ar^{\widehat{\Theta}_L}_\wr[d] \\ 
\pi_0\AG{R[M]^\wedge_p}\ar[r] & \pi_0\AG{R[L]^\wedge_p}
}.
\]
Since the upper map is surjective, we deduce that the same holds for the bottom map, and hence that (\ref{eqn:units_red_exact}) is a cofiber sequence in $\Mod_\ZZ^\cn$.
\end{proof}

%We shall apply this result in the following special case:

%\begin{cor}
%Let $\Lambda$ be a free abelian group of finite rank (i.e. $\Lambda \simeq \ZZ^k$) and let $R$ be a connective commutative ring spectrum. If $R[\Lambda]$ is $p$-rigid then there is a natural cofiber sequence 
%\[
%\xymatrix{
%\Gmr{R}{\Lambda} \ar^{[p]_*}[r] &\Gmr{R}{\Lambda} \ar[r]& \ct{\Lambda/p}(R). 
%}
%\]
%Here, $[p]\colon \Lambda \to \Lambda$ is multiplication by $p$. 
%\end{cor}

%\begin{proof}
%If $R[\Lambda]$ is $p$-rigid then it is in particular $(\Lambda/p)^{\oplus \ell}$-rigid for every $\ell$. Since $R\to R[\Lambda]$ induces a bijection on Zariski-connected components, the result follows from  \Cref{???} applied to the sequence
%\[
%\xymatrix{
%0 \ar[r] & \Lambda \ar^{[p]}[r] & \Lambda  \ar[r] & \Lambda/p \ar[r] & 0
%}
%\]
%\end{proof}

%\subsubsection{A Fiber Sequence for Free Abelian Groups}

We will be particularly interested in  \Cref{prop:units_red_exact} for the exact sequence
%Let $\Lambda$ be a free abelian group of finite rank and let $R$ be a $p$-complete commutative ring spectrum.
%We wish to apply the result of \Cref{prop:units_red_exact} to the exact sequence 
\[
\xymatrix{
0\ar[r] & \Lambda\ar^{[p]}[r] & \Lambda\ar[r] & \Lambda/p \ar[r] & 0
}.
\]

To check the conditions of the proposition, we need the following:
%For this, we have to verify 
% that $R$ satisfies the required rigidity conditions. The rigidity for $\Lambda/p$ follows from the $p$-rigidity results of the previous section, and the second requirement follows from the following general result.
\begin{lem} \label{group_algebra_components}
Let $R$ be a $p$-complete commutative ring spectrum and let $M$ be a finitely generated abelian group which contains no elements of order prime to $p$. Then, the unit map $R\to R[M]^\wedge_p$ induces a bijection 
\[
\pi_0\Spec(R[M]^\wedge_p)\simeq \pi_0\Spec(R)
\]
\end{lem}
\begin{proof}
As both sides depend only on $\pi_0$, we can immediately reduce to $R$ discrete.  Moreover, the rings are derived $p$-complete, so we can in fact reduce modulo $p$ and further assume that $R$ is a ring of characteristic $p$ (in which case we can drop the extra $p$-completion).
%First, both sides depend only on $\pi_0$ of the rings in question, so we may assume without loss of generality that $R$ is discrete. Moreover,  
%since both $R$ and $R[M]^\wedge_p$ are derived $p$-complete we can reduce modulo $p$ and further assume that $R$ is a ring of characteristic $p$ (in which case we can drop the extra $p$-completion).%may drop the unnecessary completions from the group algebras). 

Write $M = M' \oplus \Z^k$ where $M'$ is a finite $p$-group. Then $R[M]$ is a nil-thickening of $R[\Z^k]$, so it suffices to consider the case $M'=0$, and in fact we can assume $k=1$ and $R$ is reduced.   But here, one sees by inspection that any idempotent in $R[t^{\pm 1}]$ is inside $R$, and the conclusion follows. 
%show that $\pi_0\Spec(R[\Lambda]) \simeq \pi_0\Spec(R)$.  For this, it suffices to consider the case when $\Lambda$ has rank $1$ and when $R$ is reduced -- i.e., to see that $\pi_0\Spec(R[t^{\pm 1}]) \simeq \pi_0\Spec(R)$.  But for $R$ reduced, it is immediate to see that any idempotent in $R[t^{\pm 1}]$ is inside $R$, and the conclusion follows. 
\end{proof}

We are ready to prove the main result of this subsection:

\begin{prop}\label{prop:cofiber_sequence_free_group}
Let $R$ be a $p$-complete commutative ring spectrum of finite Tor-amplitude, let $\Lambda$ be a free abelian group of finite rank (i.e. $\Lambda \simeq \ZZ^k$), and let $[p]\colon \Lambda \to \Lambda$ denote the multiplication by $p$ map. Then, the null sequence
\[
\xymatrix{
\AG{R[\Lambda]^\wedge_p} \ar^{[p]_*}[r] & \AG{R[\Lambda]^\wedge_p}\ar[r] & \AG{R[\Lambda/p]} 
}
\]
is a cofiber sequence.
\end{prop}

\begin{proof}
In this case, $R[\Lambda \oplus (\Lambda/p)^{\oplus \ell}]^\wedge_p$ is also of finite Tor-amplitude and hence $p$-rigid by \Cref{cor:p-rigid}. In particular, it is $p$-completely $\Lambda/p$-rigid. %(as $\Lambda/p$ is finite, the $p$-completion does not play a role).  
The result now follows by combining this with \Cref{group_algebra_components} and \Cref{prop:units_red_exact}.
\end{proof}

\subsection{The Frobenius on strict units of group algebras}\label{sub:frob}
We turn to (\ref{enum:free_groups_plan_2}) in our strategy; that is, we relate the map induced by $[p]\colon \Lambda \to \Lambda$ to the multiplication by $p$ map on the spectrum $\AG{R[\Lambda]^\wedge_p}.$
%\[
%p\colon \AG{R[\Lambda]^\wedge_p} \to \AG{R[\Lambda]^\wedge_p}.
%\]
More specifically, for $R$ $p$-complete of finite Tor-amplitude, we show:
 \begin{itemize}
 \item The Frobenius map (cf. \Cref{sub:moore}) of the group algebra $R[\Lambda]^\wedge_p$ (``absolute Frobenius'') is the composite of $[p]_*\colon R[\Lambda]^\wedge_p \to R[\Lambda]^\wedge_p$ (``geometric Frobenius'') and the Frobenius $\varphi_R$ of $R$ (``arithmetic Frobenius'', cf. \S \ref{sub:moore}).
%The map $[p]_*\colon R[\Lambda]^\wedge_p \to R[\Lambda]^\wedge_p$ canonically factors the Frobenius map of the group algebra $R[\Lambda]^\wedge_p$ through the Frobenius of $R$. In other words, the ``absolute Frobenius'' $\varphi_{R[\Lambda]^\wedge_p}$ is the composition of the ``geometric Frobenius'' $[p]_*$ and the ``arithmetic Frobenius'' $\varphi_R$. 
 \item The absolute Frobenius map $\varphi_{R[\Lambda]^\wedge_p}$ acts on strict units via raising to the $p$-th power. %Namely, if $x$ is a strict unit then there is a canonical identification $\varphi_{R[\Lambda]^\wedge_p}(x) = x^p$ of strict units. 
 %\atodo{Is it better to phrase this in terms of multiplication by $p$ on the $\ZZ$-module spectrum of strict units?}
 \end{itemize}

 It will follow that when $R$ is perfect (i.e., $\varphi_R$ is an isomorphism), the map $[p]_*: \AG{R[\Lambda]^\wedge_p} \to \AG{R[\Lambda]^\wedge_p}$ differs from multiplication by $p$ by an isomorphism.  We therefore obtain a variant of \Cref{prop:cofiber_sequence_free_group} with $[p]_*$ replaced by $p:\AG{R[\Lambda]^\wedge_p}\to \AG{R[\Lambda]^\wedge_p}$.  

% We now explain these two facts and their consequences.

% \subsubsection{The Frobenius Map and Perfect Ring Spectra}  For a $p$-complete ring spectrum of finite Tor-amplitude, there is a canonical \tdef{Frobenius morphism} 
% \[
% \mdef{\varphi_R} \colon R\to R,
% \] 
% given by composition of the Tate-valued Frobenius $\varphi^{\mathrm{Tate}}_R: R\to R^{tC_p}$ (\cite[Definition IV.1.1]{NS}) with the inverse of the canonical map $\can_R: R \to R^{hC_p} \to R^{tC_p}$.\footnote{When $R$ is clear from the context, we shall denote $\varphi_R$ simply by $\varphi$.  }\atodo{we should explain why can is an equivalence in this case; it can be part of the discussion of what finite Tor amplitude is.}  The homomorphism $\varphi_R^p$ is a lift of the usual Frobenius of the ordinary $\FF_p$-algebra $\pi_0R/p$.  By analogy to the classical notion of a perfect $\FF_p$-algebra, one can make the following definition:

We begin with a lemma on the compatibility of Frobenius with tensor products.  Let $\widehat{\otimes}$ denote the $p$-completed tensor product in $\Sp^\wedge_p$. If $R$ and $S$ have finite Tor-amplitude, then so does $R\widehat{\otimes} S$, so we have a Frobenius $\varphi_{R\widehat{\otimes} S}$.  

% \subsubsection{The Frobenius on group algebras}
% Let $\widehat{\otimes}$ denote the $p$-completed tensor product in $\Sp^\wedge_p$. If $R$ and $S$ have finite Tor-amplitude, then so does $R\widehat{\otimes} S$, so we have a Frobenius $\varphi_{R\widehat{\otimes} S}$.  
% This Frobenius is compatible with tensor products in the following sense.
\begin{lem} \label{Frob_tensor}
Let $R$ and $S$ be $p$-complete commutative ring spectra of finite Tor-amplitude. Then the two maps $R\widehat{\otimes} S \to R\widehat{\otimes} S$ given by $\varphi_R \widehat{\otimes} \varphi_S $ and $\varphi_{R\widehat{\otimes} S}$ are homotopic.  
% \[
% \varphi_R \widehat{\otimes} \varphi_S \colon  R\widehat{\otimes} S \to R\widehat{\otimes} S
% \]
% and 
% \[
% \varphi_{R \widehat{\otimes} S}  \colon  R\widehat{\otimes} S \to R\widehat{\otimes} S
% \]
% are homotopic.
\end{lem}

\begin{proof}
Immediate from the diagram (with completions left implicit for clarity)
\[
\begin{tikzcd}[row sep=large, column sep=huge]
 &(R \otimes S)^{tC_p} & \\
 R  \otimes S \arrow[ru, "\frob_{R \otimes S}"] \arrow[r,"\frob_R  \otimes \frob_S "'] & R^{tC_p}  \otimes S^{tC_p}\arrow[u] & R \otimes S \arrow[lu,"\triv_{R \otimes S}"'] \arrow[l,"\triv_R  \otimes \triv_S"],
\end{tikzcd}
\]
where the triangles arise from $\frob$ and $\triv$ being lax monoidal natural transformations \cite[Section I.3]{NS}.
\end{proof}

As a result, we obtain a formula for the Frobenius of group algebras. 

\begin{prop} \label{Frobenius_group_algebra}
Let $R$ be a $p$-complete commutative ring spectrum of finite Tor-amplitude and let $M$ be an abelian group. 
Let $[p]\colon R[M]^\wedge_p \to R[M]^\wedge_p$ be the morphism induced by the multiplication by $p$ endomorphism of $M$. 
Then, the Frobenius morphism $\varphi_{R[M]^{\wedge}_p}$ is given by the composite
\[
 R[M]^\wedge_p \oto{\varphi_R[M]^\wedge_p} R[M]^\wedge_p \oto{[p]} R[M]^\wedge_p.
\]
In particular, on $\pi_*R[M]^\wedge_p$ it is given by
\[
\varphi_{{R[M]^\wedge_p}}\left(\sum_{m\in M} r_m [m]\right) = \sum_{m\in M} \varphi_R(r_m)[pm].
\]
\end{prop}

\begin{proof}
Since $R[M]^\wedge_p \simeq R \widehat{\otimes} (\Sph_p[M]^\wedge_p)$, the compatibility of Frobenius with tensor products (\Cref{Frob_tensor}) reduces the general case to the special case $R=\Sph_p$. Since $\varphi_{\Sph_p}$ is the identity (\cite[Example IV.1.2(ii)]{NS}), it suffices to see that $\varphi_{\Sph_p[M]^\wedge_p} \simeq [p]$. This follows, by passing to $p$-completions, from the identification of the Tate diagonal of $\Sph[M]$ as in \cite[Lemma IV.1.3]{NS}.
\end{proof}

%\subsubsection{The Frobenius on strict units}

To identify the effect of these maps on strict units, we use the following result of \cite{CarmeliStrict}.  %Next, we identify the action of Frobenius on strict units following \cite{CarmeliStrict}.  
%The next step is to identify the action of Frobenius on strict units. This is essentially carried in \cite{CarmeliStrict}; we review and slightly extend this treatment here.

\begin{prop}[{\cite[Proposition 4.5]{CarmeliStrict}}] 
\label{Frob_strict_units_p}
Let $R$ be a $p$-complete commutative ring spectrum of finite Tor-amplitude and let $\varphi:=\varphi_R$ be its Frobenius map.  Then the map 
\[
\varphi_*\colon \G_m(R)\to \G_m(R)
\] 
is homotopic to the multiplication by $p$ map
on the $\ZZ$-module spectrum $\G_m(R).$
\end{prop}

% \begin{proof}
% This follows immediately from \cite[Proposition 4.5]{CarmeliStrict} and the fact that the canonical map of $R$ is an isomorphism.
% \end{proof}

Combining this with \Cref{Frobenius_group_algebra}, we obtain the following statement for group algebras:

\begin{cor} \label{Frob_on_reduced_G_m}
Let $R$ be a commutative ring spectrum of finite Tor-amplitude and let $M$ be an abelian group. Denote by $[p]:R[M]^\wedge_p \to R[M]^{\wedge}_p$ the map induced by multiplication by $p$ on $M$. Then there is a commutative triangle
\[
\xymatrix{ 
      & \AG{R[M]^\wedge_p}\ar^{[p]_*}[rd]  & \\ 
\AG{R[M]^\wedge_p}\ar^p[rr]\ar^{(\varphi_R)_*}[ru]&        & \AG{R[M]^\wedge_p}.
}
\]
\end{cor}

% \begin{proof}
%  Since $\AG{R[M]^\wedge_p}$ is a retract of $\G_m(R[M]^\wedge_p)$ which is preserved by all the maps in the diagram, it suffices to show that the triangle 
%  \begin{equation}\label{eqn:Gmtriangle}
%  \xymatrix{ 
%        & \G_m(R[M]^\wedge_p)\ar^{[p]_*}[rd]  & \\ 
%  \G_m(R[M]^\wedge_p)\ar^p[rr]\ar^{(\varphi_R)_*}[ru]&        & \G_m(R[M]^\wedge_p)
%  }
%  \end{equation}
%  commutes. 
%  This follows from the combination of \Cref{Frobenius_group_algebra} and \Cref{Frob_strict_units_p}.
% \end{proof}

We will be particularly interested in rings $R$ satisfying the following additional hypothesis:

\begin{defn}\label{dfn:perfect}
A $p$-complete commutative ring spectrum $R$ is \tdef{perfect} if $\frob: R \to R^{tC_p}$ is an isomorphism.
\end{defn}

\begin{exm}
    If $\kappa$ is a (discrete) perfect $\FF_p$-algebra, then the ring of spherical Witt vectors $\Sph_{W(\kappa)}$ is perfect. More generally, any algebra of the form $(\Sph_{W(\kappa)})^A,$ for a finite complex $A$, is perfect. 
\end{exm}

The notion of a perfect commutative ring spectrum was previously defined in \cite{CarmeliStrict} and \cite{yuan2022integral} in slightly different ways, but the two definitions agree (and agree with this one) for rings of finite Tor-amplitude.

\begin{rem}
If $R$ is perfect and \emph{$p$-completely flat}, then $\pi_0R/p$ is a perfect ring and $R\simeq \Sph_{W(\pi_0R/p)}$. 
\end{rem}

For a perfect commutative ring spectrum $R$, \Cref{Frob_on_reduced_G_m} implies:

\begin{prop}\label{prop:p-rig-fiber}
Let $R$ be a perfect $p$-complete commutative ring spectrum of finite Tor-amplitude and let $\Lambda$ be a free abelian group of finite rank. Then there is a cofiber sequence in $\Mod_\ZZ^\cn$
\[
\xymatrix{
\AG{R[\Lambda]^\wedge_p}\ar^{p}[r] & \AG{R[\Lambda]^\wedge_p} \ar[r] & \AG{R[\Lambda/p]}.
}
\]
\end{prop}
\begin{proof}
Combine \Cref{prop:cofiber_sequence_free_group} with \Cref{Frob_on_reduced_G_m} and the fact that $(\varphi_R)_*$ is an isomorphism.% when $R$ is perfect.  

% By \Cref{Frob_on_reduced_G_m} the multiplication by $p$ on $\AG{R[M]^\wedge_p}$ identifies with the map induced from the Frobenius map $\varphi_{R[M]^\wedge_p}$ on the reduced strict units. By \Cref{Frobenius_group_algebra}, the latter can be written as a composition of $[p]_*$ and the isomorphism $\varphi_R$ (base-changed to the group algebra). Thus, we can identify the cofibers of $[p]_*$ and $p$ via the commutative square with invetible vertical maps:
% \[
% \xymatrix{
% \AG{R[M]^\wedge_p}\ar^{[p]_*}[r] \ar@{=}[d]& \AG{R[M]^\wedge_p}\ar_\wr^{(\varphi_R)_*}[d] \\ 
% \AG{R[M]^\wedge_p}\ar^{p}[r] & \AG{R[M]^\wedge_p}  
% }
% \]
% via $\varphi_R$. Finally, the cofiber of $[p]_*$ identifies with $\AG{R[\Lambda/p]}$ by \Cref{prop:cofiber_sequence_free_group}. 
\end{proof}

This gives us $p$-completed $\Lambda$-rigidity up to $p$-completion:
\begin{cor}\label{cor:rigid_after_p_completing_the_units}
Let $R$ be a perfect $p$-complete commutative ring spectrum of finite Tor-amplitude and let $\Lambda$ be a free abelian group of finite rank. Then the map 
\[
\widehat{\Theta}_{\Lambda}\colon \ct{\Lambda}(R)\to \AG{R[\Lambda]^\wedge_p}
\]
induces an isomorphism after $p$-completion.%:
%\[
%\ct{\Lambda}(R)^\wedge_p\iso \AG{R[\Lambda]^\wedge_p}^\wedge_p.
%\]
\end{cor}

\begin{proof}
By definition of $p$-completion, it suffices to show that the induced map 
\[
\ct{\Lambda}(R)/p \to \AG{R[\Lambda]^\wedge_p}/p
\]
is an isomorphism.  But $\ct{\Lambda}(R)/p  \simeq \ct{\Lambda/p}(R)$ and by \Cref{prop:p-rig-fiber}, we have $\Gmrp{R}{\Lambda}/p \simeq \Gmr{R}{\Lambda/p}$, so the statement follows from \Cref{cor:p-rigid}.  
\end{proof}

\subsection{Staticity and rigidity for free abelian groups}\label{sub:Gmstaticity}

We now turn to part (3) of our plan, completing the proof of rigidity for free abelian groups.  

%which will essentially reduce our problem to an algebraic problem about $\delta_p$-rings.  
\begin{rem}\label{rem:restricttowitt}
    At this point, we restrict to $R= \Sph_{W(\kappa)}$ for clarity of proofs, despite the only critical additional hypothesis being that $R$ is connective.  While this restriction appears dramatic, we expect that any $p$-complete perfect connective ring $R$ is of the form $\Sph_{W(\kappa)}$ (though we do not pursue this further here), so we do not really lose any generality.    
\end{rem}

The main technique is the notion of staticity introduced in \S\ref{sub:staticity}. The following  statement asserts that the space of reduced strict units in $\Sph_{W(\kappa)}[\Lambda]^{\wedge}_p$ is a \emph{subset} of the reduced (strict) units of its $\pi_0$.  This  will reduce the computation of strict units to the algebraic question of finding rank $1$ units in the $\delta_p$-ring $W(\kappa)[\Lambda]^{\wedge}_p$.

\begin{prop}\label{prop:red_units_static}
Let $\kappa$ be a perfect $\F_p$-algebra and $\Lambda$ be a free abelian group of finite rank.  Then $\Sph_{W(\kappa)}$ is $\AG{-[\Lambda]^\wedge_p}$-static.
%Let $R$ be a connective, perfect, $p$-complete commutative ring spectrum of finite Tor amplitude, and let $\Lambda$ be a free abelian group of finite rank. Then, $R$ is $\AG{-[\Lambda]^\wedge_p}$-static. %That is, the map $\AG{R[\Lambda]^\wedge_p} \to \AG{\pi_0R[\Lambda]^\wedge_p}$ is an inclusion of connective $\ZZ$-module spectra. 
\end{prop}

\begin{proof}
For a connective spectrum $X$, let  
\[
\mdef{V_pX}\simeq \invlim(\dots\oto{p}X\oto{p}X\oto{p}X) \qin \Sp^\cn,
\]
and consider the fiber sequence of connective spectra
\[
V_p\AG{R[\Lambda]^\wedge_p}\to  \AG{R[\Lambda]^\wedge_p} \to \AG{R[\Lambda]^\wedge_p}^\wedge_p.
\]

By \Cref{prop:static_limits_groups}(2), the statement reduces to the following two facts:
\begin{enumerate}
    \item $\Sph_{W(\kappa)}$ is $V_p\AG{-[\Lambda]^\wedge_p}$-static.  Since $\Gmred(-)$ is a summand of $\G_m(-)$, it suffices to see that $\Sph_{W(\kappa)}[\Lambda]^{\wedge}_p$ is $V_p\G_m$-static.  But this follows from \Cref{prop:p-invert_static} and the fact that multiplication by $p$ is invertible on $V_p\G_m\simeq \Spec(\Sph[\ZZ[\inv{p}]])$.  
    \item $\Sph_{W(\kappa)}$ is $\AG{-[\Lambda]^\wedge_p}^\wedge_p$-static.  To see this, we consider the commutative square 
\[
\xymatrix{
\ct{\Lambda}(\Sph_{W(\kappa)})^\wedge_p \ar^-{\widehat{\Theta}}[r]\ar[d] & \AG{\Sph_{W(\kappa)}[\Lambda]^\wedge_p}^\wedge_p \ar[d]\\ 
\ct{\Lambda}(W(\kappa))^\wedge_p \ar^-{\widehat{\Theta}}[r] & \AG{W(\kappa)[\Lambda]^\wedge_p}^\wedge_p.
}
\]
We wish to show that the right vertical map is an inclusion. The left vertical map is clearly an isomorphism, and by \Cref{cor:rigid_after_p_completing_the_units}, the upper horizontal map is also an isomorphism.  %a limit of the maps $\ct{\Lambda/p^n}(R) \to \ct{\Lambda/p^n}(\pi_0R)$, each of which is an inclusion by \Cref{prop:const_static}, and hence it is an inclusion. 
Thus, since the bottom horizontal map is an injection of discrete abelian groups (and in particular an inclusion), the right vertical map is an inclusion as well.   
\end{enumerate}
\end{proof}

We now turn to the main result of this section, regarding rigidity for free abelian groups.

\begin{prop}\label{prop:rank1pi0}
    The image of truncation $\Gmrp{\Sph_{W(\kappa)}}{\Lambda} \to \Gmrp{W(\kappa)}{\Lambda}$ is contained in the image of 
    \[
    \widehat{\Theta}_{W(\kappa)}: \ct{\Lambda}(W(\kappa)) \to \Gmrp{W(\kappa)}{\Lambda}.
    \]
\end{prop}
\begin{proof}
    Any (reduced) strict unit in $\Sph_{W(\kappa)}[\Lambda]^{\wedge}_p$ is classified by a map from $\Sph[t^{\pm 1}]$.  By \Cref{prop:Moore_delta_ring}, the induced map on $\pi_0$ is naturally a map of $\delta_p$-rings and in particular, the image of $t \in \Gmdred{\Z_p[t^{\pm 1}]}$ is a reduced rank $1$ unit, i.e., an element of $\Gmdred{W(\kappa)[\Lambda]^{\wedge}_p}$.  The statement then follows from \Cref{thm:deltarig}.  
\end{proof}

\begin{thm}\label{thm:free-rig}
Let $\kappa$ be a perfect $\F_p$-algebra and let $\Lambda$ be a free abelian group of finite rank. Then $\Sph_{W(\kappa)}$ is $p$-completely $\Lambda$-rigid. 
\end{thm}
\begin{proof}
Applying \Cref{prop:static_rigid}, the statement follows from \Cref{prop:red_units_static} and \Cref{prop:rank1pi0}.
\end{proof}

Combining this with \Cref{cor:p-rigid} and \Cref{prop:red_units_direct_sum}, we obtain:

\begin{cor}\label{cor:freeprig}
    Let $\kappa$ be a perfect $\F_p$-algebra and $M$ be a finitely generated abelian group with no prime to $p$ torsion.  Then $\Sph_{W(\kappa)}$ is $M$-rigid.  
\end{cor}

%% file: everything.tex
%maps between group algebras
%1. maps between group algebras over spherical witt vectors = delta_p units
%2. integral statements, just about S

In this final section, we finish proving the main results about strict units and maps between spherical group rings.  We handle the $p$-complete case in \S\ref{sub:p-comp} and the integral case in \S\ref{sub:int}.% of group rings and maps between compute the strict units of commutative ring spectra of the form $\Sph[M]$ for finitely generated abelian groups $M$.  

\subsection{Group algebras over spherical Witt vectors}\label{sub:p-comp}
In the $p$-complete case, we have:%We shall now exploit our rigidity results for strict units and rank $1$ units in $\delta$-rings to compare the two notions for group algebras over spherical Witt vectors, answering to the affirmative \Cref{question} in this case. 

%\begin{lem} \label{lem:spherical_witt_G_m}
%Let $\kappa$ be a perfect ring of characteristic $p$. Then we have isomorphisms
%\[
%\kappa^\times \simeq \Gmdp{W(\kappa)}\simeq \GG_m(\Sph_{W(\kappa)}),
%\GG_m(\Sph_{W(\kappa)})\simeq \Gmdp{W(\kappa)}\simeq \kappa^\times.
%\]
%where the first isomorphism is given by the multiplicative lift $[-]$ and the second by $\pi_0$.  
%\end{lem}

%\begin{proof}
%By \cite[Corollary 4.15]{CarmeliStrict}, we have  
%$\GG_m(\Sph_{W(\kappa)})\iso \kappa^\times$. Thus, it  suffices to prove that the rank $1$ units of $W(\kappa)$ are precisely the units in the image of the multiplicative lift $[-]\colon \kappa^\times \to W(\kappa)^\times$.
%Given $x\in \Gmdp{W(\kappa)}$ with image $\overline{x}$ in $\kappa$, we wish to show that $x = [\overline{x}]$. But, on the one hand, the $p$-adic valuation of the difference $x-[\overline{x}]$ clearly satisfies
%\[
%\nu_p(x - [\overline{x}]) \ge 1.
%\] 
%On the other hand, since $x$ and $[\overline{x}]$ are both units of rank $1$ and $\varphi$ is a $p$-adically continuous automorphism of $W(\kappa)$, we have
%\[
%\nu_p(x - [\overline{x}]) = \nu_p(\varphi(x) - \varphi([\overline{x}])) =  
%\nu_p(x^p - [\overline{x}]^p) > \nu_p(x-[\overline{x}]).
%\] 
%This implies that $x=[\overline{x}]$ as we wanted to show.
%\end{proof}

\begin{thm} \label{thm:strict_units_delta_units}
Let $M$ be a finitely generated abelian group and let $\kappa$ be a perfect ring of characteristic $p$. Then applying $\pi_0$ induces an isomorphism
\[
\GG_m(\Sph_{W(\kappa)}[M]^\wedge_p) \iso \Gmdp{W(\kappa)[M]^\wedge_p}.  
\]
\end{thm}
\begin{proof}
When $M$ is a finite group of order prime to $p$, we have $\Sph_{W(\kappa)}[M]^\wedge_p\simeq \Sph_{W(\kappa[M])}$ and so the result follows from \cite[Corollary 4.15]{CarmeliStrict} and \Cref{exm:wittunit}.

In general, we may split $M$ as $M= M_1 \oplus M_2$ where $M_1$ is finite of order prime to $p$ and $M_2$ has only $p$-primary torsion part. We have compatible splittings 
\begin{align*}
\GG_m(\Sph_{W(\kappa)}[M]^\wedge_p) &\simeq \GG_m(\Sph_{W(\kappa)}[M_1]^\wedge_p) \oplus \Gmrp{\Sph_{W(\kappa)}[M_1]}{M_2} \\
\Gmdp{W(\kappa)[M]^\wedge_p} &\simeq \Gmdp{W(\kappa)[M_1]^\wedge_p} \oplus \G_m^{\delta_p,\mathrm{red}}(W(\kappa)[M_1][M_2]^\wedge_p), 
\end{align*}
where the second summand denotes the rank $1$ units which map to $1$ in $W(\kappa)[M_1]^\times$.  By the previous special case, the $\pi_0$ map identifies the first summands, and so it suffices to show that 
\[
\Gmrp{\Sph_{W(\kappa)}[M_1]}{M_2}  \to \G_m^{\delta_p,\mathrm{red}}(W(\kappa)[M_1][M_2]^\wedge_p)
\]
is an isomorphism.  
Replacing $\kappa$ by $\kappa[M_1]$, we may further assume without loss of generality that $M_1=0$.  But here, the result follows by \Cref{prop:red_delta_p_taut_finite}, \Cref{cor:freeprig}, and the evident commutativity of the diagram 
\[
\xymatrix{
\ct{M}(\Sph_{W(\kappa)})\ar^\wr[d] \ar^-\sim[r] & \Gmrp{\Sph_{W(\kappa)}}{M} \ar[d] \\ 
\ct{M}(W(\kappa)) \ar^-\sim[r]       & \G_m^{\delta_p,\mathrm{red}}(W(\kappa)[M]^\wedge_p).
}
\]

\end{proof}
% We are ready to compute the space of maps between group algebras over the rings of spherical Witt vectors.
\begin{cor}\label{cor:maps_group_alg_p_complete}
Let $p$ be a prime.  Then the functor $\pi_0: \CAlg^{\mathrm{M}} \to \dhring$ of \Cref{prop:Moore_delta_ring} is fully faithful on the subcategory of rings of the form $\Sph_{W(\kappa)}[M]^{\wedge}_p$, where $\kappa$ is a perfect ring of characteristic $p$ and $M$ is a finitely generated abelian group.  
% Let $\kappa$ be a perfect ring of characteristic $p$ and let $\Sph_{W(\kappa)}$ be the spherical Witt vectors over $\kappa$. Then for every pair of finitely generated abelian groups $M,N$, 
% \[
% \Map_{\calg(\Sp)}(\Sph_{W(\kappa)}[M]^\wedge_p,\Sph_{W(\kappa)}[N]^\wedge_p) \iso \Map_{\Acr^{\delta_p}}(W(\kappa)[M]^\wedge_p,W(\kappa)[N]^\wedge_p).
% \]
\end{cor}

\begin{proof}
Consider two such objects corresponding to $(\kappa,M)$ and $(\kappa',M')$.  We have a map of fiber sequences of spaces
\[
\begin{tikzcd}
    % \Map_{\Sph_{W(\kappa)}}(\Sph_{W(\kappa)}[M]^{\wedge}_p,\Sph_{W(\kappa')}[M']^{\wedge}_p)\arrow[r]\arrow[d]&    \calg(\Sph_{W(\kappa)}[M]^{\wedge}_p,\Sph_{W(\kappa')}[M']^{\wedge}_p) \arrow[r]\arrow[d] & \calg(\Sph_{W(\kappa)},\Sph_{W(\kappa')}[M']^{\wedge}_p)\arrow[d]\\
    % \dhring_{W(\kappa)/}(W(\kappa)[M]^{\wedge}_p, W(\kappa')[M']^{\wedge}_p)\arrow[r] & 
    % \dhring(W(\kappa)[M]^{\wedge}_p, W(\kappa')[M']^{\wedge}_p)\arrow[r] & 
    % \dhring(W(\kappa), W(\kappa')[M']^{\wedge}_p).
     \Map_{\Sph_{W(\kappa)}}(\Sph_{W(\kappa)}[M]^{\wedge}_p,\Sph_{W(\kappa')}[M']^{\wedge}_p)\arrow[r]\arrow[d]&\Map^{\widehat{\delta}}_{W(\kappa)/}(W(\kappa)[M]^{\wedge}_p, W(\kappa')[M']^{\wedge}_p)\arrow[d]\\
     \Map_{\calg}(\Sph_{W(\kappa)}[M]^{\wedge}_p,\Sph_{W(\kappa')}[M']^{\wedge}_p) \arrow[r]\arrow[d] &      \Map^{\widehat{\delta}}(W(\kappa)[M]^{\wedge}_p, W(\kappa')[M']^{\wedge}_p)\arrow[d] \\
     \Map_{\calg}(\Sph_{W(\kappa)},\Sph_{W(\kappa')}[M']^{\wedge}_p)\arrow[r]
      &     \Map^{\widehat{\delta}}(W(\kappa), W(\kappa')[M']^{\wedge}_p).
\end{tikzcd}
\]
The bottom horizontal map is an isomorphism because even the composite with the inclusion into $\Map(W(\kappa),W(\kappa')[M']^{\wedge}_p)$ is an isomorphism, by the universal property of spherical Witt vectors \cite[Example 5.2.7]{Lurie_Ell2}.  But the top horizontal map is also an isomorphism by considering $\Z$-module maps from $M$ into the equivalence of \Cref{thm:strict_units_delta_units}.  Therefore, the middle map is an equivalence, and the desired fully faithfulness follows. %for $M=\Z$ and the general case follows by considering $\Z$-module maps from $M$ into the equivalence 

% Fix $N$ and consider the displayed map as a natural transformation of functors 
% \[
% \Mod_\ZZ^\cn \to \Spc^{\op}. 
% \]
% Since the functors preserve colimits, it suffices to check the case $M=\Z$, which is \Cref{thm:strict_units_delta_units}.
%it suffices to check that the natural transformation is an isomorphism at $M=\ZZ$, which follows from \Cref{thm:strict_units_delta_units}.
\end{proof}

\subsection{Group algebras over the sphere}\label{sub:int}
Over the sphere spectrum, we can describe the mapping spaces between group algebras in a particularly simple way. 

\begin{prop}\label{prop:spherestatic}
%Let $R$ be a commutative ring spectrum flat over $\Sph$ \atodo{what's a natural hypothesis on $R$ for this?}
 Let $M$ be a finitely generated abelian group. Then $\Sph[M]$ is $\Gm$-static, i.e. 
 \[
 \Gm(\Sph[M]) \to \Gm(\ZZ[M])
 \]
 is an inclusion of connective $\ZZ$-module spectra.  Consequently, the same holds for $\Gmred(\Sph[M])$.
\end{prop}

\begin{proof}
Since $\G_m$ is representable, the collection of $\G_m$-static connective commutative ring spectra are closed under limits which are preserved by $\pi_0$ by \Cref{prop:static_limits_rings}.  
By the arithmetic fracture square
\[
\begin{tikzcd}
\Sph[M] \arrow[r]\arrow[d] & \prod_p \Sph[M]^{\wedge}_p \arrow[d]\\
\Sph[M]\otimes \Q \arrow[r] & (\prod_p \Sph[M]^{\wedge}_p) \otimes \Q,
\end{tikzcd}
\]
it suffices to show that 
$\Sph[M]^{\wedge}_p$ is $\G_m$-static for all $p$ and that $\Sph[M]\otimes \Q$ and $(\prod_p \Sph[M]^{\wedge}_p) \otimes \Q$ are $\G_m$-static. The latter two rings are discrete and hence static, and the staticity of $\Sph[M]^{\wedge}_p$ is \Cref{thm:strict_units_delta_units}.   The statement about $\Gmred(\Sph[M])$ follows immediately, as $\Gmred$ is a retract of $\Gm$.%because it is a retract of $\Gm(\Sph[M])$.  
\end{proof}

% This immediately implies the following staticity result. 

% \begin{cor}
% Let $M$ be a finitely generated abelian group. Then the sphere spectrum $\Sph$ is $\Gmr{-}{M}$-static. 
% \end{cor}

% \begin{proof}
% The abelian group sheaf $\Gmr{-}{M}$ is a summand of $\Gm(-[M])$ and hence $\Gmr{\Sph}{M}$ also includes into $\Gmr{\ZZ}{M}$. 
% \end{proof}

To prove the rigidity of the sphere in general, it remains to establish the following:% with respect to all finitely generated abelian groups, it remains to demonstrate the following claim.

\begin{prop}\label{prop:pi0im}
Let $M$ be a finitely generated abelian group. The image of the inclusion $\Gmr{\Sph}{M} \to \Gmr{\ZZ}{M}$ is contained in the image of the map $\Theta_M \colon M\to \Gmr{\ZZ}{M}$. 
\end{prop}

\begin{proof}
Since $\pi_0\Sph[t^{\pm 1}]\simeq \ZZ[t^{\pm 1}]$ as $\hat{\delta}$-rings, the image of $\Gmr{\Sph}{M}$ is contained in the reduced rank $1$ units, i.e. in $\Gmdhat{\ZZ[M]}\subseteq \Gm(\ZZ[M])$. Hence, the result follows from  \Cref{thm:delta_gp_alg_ff}.
\end{proof}

%\subsection{Strict Units and Frobenious} 
%In this section, for a strict units $a\in \G_m(\Sph[M])$,  we analyse the constraints on its image in $\G_m(\Z[M]) = \Z[M]^\times$. 

%\begin{prop}
%Let $M$ be a finitely generated abelian group and let $a\in \pi_0\G_m(\Sph[M])$. Then, the image of $a$ is $\Z[M]$ belongs to the image of the unit map $M\to \Z[M]^\times$.
%\end{prop}

%\begin{proof}
%Let $\varphi_p \colon \Sph[M]\to \Sph[M]^\wedge_p$ be the $p$-th Frobenious map [Ref]. [TODO]
%\end{proof}

%\subsection{Strict Units of Group Algebras}

We are finally ready to prove our main theorem.  
\begin{thm} \label{strict_units_group_algebra_S}
Let $M$ be a finitely generated abelian group. Then the natural map
\[
\Theta: M \to \G_m(\Sph[M])
\]
is an isomorphism.  
\end{thm}

\begin{proof}
Since by \cite[Theorem B, Remark 1.3]{CarmeliStrict} we have $\G_m(\Sph)\simeq 0$, it suffices to show that $\Sph$ is $M$-rigid. By \Cref{prop:spherestatic}, $\Sph$ is $M$-static, and by \Cref{prop:pi0im}, the map $\Gmr{\Sph}{M}\to \Gmr{\pi_0\Sph}{M}$ lands in the image of $\Theta$. Hence, the result follows from \Cref{prop:static_rigid}(2).  
\end{proof}

%We can finally compute the mapping spaces between spherical group algebras. 
\begin{cor}
%Let $\Ab^{\fg}$ denote the (ordinary) category of finitely generated abelian groups.  Then t
The spherical group ring functor $\Sph[-]: \Ab^{\fg} \to \CAlg$
%\begin{align*}
%\Sph[-]: \Ab^{\fg} &\to \CAlg 
%\end{align*}
is fully faithful.
\end{cor}
\begin{proof}
We want to see that for finitely generated abelian groups $M$ and $N$, the natural map
\[
\Hom_{\Ab^{\fg}}(N, M) \to \Hom_{\CAlg}(\Sph[N],\Sph[M])
\]
is an isomorphism.  This follows by taking $\Z$-module maps from $N$ to the equivalence of \Cref{strict_units_group_algebra_S}.
% Note that if 
% \[
% 0 \to M' \to M \to M'' \to 0
% \]
% is a short exact sequence of abelian groups, then 
% \[
% \begin{tikzcd}
%     M' \arrow[r]\arrow[d] & M\arrow[d] \\
%     0 \arrow[r] & M'' 
% \end{tikzcd}
% \]
% is a pushout of connective spectra  Since $\Sph[-]: \Sp_{\geq 0} \to \CAlg$ preserves colimits, this induces a corresponding pushout in $\CAlg$.  We may therefore immediately reduce to the case $M = \Z$, where the result follows from \Cref{strict_units_group_algebra_S}. 
\end{proof}